\documentclass[reqno,dvips]{amsart}
\title{Categorical aspects of cointegrals on quasi-Hopf algebras}
\author[T. Shibata]{Taiki Shibata}
\address[T. Shibata]{Department of Applied Mathematics,
  Okayama University of Science \\
  1-1 Ridai-cho, Kita-ku Okayama-shi, Okayama 700-0005, Japan.}
\email{shibata@xmath.ous.ac.jp}
\author[K. Shimizu]{Kenichi Shimizu}
\address[K. Shimizu]{Department of Mathematical Sciences,
  Shibaura Institute of Technology \\
  307 Fukasaku, Minuma-ku, Saitama-shi, Saitama 337-8570, Japan.}
\email{kshimizu@shibaura-it.ac.jp}

\usepackage{amsmath,amssymb,amscd,bbm,stmaryrd}
\usepackage{mathtools}
\usepackage[all,arc,knot]{xy}

\usepackage{amsthm}
\numberwithin{equation}{section}

\newtheorem{counter}{}[section]

\theoremstyle{definition}
\newtheorem{definition}         [counter]{Definition}

\theoremstyle{plain}
\newtheorem{lemma}              [counter]{Lemma}
\newtheorem{claim}              [counter]{Claim}
\newtheorem{proposition}        [counter]{Proposition}
\newtheorem{theorem}            [counter]{Theorem}
\newtheorem{corollary}          [counter]{Corollary}

\newtheorem*{theorem*}           {Theorem}
\theoremstyle{remark}
\newtheorem{remark}             [counter]{Remark}

\newcommand{\id}{\mathrm{id}}
\newcommand{\ad}{\mathsf{ad}}

\newcommand{\cop}{\mathop{\sf cop}}
\newcommand{\Hom}{\mathrm{Hom}}
\newcommand{\End}{\mathrm{End}}
\newcommand{\Img}{\mathrm{Im}}

\newcommand{\eval}{{\rm ev}}
\newcommand{\coev}{{\rm coev}}

\newcommand{\op}{\mathsf{op}}
\newcommand{\rev}{\mathsf{rev}}

\newcommand{\unitobj}{\mathbbm{1}}
\newcommand{\coinv}{\mathsf{co}}
\newcommand{\CF}{\mathsf{CF}}

\usepackage{upgreek}
\newcommand{\qhphi}{\upphi}
\newcommand{\qhalpha}{\upalpha}
\newcommand{\qhbeta}{\upbeta}
\newcommand{\qhmu}{\upmu}
\newcommand{\qhepsilon}{\upepsilon}
\newcommand{\qhomega}{\upomega}
\newcommand{\qhS}{\mathcal{S}}
\newcommand{\qhf}{\mathbbm{f}}
\newcommand{\qhg}{\mathbbm{g}}
\newcommand{\qhp}{\mathbbm{p}}
\newcommand{\qhq}{\mathbbm{q}}
\newcommand{\qhL}{\mathsf{L}}
\newcommand{\qhR}{\mathsf{R}}
\newcommand{\qhU}{\mathbbm{u}}
\newcommand{\qhV}{\mathbbm{v}}
\newcommand{\qhW}{\mathbbm{w}}
\newcommand{\qhUL}{\qhU^{\qhL}}
\newcommand{\qhUR}{\qhU^{\qhR}}
\newcommand{\qhVL}{\qhV^{\qhL}}
\newcommand{\qhVR}{\qhV^{\qhR}}
\newcommand{\qhWR}{\qhW^{\qhR}}
\newcommand{\qhWL}{\qhW^{\qhL}}
\newcommand{\qhpl}{\qhp^{\qhL}}
\newcommand{\qhql}{\qhq^{\qhL}}
\newcommand{\qhpr}{\qhp^{\qhR}}
\newcommand{\qhqr}{\qhq^{\qhR}}

\newcommand{\ptrace}{\mathsf{ptr}}
\newcommand{\Int}{\mathord{\smallint}}
\newcommand{\cointl}{\Int^{\qhL}}
\newcommand{\cointr}{\Int^{\qhR}}
\newcommand{\musymint}{\Int^{\qhmu\text{-{\sf sym}}}}
\newcommand{\catcoint}{\Int^{\mathsf{cat}}}
\newcommand{\catlambda}{\lambda^{\mathsf{cat}}}
\newcommand{\HH}{\mathrm{HH}}
\newcommand{\HHm}{\HH^{\mathsf{mod}}}

\usepackage{mathrsfs}
\newcommand{\Mod}{\mathscr{M}}
\newcommand{\Proj}{\mathscr{P}}
\newcommand{\fd}{\mathsf{fd}}
\newcommand{\Vect}{\mathsf{Vec}}
\newcommand{\hatotimes}{\mathbin{\hat{\otimes}}}

\newcommand{\YD}{\mathscr{Y}\!\!\mathscr{D}}
\newcommand{\YDsigma}{\sigma}
\newcommand{\YDAdelta}{\delta^{\mathrm{1st}}}
\newcommand{\YDA}[1]{\langle{#1}\rangle}
\newcommand{\YDBdelta}{\delta^{\mathrm{2nd}}}
\newcommand{\YDB}[1]{[#1]}

\usepackage{color}
\usepackage[normalem]{ulem}
\newcommand{\qhulcolor}{red}
\newcommand{\qhano}[1]{\color{\qhulcolor}\dotuline{{\color{black}#1}}\color{black}}

\usepackage[all,knot]{xy}
\newcommand{\mycap}[1]{\POS{:
    (0,0); (0,#1) **@{} ?(.5)="_P1",
    (1,0); (1,#1) **@{} ?(.5)="_P2",
    (0,0); (1,0) **\crv{"_P1" & (0.5,#1) & "_P2"}}}
\newcommand{\makecoord}[2]{\POS{#1; p+(0, 1) **\dir{} ?!{#2; p+(1, 0)}}}

\usepackage{comment}
\usepackage{hyperref}
\begin{document}

\begin{abstract}
We discuss relations between some category-theoretical notions for a finite tensor category and cointegrals on a quasi-Hopf algebra. Specifically, for a finite-dimensional quasi-Hopf algebra $H$, we give an explicit description of categorical cointegrals of the category ${}_H \Mod$ of left $H$-modules in terms of cointegrals on $H$. Provided that $H$ is unimodular, we also express the Frobenius structure of the `adjoint algebra' in the Yetter-Drinfeld category ${}^H_H \YD$ by using an integral in $H$ and a cointegral on $H$. Finally, we give a description of the twisted module trace for projective $H$-modules in terms of cointegrals on $H$. 
\end{abstract}

\maketitle

\tableofcontents


%



\section{Introduction}

Results on Hopf algebras have been reexamined from the viewpoint of the theory of tensor categories since a result on Hopf algebras generalized to tensor categories is expected to be useful in applications to, {\it e.g.}, low-dimensional topology and conformal field theories. Integrals and cointegrals for Hopf algebras are introduced by Sweedler \cite{MR0242840} and play an important role in the study of Hopf algebras. Let $H$ be a finite-dimensional Hopf algebra over a field $k$. Recently, we have observed that cointegrals on $H$ appear in several results on finite tensor categories. For example, it is known that, if $\mathcal{C}$ is a unimodular finite tensor category, then there is a Frobenius algebra $\mathbf{A}$ in the Drinfeld center $\mathcal{Z}(\mathcal{C})$ of $\mathcal{C}$ arising from an adjunction between $\mathcal{Z}(\mathcal{C})$ and $\mathcal{C}$ \cite{MR3265394,MR3632104}. As shown in \cite{MR3265394}, the Frobenius structure of $\mathbf{A}$ is written in terms of integrals and cointegrals of $H$ if $\mathcal{C} = {}_H \Mod_{\fd}$ is the category of finite-dimensional left $H$-modules. As another example, we mention that the space of certain kind of traces for a pivotal finite tensor category $\mathcal{C}$ can be identified with the space of `$\qhmu$-symmetrized' cointegrals on $H$ if $\mathcal{C} = {}_H \Mod_{\fd}$ and $\qhmu$ is the modular function on $H$ \cite{2018arXiv180100321B,2018arXiv180901122F}.

The above-mentioned results may be explained by the integral theory for finite tensor categories established by the second-named author \cite{MR3632104,MR3921367}. He introduced the notions of categorical integrals and categorical cointegrals for a finite tensor category $\mathcal{C}$ and showed that they can be identified with ordinary (co)integrals if $\mathcal{C} = {}_H \Mod_{\fd}$. As demonstrated in \cite{MR3921367}, several results in the Hopf algebra theory are extended to finite tensor categories by using categorical (co)integrals. However, a relation between module traces and categorical cointegrals is still open. We also remark that an explicit description of a categorical cointegral is not known except the case where $\mathcal{C} = {}_H \Mod_{\fd}$ or if $\mathcal{C}$ is a fusion category.

In this paper, we discuss these problems in the case where $\mathcal{C}$ is the category ${}_H \Mod$ of left modules over a {\em finite-dimensional quasi-Hopf algebra} $H$ over a field $k$. In this case, the Drinfeld center $\mathcal{Z}(\mathcal{C})$ is identified with the category ${}^H_H \YD$ of Yetter-Drinfeld $H$-modules \cite{MR1631648,MR2194347}. By using the fundamental theorem for quasi-Hopf bimodules \cite{1999math4164H}, we give a convenient expression of a right adjoint $R: {}_H \Mod \to {}^H_H \YD$ of the forgetful functor from ${}^H_H \YD$ to ${}_H \Mod$ (Theorem~\ref{thm:Rad-ind-2nd}). As an application, we express the space of categorical cointegrals in terms of ordinary cointegrals on $H$ (Theorem~\ref{thm:q-Hopf-cat-coint}). Under the assumption that $H$ is unimodular, we give an explicit description of the Frobenius structure of the algebra $\mathbf{A} = R(k)$ in terms of (co)integrals of $H$ (Theorem~\ref{thm:Ishii-Masuoka-q-Hopf}). Finally, we show that $\qhmu$-twisted module traces on ${}_H \Mod_{\fd}$ (see Subsection \ref{subsec:modified-trace} for the definition) are expressed by `$\qhmu$-symmetrized' cointegrals on $H$, where $\qhmu$ is the modular function on $H$ (Theorem~\ref{thm:q-Hopf-modified-trace}).

Recently, several interesting examples of finite-dimensional quasi-Hopf algebras are introduced and studied from the viewpoint of logarithmic conformal field theories \cite{2017arXiv170608164F,2017arXiv171207260C,2018arXiv181202277N,2018arXiv180902116G}. We hope our abstract results are useful in future study of the representation theory of these algebras and their applications to conformal field theories.

This paper is organized as follows: In Section \ref{sec:preliminaries}, we recall basic notions in the theory of monoidal categories and their module categories. In Section \ref{sec:quasi-hopf}, we recall the definition of quasi-Hopf algebras from \cite{MR1047964,MR1321145} and elementary results on the representation theory of quasi-Hopf algebras. We also collect useful identities in a quasi-Hopf algebra from \cite{MR1669685,MR1696105,1999math4164H,MR3929714}.

In Section \ref{sec:integrals}, we review the integral theory for quasi-Hopf algebras. Let $H$ be a quasi-Hopf algebra and, for simplicity, assume that $H$ is finite-dimensional in this Introduction. The category ${}_H \Mod_H$ of $H$-bimodules is a monoidal category as it can be identified with the category of left modules over the quasi-Hopf algebra $H \otimes H^{\op}$. The axioms of a quasi-Hopf algebra ensure that $H$ is a coalgebra in ${}_H \Mod_H$. The categories ${}_H^H \Mod_H^{}$ and ${}_H^{} \Mod_H^H$ of left and right quasi-Hopf bimodules are defined as the categories of left $H$-comodules and right $H$-comodules in ${}_H \Mod_H$, respectively. The fundamental theorem for quasi-Hopf bimodules \cite{1999math4164H} gives equivalences ${}_H^{} \Mod_H^{H} \approx {}_H \Mod \approx {}_H^H \Mod_H^{}$. Cointegrals on $H$ are defined by using these equivalences. We give some characterizations of cointegrals for later use.

In Section~\ref{sec:yetter-drinfeld}, we review basic results on Yetter-Drinfeld $H$-modules  established in \cite{MR2106925,MR2194347,MR1897403} and give new results on an adjunction between ${}_H \Mod$ and the category ${}^H_H \YD$ of Yetter-Drinfeld $H$-modules. Schauenburg \cite{MR1897403} proved that ${}^H_H \YD$ is equivalent to the category ${}^H_H \Mod{}^H_H$ of $H$-bicomodules in ${}_H \Mod_H$. The equivalence is actually a monoidal equivalence \cite{MR3929714} and there is a commutative diagram
\begin{equation*}
  \xymatrix@C=128pt{
    {}^H_H \YD
    \ar[r]^{\text{Schauenburg \cite{MR1897403}}}_{\approx}
    \ar[d]
    & {}^H_H \Mod^H_H
    \ar[d] \\
    {}_H \Mod
    \ar[r]^{\text{Hausser-Nill \cite{1999math4164H}}}_{\approx}
    & {}_H^{} \Mod^H_H
  }
\end{equation*}
of monoidal  functors, where the vertical arrows are the forgetful functors. This commutative diagram gives a nice description of a right adjoint of the forgetful functor $F: {}^H_H \YD \to {}_H \Mod$ (Theorem~\ref{thm:Rad-ind-2nd}).

Now let $R$ be the right adjoint of $F$. Since $F$ is strong monoidal, $R$ has a natural structure of a monoidal functor. We also express the monoidal structure of $R$ (Lemma \ref{lem:R-monoidal-struc}) and the structure of the algebra $\mathbf{A} = R(k) \in {}^H_H \YD$ explicitly (Theorem~\ref{thm:BCP-alg}).
It turns out that the algebra $\mathbf{A}$ is identical to the algebra $H_0$ investigated in \cite{MR2106925,MR2194347}.

As the algebra $\mathbf{A}$ is important in the theory of tensor categories, we also include the following three remarks on $\mathbf{A}$ which will not be mentioned in later sections:
Firstly, we give a different proof of \cite[Proposition 4.2]{MR2106925} that states that $\mathbf{A}$ is commutative (Theorem~\ref{thm:BCP-alg-1}).
Secondly, we give a description of the algebra $\CF(H) := \Hom_{H}(\mathbf{A}, \unitobj)$ of `class functions' on $H$ introduced in \cite{MR3631720} (Theorem \ref{thm:BCP-alg-3}).
Thirdly, we show that the category of $\mathbf{A}$-modules in ${}^H_H \YD$ is equivalent to ${}_H \Mod$ (Theorem~\ref{thm:BCP-alg-2}). This thorem is in fact a consequence of co-Hopfness of the monoidal adjunction $F \dashv R$ (Theorem~\ref{thm:YD-co-Hopf-adj}) and a general result in the theory of Hopf monads \cite{MR2793022}.

In Section \ref{sec:categori-coint}, we study categorical cointegrals of ${}_H \Mod$ (introduced in \cite{MR3921367}) and the Frobenius structure of the algebra $\mathbf{A}$. In our notation, the space of categorical cointegrals of ${}_H \Mod$ can be identified with ${}^H_H \YD(R(\qhmu), k)$, where $\qhmu$ is the modular function on $H$ regarded as a one-dimensional left $H$-module. The above commutative diagram shows that the functor $L: {}_H \Mod \to {}_H^H \YD$ given by tensoring with $H^{\vee} \in {}_H^{H} \Mod_H^{}$ is left adjoint to $F$ (Theorem~\ref{thm:left-adj}). The fundamental theorem for quasi-Hopf bimodules gives a relation between $H$ and $H^{\vee}$, and this relation gives rise to a relation between $L$ and $R$ through the equivalence of Schauenburg. Specifically, there is an isomorphism
\begin{equation*}
  R(\qhmu \otimes V) \cong L(V)
  \quad (V \in {}_H \Mod)
\end{equation*}
of functors written by using a non-zero cointegral on $H$ (Theorem~\ref{thm:left-adj}). By using this isomorphism, we give a description of a categorical cointegral of ${}_H \Mod$ in terms of a cointegral on $H$ (Theorem~\ref{thm:q-Hopf-cat-coint}). The above isomorphism also implies that $L \cong R$ when $H$ is unimodular, {\it i.e.}, $\qhmu$ is identical to the counit of $H$. It is known that if $H$ is unimodular, then the algebra $\mathbf{A}$ is a Frobenius algebra in ${}^H_H \YD$ \cite{MR3632104}. We write the Frobenius structure of $\mathbf{A}$ explicitly in terms of an integral in $H$ and a cointegral on $H$.

In Section~\ref{sec:modified-trace}, we suppose that $H$ has a pivotal element so that ${}_H \Mod_{\fd}$ is a pivotal monoidal category. Let ${}_H \Proj_{\fd}$ be the full subcategory of ${}_H \Mod_{\fd}$ consisting of projective objects. A $\qhmu$-twisted module traces on ${}_H \Proj_{\fd}$ ({\it cf}. \cite{MR2803849,MR2998839,2018arXiv180900499G,2018arXiv180100321B,2018arXiv180901122F}) is a family of linear maps $\mathsf{t}_{\bullet} = \{ \mathsf{t}_P: \Hom_H(P, \qhmu \otimes P) \to k \}_{P \in {}_H \Proj_{\fd}}$ satisfying the $\qhmu$-twisted cyclicity and being compatible with the partial pivotal trace in ${}_H \Mod_{\fd}$; see Subsection \ref{subsec:modified-trace} for the precise definition. According to \cite{2018arXiv180901122F}, such a trace is constructed from a `$\qhmu$-symmetrized' cointegral on $H$ if $H$ is an ordinary Hopf algebra. The main result of this section is a generalization of this result to the case where $H$ is a quasi-Hopf algebra (Theorem~\ref{thm:q-Hopf-modified-trace}). The proof goes along almost the same way as \cite{2018arXiv180901122F} but also uses a certain technical result on cointegrals on a quasi-Hopf algebra discussed in Section~\ref{sec:integrals}.

\subsection*{Acknowledgment}
We are grateful to Peter Schauenburg for helpful discussion.
We are also grateful to the referee for careful reading of the manuscript and several valuable suggestions.
The first author (T.S.) is supported by JSPS KAKENHI Grant Number JP19K14517.
The second author (K.S.) is supported by JSPS KAKENHI Grant Numbers JP16K17568 and JP20K03520.

\section{Preliminaries}
\label{sec:preliminaries}

\subsection{Monoidal categories}
\label{subsec:monoidal-cats}

We recall basic results on monoidal categories from \cite{MR1712872,MR1321145,MR3242743}. 
A {\em monoidal category} is a datum $\mathcal{C} = (\mathcal{C}, \otimes, \unitobj, \Phi, l, r)$ consisting of a category $\mathcal{C}$, a functor $\otimes: \mathcal{C} \times \mathcal{C} \to \mathcal{C}$, an object $\unitobj \in \mathcal{C}$ and natural isomorphisms $\Phi_{X,Y,Z}: (X \otimes Y) \otimes Z \to X \otimes (Y \otimes Z)$, $l_{X}: \unitobj \otimes X \to X$ and $r_X: X \otimes \unitobj \to X$ ($X, Y, Z \in \mathcal{C}$) satisfying the pentagon and the triangle axioms. The natural isomorphisms $\Phi$, $l$ and $r$ are called the associator, the left unit isomorphism and the right unit isomorphism, respectively. We always assume that the left and the right unit isomorphisms are identities. Although the associator is often assumed to be the identity in the study of monoidal categories, we do not so in this paper as our main concern is the category of modules over a quasi-Hopf algebra.

Given a category $\mathcal{A}$, we denote by $\mathcal{A}^{\op}$ the opposite category of $\mathcal{A}$. If $\mathcal{C}$ is a monoidal category, then $\mathcal{C}^{\op}$ is also a monoidal category. We use the symbol $\mathcal{C}^{\rev}$ to mean the monoidal category obtained from $\mathcal{C}$ by reversing the order of the tensor product of $\mathcal{C}$.

Let $\mathcal{C}$ and $\mathcal{D}$ be monoidal categories. A {\em monoidal functor} from $\mathcal{C}$ to $\mathcal{D}$ is a triple $(F, F^{(0)}, F^{(2)})$ consisting of a functor $F: \mathcal{C} \to \mathcal{D}$, a morphism $F^{(0)}: \unitobj \to F(\unitobj)$ in $\mathcal{D}$ and  a natural transformation
\begin{equation*}
  F_{X,Y}^{(2)}: F(X) \otimes F(Y) \to F(X \otimes Y)
  \quad (X, Y \in \mathcal{C})
\end{equation*}
such that the equations
\begin{gather*}
  \begin{aligned}
    F(\Phi_{X,Y,Z}) \circ F^{(2)}_{X \otimes Y, Z} \, \circ
    & \, (F^{(2)}_{X, Y} \otimes \id_{F(Z)}) \\
    = F^{(2)}_{X, Y \otimes Z} \, \circ
    & \, (\id_{F(X)} \otimes F_{Y,Z}^{(2)}) \circ \Phi_{F(X), F(Y), F(Z)}, \\
  \end{aligned} \\
  F^{(2)}_{\unitobj, X} \circ (F^{(0)} \otimes \id_{F(X)}) = \id_{F(X)} = F^{(2)}_{X, \unitobj} \circ (\id_{F(X)} \otimes F^{(0)})
\end{gather*}
hold for all objects $X, Y, Z \in \mathcal{C}$. A monoidal functor $F: \mathcal{C} \to \mathcal{D}$ is said to be {\em strong} if $F^{(0)}$ and $F^{(2)}$ are invertible. It is said to be {\em strict} if $F^{(0)}$ and $F^{(2)}$ are identities.

Let $F, G: \mathcal{C} \to \mathcal{D}$ be monoidal functors. A {\em monoidal natural transformation} from $F$ to $G$ is a natural transformation $\xi: F \to G$ such that the equations $\xi_{\unitobj} \circ F^{(0)} = G^{(0)}$ and $\xi_{X \otimes Y} \circ F^{(2)}_{X,Y} = G^{(2)}_{X,Y} \circ (\xi_X \otimes \xi_Y)$ hold for all objects $X, Y \in \mathcal{C}$. Suppose that a strong monoidal functor $F: \mathcal{C} \to \mathcal{D}$ admits a right adjoint $R: \mathcal{D} \to \mathcal{C}$. Let $\eta: \id_{\mathcal{C}} \to R F$ and $\varepsilon: F R \to \id_{\mathcal{D}}$ be the unit and the counit of the adjunction, respectively. The functor $R$ has a monoidal structure given by
\begin{gather*}
  R^{(0)} = R( (F^{(0)})^{-1} ) \circ \eta_{\unitobj}, \\
  R^{(2)}_{X,Y} = R(\varepsilon_X \otimes \varepsilon_Y)
  \circ R( (F^{(2)}_{R(X),R(Y)})^{-1} ) \circ \eta_{R(X) \otimes R(Y)}
\end{gather*}
for $X, Y \in \mathcal{C}$. The adjunction $(F, R, \eta, \varepsilon)$ is in fact a {\em monoidal adjunction} in the sense that $F$ and $R$ are monoidal functors and $\eta$ and $\varepsilon$ are monoidal natural transformations.

\subsection{Duality in a monoidal category}
\label{subsec:duality}

Let $\mathcal{C}$ be a monoidal category, and let $X$ be an object of $\mathcal{C}$. A {\em left dual object} of $X$ is a triple $(Y, e, c)$ consisting of an object $Y \in \mathcal{C}$ and morphisms $e: Y \otimes X \to \unitobj$ and $c: \unitobj \to X \otimes Y$ in $\mathcal{C}$ such that the equations
\begin{gather*}
  (\id_X \otimes e) \Phi_{X,Y,X} (c \otimes \id_X) = \id_X
  \quad \text{and} \quad
  (e \otimes \id_{Y}) \Phi_{Y,X,Y}^{-1} (\id_{Y} \otimes c) = \id_Y
\end{gather*}
hold. The morphisms $e$ and $c$ are referred to as the {\em evaluation} and the {\em coevaluation}, respectively. A left dual object of $X$ is, if it exists, unique up to unique isomorphism in the following sense: If $(Y, e, c)$ and $(Y', e', c')$ are both left dual objects of $X$, then there exists a unique isomorphism $f: Y \to Y'$ in $\mathcal{C}$ such that $e' \circ (f \otimes \id_X) = e$ and $c' = (\id_X \otimes f) \circ c$.

Now we suppose that $X$ has a left dual object $(X^{\vee}, \eval_X, \coev_X)$. Given an object $A \in \mathcal{C}$, we denote by $\mathfrak{L}_{A}: \mathcal{C} \to \mathcal{C}$ the functor defined by $\mathfrak{L}_A(V) = A \otimes V$. The definition of a left dual object implies that the functor $\mathfrak{L}_{X^{\vee}}$ is left adjoint to $\mathfrak{L}_X$. More precisely, there is a natural isomorphism
\begin{gather}
  \label{eq:reciprocity-1}
  \mathbbm{Q} = \mathbbm{Q}_{V,W,X}:
  \Hom_{\mathcal{C}}(V, X \otimes W)
  \to \Hom_{\mathcal{C}}(X^{\vee} \otimes V, W), \\
  \notag \mathbbm{Q}(f)
  = (\eval_{X} \otimes \id_{W}) \circ \Phi_{X^{\vee}, X, W}^{-1} \circ (\id_{X^{\vee}} \otimes f)
\end{gather}
for $V, W \in \mathcal{C}$. The inverse of $\mathbbm{Q}$ is given by
\begin{equation*}
  \mathbbm{Q}^{-1}(g)
  = (\id_X \otimes g) \circ \Phi_{X,X^{\vee},V} \circ (\coev_X \otimes \id_{V}).
\end{equation*}
If we use the graphical calculus (see, {\it e.g.}, \cite{MR1321145}), then the natural isomorphism $\mathbbm{Q}$ and its inverse are expressed by string diagrams as follows:
\begin{equation*}
  \mathbbm{Q} \left(
    {\xy /r1pc/: (0,1.5)="T1", (0,-1.5)="T0",
      (0,0) *+!{\makebox[2.5em]{$f$}} *\frm{-}="BX1",
      "BX1"!U; \makecoord{p}{"T1"} **\dir{-} ?>="L1",
      "BX1"!D!L(.6); \makecoord{p}{"T0"} **\dir{-} ?>="L2",
      "BX1"!D!R(.6); \makecoord{p}{"T0"} **\dir{-} ?>="L3",
      "L1" *+!D{V},
      "L2" *+!U{X},
      "L3" *+!U{W}, \endxy} \right)
  = {\xy /r1pc/: (0,1.5)="T1", (0,-1.5)="T0",
    (0,0) *+!{\makebox[2.5em]{$f$}} *\frm{-}="BX1",
    "BX1"!U; \makecoord{p}{"T1"} **\dir{-} ?>="L1",
    "BX1"!D!R(.6); \makecoord{p}{"T0"} **\dir{-} ?>="L2",
    "BX1"!D!L(.6); p-(1.5,0) \mycap{.75}
    ?>; \makecoord{p}{"T1"} **\dir{-} ?>="L3",
    "L1" *+!D{V},
    "L2" *+!U{W},
    "L3" *+!D{X\smash{^{\vee}}}, \endxy},
  \qquad
  \mathbbm{Q}^{-1} \left(
    {\xy /r1pc/: (0,1.5)="T1", (0,-1.5)="T0",
      (0,0) *+!{\makebox[2.5em]{$g$}} *\frm{-}="BX1",
      "BX1"!U!L(.6); \makecoord{p}{"T1"} **\dir{-} ?>="L1",
      "BX1"!U!R(.6); \makecoord{p}{"T1"} **\dir{-} ?>="L2",
      "BX1"!U; \makecoord{p}{"T0"} **\dir{-} ?>="L3",
      "L1" *+!D{X\smash{^{\vee}}},
      "L2" *+!D{V},
      "L3" *+!U{W}, \endxy} \right)
  = {\xy /r1pc/: (0,1.5)="T1", (0,-1.5)="T0",
    (0,0) *+!{\makebox[2.5em]{$g$}} *\frm{-}="BX1",
    "BX1"!U!R(.6); \makecoord{p}{"T1"} **\dir{-} ?>="L1",
    "BX1"!U!L(.6); p-(1.5,0) \mycap{-.75}
    ?>; \makecoord{p}{"T0"} **\dir{-} ?>="L2",
    "BX1"!U; \makecoord{p}{"T0"} **\dir{-} ?>="L3",
    "L1" *+!D{V},
    "L2" *+!U{X},
    "L3" *+!U{W}, \endxy} \,.
\end{equation*}
When we express a morphism by such a string diagram, we adopt the convention that a morphism goes from the top to the bottom of the diagram. The evaluation and the coevaluation are represented by a cup ($\cup$) and a cap ($\cap$), respectively. Although the graphical calculus is a very useful tool in the theory of monoidal categories, it hides the associator that should not be ignored in the study of quasi-Hopf algebras. For this reason, we use string diagrams only to give readers graphical intuition.

There is also a natural isomorphism
\begin{gather}
  \label{eq:reciprocity-2}
  \mathbbm{P} = \mathbbm{P}_{V,W,X}: \Hom_{\mathcal{C}}(V \otimes X, W)
  \to \Hom_{\mathcal{C}}(V, W \otimes X^{\vee}), \\
  \notag \mathbbm{P}(f)
  = (f \otimes \id_{X^{\vee}}) \circ \Phi_{V,X,X^{\vee}}^{-1} \circ (\id_{V} \otimes \coev_X)
\end{gather}
for $V, W \in \mathcal{C}$. The inverse of $\mathbbm{P}$ is given by
\begin{equation*}
  \mathbbm{P}^{-1}(g)
  = (\id_{W} \otimes \eval_X) \circ \Phi_{W,X^{\vee},X} \circ (g \otimes \id_{X}).
\end{equation*}
The isomorphism $\mathbbm{P}$ and its inverse are expressed as follows:
\begin{equation*}
  \mathbbm{P} \left(
    {\xy /r1pc/: (0,1.5)="T1", (0,-1.5)="T0",
      (0,0) *+!{\makebox[2.5em]{$f$}} *\frm{-}="BX1",
      "BX1"!U!L(.6); \makecoord{p}{"T1"} **\dir{-} ?>="L1",
      "BX1"!U!R(.6); \makecoord{p}{"T1"} **\dir{-} ?>="L2",
      "BX1"!U; \makecoord{p}{"T0"} **\dir{-} ?>="L3",
      "L1" *+!D{V},
      "L2" *+!D{X},
      "L3" *+!U{W}, \endxy} \right)
  = {\xy /r1pc/: (0,1.5)="T1", (0,-1.5)="T0",
    (0,0) *+!{\makebox[2.5em]{$f$}} *\frm{-}="BX1",
    "BX1"!U!L(.6); \makecoord{p}{"T1"} **\dir{-} ?>="L1",
    "BX1"!U!R(.6); p+(1.5,0) \mycap{.75}
    ?>; \makecoord{p}{"T0"} **\dir{-} ?>="L2",
    "BX1"!U; \makecoord{p}{"T0"} **\dir{-} ?>="L3",
    "L1" *+!D{V},
    "L2" *+!U{X\smash{^{\vee}}},
    "L3" *+!U{W}, \endxy},
  \qquad
  \mathbbm{P}^{-1} \left(
    {\xy /r1pc/: (0,1.5)="T1", (0,-1.5)="T0",
      (0,0) *+!{\makebox[2.5em]{$g$}} *\frm{-}="BX1",
      "BX1"!U; \makecoord{p}{"T1"} **\dir{-} ?>="L1",
      "BX1"!D!L(.6); \makecoord{p}{"T0"} **\dir{-} ?>="L2",
      "BX1"!D!R(.6); \makecoord{p}{"T0"} **\dir{-} ?>="L3",
      "L1" *+!D{V},
      "L2" *+!U{W},
      "L3" *+!U{X\smash{^{\vee}}}, \endxy} \right)
  = {\xy /r1pc/: (0,1.5)="T1", (0,-1.5)="T0",
    (0,0) *+!{\makebox[2.5em]{$g$}} *\frm{-}="BX1",
    "BX1"!U; \makecoord{p}{"T1"} **\dir{-} ?>="L1",
    "BX1"!D!L(.6); \makecoord{p}{"T0"} **\dir{-} ?>="L2",
    "BX1"!D!R(.6); p+(1.5,0) \mycap{-.75}
    ?>; \makecoord{p}{"T1"} **\dir{-} ?>="L3",
    "L1" *+!D{V},
    "L2" *+!U{W},
    "L3" *+!D{X}, \endxy}.
\end{equation*}  

A {\em right dual object} of $X$ is a triple $(Y, e, c)$ consisting of an object $Y \in \mathcal{C}$ and morphisms $e: X \otimes Y \to \unitobj$ and $c: \unitobj \to Y \otimes X$ such that, in a word, the triple $(X, e, c)$ is a left dual object of $Y$. A right dual object of $X$ is nothing but a left dual object of $X \in \mathcal{C}^{\rev}$. Thus, if ${}^{\vee} \! X$ is a right dual object of $X$, then there are natural isomorphisms
\begin{align*}
  \Hom_{\mathcal{C}}(X \otimes V, W) & \cong \Hom_{\mathcal{C}}(V, {}^{\vee} \! X \otimes W), \\ 
  \Hom_{\mathcal{C}}(V, W \otimes X) & \cong \Hom_{\mathcal{C}}(V \otimes {}^{\vee} \! X, W)
\end{align*}
for $V, W \in \mathcal{C}$.

\subsection{Modules over an algebra}
\label{subsec:modules-over-alg}

Let $\mathcal{C}$ be a monoidal category. A {\em left $\mathcal{C}$-module category} is a category $\mathcal{M}$ endowed with a functor $\ogreaterthan: \mathcal{C} \times \mathcal{M} \to \mathcal{M}$ (called the action) and natural isomorphisms $\Omega_{X,Y,M}: (X \otimes Y) \ogreaterthan M \to X \ogreaterthan (Y \ogreaterthan M)$ and $l_M: \unitobj \ogreaterthan M \to M$ ($X, Y \in \mathcal{C}$, $M \in \mathcal{M}$) obeying certain axioms similar to those for monoidal categories. A {\em right $\mathcal{C}$-module category} and a {\em $\mathcal{C}$-bimodule category} are defined analogously. We omit the definitions of $\mathcal{C}$-module functors and their morphisms; see \cite{MR3242743}.

An {\em algebra} in $\mathcal{C}$ is a triple $(A, m, u)$ consisting of an object $A \in \mathcal{C}$ and morphisms $m: A \otimes A \to A$ and $u: \unitobj \to A$ in $\mathcal{C}$ such that the following equations hold:
\begin{gather*}
  m \circ (m \otimes \id_A) = m \circ (\id_A \otimes m) \circ \Phi_{A,A,A}, \\
  m \circ (u \otimes \id_A) = \id_A = m \circ (\id_A \otimes u).
\end{gather*}
Dually, a {\em coalgebra} in $\mathcal{C}$ is a triple $(C, \Delta, \varepsilon)$ consisting of an object $C \in \mathcal{C}$ and morphisms $\Delta: C \to C \otimes C$ and $\varepsilon: C \to \unitobj$ in $\mathcal{C}$ such that, in a word, $(C, \Delta, \varepsilon)$ is an algebra in $\mathcal{C}^{\op}$. Now let $\mathcal{M}$ be a left $\mathcal{C}$-module category. Given an object $A \in \mathcal{C}$, we define the functor $\mathfrak{L}_A: \mathcal{M} \to \mathcal{M}$ by $\mathfrak{L}_A(M) = A \ogreaterthan M$ for $M \in \mathcal{M}$. If $A$ is an algebra in $\mathcal{C}$, then the functor $\mathfrak{L}_A$ is a monad \cite[VI.1]{MR1712872} on $\mathcal{M}$ in a natural way. Similarly, if $C$ is a coalgebra in $\mathcal{C}$, then the functor $\mathfrak{L}_C$ has a natural structure of a comonad on $\mathcal{M}$.

\begin{definition}
  \label{def:modules-in-module-cat}
  Given an algebra $A$ in $\mathcal{C}$, we define the category ${}_A \mathcal{M}$ of left $A$-modules in $\mathcal{M}$ to be the category of $\mathfrak{L}_A$-modules ($=$ the Eilenberg-Moore category of $\mathfrak{L}_A$ \cite{MR1712872}). Given a coalgebra $C$ in $\mathcal{C}$, we define the category ${}^{C} \! \mathcal{M}$ of left $C$-comodules in $\mathcal{M}$ to be the category of $\mathfrak{L}_C$-comodules. The category of right (co)modules in a right $\mathcal{C}$-module category is defined and denoted in an analogous way.
\end{definition}

We note that $\mathcal{C}$ is a $\mathcal{C}$-bimodule category by the tensor product. Thus, given an algebra $A$ in $\mathcal{C}$, the notions of a left $A$-module in $\mathcal{C}$ and a right $\mathcal{C}$-module in $\mathcal{C}$ are defined. Let $X$ be a left $A$-module in $\mathcal{C}$ with action $\rho: A \otimes X \to X$. If $X$ has a left dual object $(X^{\vee}, \eval_X, \coev_X)$, then we define $\rho^{\sharp}: X^{\vee} \otimes A \to X^{\vee}$ to be the morphism corresponding to $\rho$ via
\begin{equation*}
  \Hom_{\mathcal{C}}(A \otimes X, X)
  \xrightarrow[\text{\eqref{eq:reciprocity-2}}]{\quad \cong \quad}
  \Hom_{\mathcal{C}}(A, X \otimes X^{\vee})
  \xrightarrow[\text{\eqref{eq:reciprocity-1}}]{\quad \cong \quad}
  \Hom_{\mathcal{C}}(X^{\vee} \otimes A, X^{\vee}).
\end{equation*}
It is known that $X^{\vee}$ is a right $A$-module in $\mathcal{C}$ by the action $\rho^{\sharp}$. Graphically, the morphism $\rho^{\sharp}$ can be expressed as follows:
\begin{equation*}
  {\xy /r1pc/: (0,1.5)="T1", (0,-1.5)="T0",
    (0,-.25) *+!{\makebox[3em]{$\rho^{\sharp}\mathstrut$}} *\frm{-}="BX1",
    "BX1"!U!L(.6); \makecoord{p}{"T1"} **\dir{-} ?>="L1",
    "BX1"!U!R(.6); \makecoord{p}{"T1"} **\dir{-} ?>="L2",
    "BX1"!D; \makecoord{p}{"T0"} **\dir{-} ?>="L3",
    "L1" *+!D{X\smash{^{\vee}}},
    "L2" *+!D{A},
    "L3" *+!U{X\smash{^{\vee}}}
    \endxy}
  \ = {\xy /r1pc/: (0,1.5)="T1", (0,-1.5)="T0",
    (0,-.25) *+!{\makebox[1.5em]{$\rho\mathstrut$}} *\frm{-}="BX1",
    "BX1"!D; p-(2,0) \mycap{.75}
    ?>; \makecoord{p}{"T1"} **\dir{-} ?>="L1",
    "BX1"!U!L(.6); \makecoord{p}{"T1"} **\dir{-} ?>="L2",
    "BX1"!U!R(.6); p+(2,0) \mycap{.75}
    ?>; \makecoord{p}{"T0"} **\dir{-} ?>="L3",
    "L1" *+!D{X\smash{^{\vee}}},
    "L2" *+!D{A},
    "L3" *+!U{X\smash{^{\vee}}}
    \endxy}.
\end{equation*}

There is a similar construction for comodules: Let $C$ be a coalgebra in $\mathcal{C}$, and let $X$ be a right $C$-comodule in $\mathcal{C}$ with coaction $\delta: X \to X \otimes C$. If $X$ has a left dual object $X^{\vee}$, then $X^{\vee}$ is a left $C$-comodule by the coaction defined to be the morphism corresponding to $\delta$ via
\begin{equation*}
  \Hom_{\mathcal{C}}(X, X \otimes C)
  \xrightarrow[\text{\eqref{eq:reciprocity-1}}]{\quad \cong \quad}
  \Hom_{\mathcal{C}}(X^{\vee} \otimes X, C)
  \xrightarrow[\text{\eqref{eq:reciprocity-2}}]{\quad \cong \quad}
  \Hom_{\mathcal{C}}(X^{\vee}, C \otimes X^{\vee}).
\end{equation*}

\section{Quasi-Hopf algebras}
\label{sec:quasi-hopf}

\subsection{Notation}

Throughout this paper, we work over a fixed field $k$. Unless otherwise noted, the symbol $\otimes$ means the tensor product over $k$. For a vector space $V$ over $k$ and a positive integer $n$, we denote by $V^{\otimes n}$ the $n$-th tensor power of $V$ over $k$. An element $x \in V^{\otimes n}$ is written symbolically as
\begin{equation*}
  x = x_1 \otimes \dotsb \otimes x_n,
\end{equation*}
although an element of $V^{\otimes n}$ is a sum of the form $\sum_i x_{i,1} \otimes \dotsb \otimes x_{i,n}$ in general. Given a permutation $\sigma$ on the set $\{ 1, \dotsc, n \}$, we write $x_{\sigma(1) \, \cdots \, \sigma(n)} = x_{\sigma(1)} \otimes \dotsb \otimes x_{\sigma(n)}$.

By a $k$-algebra, we always mean an associative unital algebra over the field $k$.
Given a $k$-algebra $A$, we denote by $A^{\op}$ the opposite algebra of $A$.
We note that $A^{\otimes n}$ ($n = 1, 2, 3, \dotsc$) is a $k$-algebra by the component-wise multiplication.
If $x$ is an invertible element of $A^{\otimes n}$, then we write its inverse as
\begin{equation*}
  x^{-1} = \overline{x}_1 \otimes \dotsb \otimes \overline{x}_n.
\end{equation*}

Given a vector space $M$, we denote its dual space by $M^* = \Hom_k(M, k)$. For $\xi \in M^*$ and $m \in M$, we often write $\xi(m)$ as $\langle \xi, m \rangle$. If $M$ is a left $A$-module, then $M^*$ is a right $A$-module by the action $\leftharpoonup$ given by $\langle \xi \leftharpoonup a, m \rangle = \langle \xi, a m \rangle$ for $\xi \in M^*$, $a \in A$ and $m \in M$. Similarly, if $M$ is a right $A$-module, then $M^*$ is a left $A$-module by $\langle a \rightharpoonup \xi, m \rangle = \langle \xi, m a \rangle$ for $a \in A$, $\xi \in M^*$ and $m \in M$.

\subsection{Quasi-Hopf algebras}

A {\em quasi-Hopf algebra} \cite{MR1047964,MR1321145} is a datum $H = (H, \Delta, \qhepsilon, \qhphi, \qhS, \qhalpha, \qhbeta)$ consisting of a $k$-algebra $H$, algebra maps $\Delta: H \to H^{\otimes 2}$ and $\qhepsilon: H \to k$, an invertible element $\qhphi \in H^{\otimes 3}$, an anti-algebra map $\qhS: H \to H$, and elements $\qhalpha, \qhbeta \in H$ satisfying the equations
\begin{gather}
  \label{eq:q-Hopf-def-1}
  h_{(1)} \otimes \Delta(h_{(2)})
  = \qhphi \cdot (\Delta(h_{(1)}) \otimes h_{(2)}) \cdot \qhphi^{-1}, \\
  \label{eq:q-Hopf-def-2}
  \qhepsilon(h_{(1)}) h_{(2)} = h = h_{(1)} \qhepsilon(h_{(2)}), \\
  \label{eq:q-Hopf-def-3}
  (\id \otimes \id \otimes \Delta)(\qhphi) (\Delta \otimes \id \otimes \id)(\qhphi)
  = (1 \otimes \qhphi) (\id \otimes \Delta \otimes \id) (\qhphi) (\qhphi \otimes 1), \\
  \label{eq:q-Hopf-def-4}
  (\id \otimes \qhepsilon \otimes \id)(\qhphi) = 1 \otimes 1, \\
  \label{eq:q-Hopf-def-5}
  \qhS(h_{(1)}) \qhalpha h_{(2)} = \qhepsilon(h) \qhalpha, \quad
  h_{(1)} \qhbeta \qhS(h_{(2)}) = \qhepsilon(h) \qhbeta, \\
  \label{eq:q-Hopf-def-6}
  \qhphi_1 \qhbeta \qhS(\qhphi_2) \qhalpha \qhphi_3 = 1
  = \qhS(\overline{\qhphi}_1) \qhalpha \overline{\qhphi}_2 \qhbeta \qhS(\overline{\qhphi}_3),
\end{gather}
for all $h \in H$, where $h_{(1)} \otimes h_{(2)}$ is the symbolic notation for $\Delta(h)$. The maps $\Delta$, $\qhepsilon$ and $\qhS$ are called the comultiplication, the counit and the antipode of $H$, respectively.

Equations \eqref{eq:q-Hopf-def-2}--\eqref{eq:q-Hopf-def-4} imply
\begin{equation}
  \label{eq:q-Hopf-phi-eps}
  (\id \otimes \id \otimes \qhepsilon)(\qhphi) = 1 \otimes 1 = (\qhepsilon \otimes \id \otimes \id)(\qhphi).
\end{equation}
By applying $\qhepsilon$ to equation \eqref{eq:q-Hopf-def-6}, we have $\qhepsilon(\qhalpha) \qhepsilon(\qhbeta) = 1$. Hence the equation
\begin{equation}
  \label{eq:q-Hopf-def-7}
  \qhepsilon(\qhalpha) = \qhepsilon(\qhbeta) = 1
\end{equation}
holds if we replace the elements $\qhalpha$ and $\qhbeta$ with  $\qhepsilon(\qhbeta) \qhalpha$ and $\qhepsilon(\qhalpha) \qhbeta$, respectively. For simplicity, all quasi-Hopf algebras in this paper are assumed to satisfy \eqref{eq:q-Hopf-def-7}. A category-theoretical meaning of this assumption is explained in Remark~\ref{rem:epsilon-alpha}.

Now we fix a quasi-Hopf algebra $H = (H, \Delta, \qhepsilon, \qhphi, \qhS, \qhalpha, \qhbeta)$, and assume that the antipode $\qhS$ is bijective with inverse $\overline{\qhS}$ (this assumption is automatically satisfied when $H$ is finite-dimensional; see \cite{MR1995128} and \cite{MR2037710}).
We define the quasi-Hopf algebras $H^{\op}$ and $H^{\cop}$ by
\begin{gather}
  \label{eq:q-Hopf-def-op}
  H^{\op} = (H^{\op}, \Delta, \qhepsilon, \qhphi^{-1}, \overline{\qhS}, \overline{\qhS}(\qhbeta), \overline{\qhS}(\qhalpha)), \\
  \label{eq:q-Hopf-def-cop}
  H^{\cop} = (H, \Delta^{\cop}, \qhepsilon, \qhphi_{321}^{-1}, \overline{\qhS}, \overline{\qhS}(\qhalpha), \overline{\qhS}(\qhbeta)),
\end{gather}
respectively, where $\Delta^{\cop}(h) = h_{(2)} \otimes h_{(1)}$ ($h \in H$). Hence,
\begin{equation*}
  H^{\op, \cop} := (H^{\op})^{\cop}
  = (H^{\op}, \Delta^{\cop}, \qhepsilon, \qhphi_{321}, \qhS, \qhbeta, \qhalpha)
  = (H^{\cop})^{\op}.
\end{equation*}

To express iterated comultiplications, we use the following variant of the Sweedler notation: For $h \in H$, we write $h_{(1)} \otimes \Delta(h_{(2)}) = h_{(1)} \otimes h_{(2,1)} \otimes h_{(2,2)}$ and $\Delta(h_{(1)}) \otimes h_{(2)} = h_{(1,1)} \otimes h_{(1,2)} \otimes h_{(2)}$. If $H$ is an ordinary Hopf algebra, then the equation $h_{(1,1)} \otimes h_{(1,2)} \qhS(h_{(2)}) = h \otimes 1$ holds for all $h \in H$. To state quasi-Hopf analogues of such equations, we introduce:
\begin{align}
  \label{eq:q-Hopf-def-pR}
  \qhpr & := \overline{\qhphi}_1 \otimes \overline{\qhphi}_2 \qhbeta \qhS(\overline{\qhphi}_3), \\
  \label{eq:q-Hopf-def-qR}
  \qhqr & := \qhphi_1 \otimes \overline{\qhS}(\qhalpha \qhphi_3) \qhphi_2, \\
  \label{eq:q-Hopf-def-pL}
  \qhpl & := \qhphi_2 \overline{\qhS}(\qhphi_1 \qhbeta) \otimes \qhphi_3, \\
  \label{eq:q-Hopf-def-qL}
  \qhql & := \qhS(\overline{\qhphi}_1) \qhalpha \overline{\qhphi}_2 \otimes \overline{\qhphi}_3.
\end{align}
Then, for all $h \in H$, we have
\begin{align}
  \label{eq:q-Hopf-pR}
  h_{(1,1)} \qhpr_1 \otimes h_{(1,2)} \qhpr_2 \qhS(h_{(2)})
  & = \qhpr (h \otimes 1), \\
  \label{eq:q-Hopf-qR}
  \qhqr_1 h_{(1,1)} \otimes \overline{\qhS}(h_{(2)}) \qhqr_2 h_{(1,2)}
  & = (h \otimes 1) \qhqr, \\
  \label{eq:q-Hopf-pL}
  h_{(2,1)} \qhpl_1 \overline{\qhS}(h_{(1)}) \otimes h_{(2,2)} \qhpl_2
  & = \qhpl (1 \otimes h), \\
  \label{eq:q-Hopf-qL}
  \qhS(h_{(1)}) \qhql_1 h_{(2,1)} \otimes \qhql_2 h_{(2,2)}
  & = (1 \otimes h) \qhql.
\end{align}
The following formulas are also useful:
\begin{align}
  \label{eq:q-Hopf-pR-qR-1}
  \qhqr_{1(1)} \qhpr_1 \otimes \qhqr_{1(2)} \qhpr_2 \qhS(\qhqr_2)
  & = 1 \otimes 1, \\
  \label{eq:q-Hopf-pR-qR-2}
  \qhqr_1 \qhpr_{1(1)} \otimes \overline{\qhS}(\qhpr_2) \qhqr_2 \qhpr_{1(2)}
  & = 1 \otimes 1, \\
  \label{eq:q-Hopf-pL-qL-1}
  \qhql_{2(1)} \qhpl_1 \overline{\qhS}(\qhql_1) \otimes \qhql_{2(2)} \qhpl_2
  & = 1 \otimes 1, \\
  \label{eq:q-Hopf-pL-qL-2}
  \qhS(\qhpl_1) \qhql_1 \qhpl_{2(1)}\otimes \qhql_2 \qhpl_{2(2)}
  & = 1 \otimes 1.
\end{align}
One can find a proof of  equations \eqref{eq:q-Hopf-pR}--\eqref{eq:q-Hopf-pL-qL-2} in, {\it e.g.}, \cite{MR1696105,MR3929714}. A graphical interpretation of these equations is given in \cite{MR1669685}.

Equations~\eqref{eq:q-Hopf-pR}--\eqref{eq:q-Hopf-qL} are equivalent in the following sense: Let $\qhpr_{\op}$, $\qhqr_{\op}$, $\qhpl_{\op}$ and $\qhql_{\op}$ be the elements $\qhpr$, $\qhqr$, $\qhpl$ and $\qhql$ for $H^{\op}$, respectively. We use the same notation for $H^{\cop}$. Then we have
\begin{gather}
  \label{eq:q-Hopf-pq-op}
  \qhpr_{\op} = \qhqr,
  \quad \qhqr_{\op} = \qhpr,
  \quad \qhpl_{\op} = \qhql,
  \quad \qhql_{\op} = \qhpl, \\
  \label{eq:q-Hopf-pq-cop}
  \qhpr_{\cop} = \qhpl_{21},
  \quad \qhqr_{\cop} = \qhql_{21},
  \quad \qhpl_{\cop} = \qhpr_{21},
  \quad \qhql_{\cop} = \qhqr_{21}
\end{gather}
by defining formulas \eqref{eq:q-Hopf-def-op}--\eqref{eq:q-Hopf-def-qL}. Thus~\eqref{eq:q-Hopf-qR}, \eqref{eq:q-Hopf-pL} and \eqref{eq:q-Hopf-qL} are obtained by applying \eqref{eq:q-Hopf-pR} to the quasi-Hopf algebras $H^{\op}$, $H^{\cop}$ and $H^{\op,\cop}$, respectively. Equations~\eqref{eq:q-Hopf-pR-qR-1}--\eqref{eq:q-Hopf-pL-qL-2} are also equivalent in the same sense.

The antipode of a Hopf algebra is known to be an anti-coalgebra map. To extend this result to quasi-Hopf algebras, it is convenient to introduce the twisting operation for quasi-Hopf algebras \cite{MR1047964,MR1321145}. A {\em gauge transformation} of $H$ is an invertible element $F \in H^{\otimes 2}$ such that $\qhepsilon(F_1) F_2 = 1 = F_1 \qhepsilon(F_2)$. Given a gauge transformation $F$ of $H$, we define
\begin{gather*}
  \Delta^F(h) = F \Delta(h) F^{-1} \quad (h \in H),
  \quad \qhalpha^F = \qhS(\overline{F}_1) \qhalpha \overline{F}_2,
  \quad \qhbeta^F = F_1 \qhbeta \qhS(F_2), \\
  \qhphi^F = (1 \otimes F) (\id \otimes \Delta)(F) \qhphi (\Delta \otimes \id)(F^{-1}) (F^{-1} \otimes 1).
\end{gather*}
Then $H^F := (H, \Delta^F, \qhepsilon, \qhphi^F, \qhS, \qhalpha^F, \qhbeta^F)$ is a quasi-Hopf algebra, called the {\em twist} of $H$ by $F$. Drinfeld \cite{MR1047964} introduced a special gauge transformation $\qhf$ given as follows: We define $\mathbbm{a}, \mathbbm{b} \in H^{\otimes 2}$ by
\begin{gather}
  \label{eq:q-Hopf-def-a}
  \begin{aligned}
    \mathbbm{a}
    & = \qhS(\overline{\qhphi}_1 \qhphi_2) \qhalpha \overline{\qhphi}_2 \qhphi_{3(1)}
    \otimes \qhS(\qhphi_1) \qhalpha \overline{\qhphi}_3 \qhphi_{3(2)} \\
    {}^{\eqref{eq:q-Hopf-def-3}}
    & = \qhS(\qhphi_2 \overline{\qhphi}_{1(2)}) \qhalpha \qhphi_3 \overline{\qhphi}_2
    \otimes \qhS(\qhphi_1 \overline{\qhphi}_{1(1)}) \qhalpha \overline{\qhphi}_3, \\
  \end{aligned} \\
  \label{eq:q-Hopf-def-b}
  \begin{aligned}
    \mathbbm{b}
    & = \qhphi_{1(1)} \overline{\qhphi}_1 \qhbeta \qhS(\qhphi_3)
    \otimes \qhphi_{1(2)} \overline{\qhphi}_2 \qhbeta \qhS(\qhphi_2 \overline{\qhphi}_3) \\
    {}^{\eqref{eq:q-Hopf-def-3}}
    & = \overline{\qhphi}_1 \qhbeta \qhS(\overline{\qhphi}_{3(2)} \qhphi_3)
    \otimes \overline{\qhphi}_2 \qhphi_1 \qhbeta \qhS(\overline{\qhphi}_{3(1)} \qhphi_2).
  \end{aligned}
\end{gather}
Then $\qhf$ and its inverse are given by the following formulas:
\begin{gather}
  \label{eq:q-Hopf-def-f}
  \qhf = (\qhS \otimes \qhS)\Delta^{\cop}(\qhpr_1) \, \mathbbm{a} \, \Delta(\qhpr_2),
  \quad
  \qhf^{-1} = \Delta(\qhql_1) \, \mathbbm{b} \, (\qhS \otimes \qhS) \Delta^{\cop}(\qhql_2).
\end{gather}
The antipode of $H$ is shown to be an isomorphism $\qhS: H^{\op, \cop} \to H^{\qhf}$ of quasi-Hopf algebras. Namely, we have
\begin{gather}
  \label{eq:q-Hopf-f-1}
  \Delta^{\qhf}(\qhS(h)) = \qhS(h_{(2)}) \otimes \qhS(h_{(1)}),
  \quad \qhepsilon(\qhS(h)) = \qhepsilon(h), \\
  \label{eq:q-Hopf-f-3}
  \qhS(\qhphi_3) \otimes \qhS(\qhphi_2) \otimes \qhS(\qhphi_1) = \qhphi^{\qhf}, \\
  \label{eq:q-Hopf-f-4}
  \qhS(\qhbeta) = \qhS(\overline{\qhf}_1) \qhalpha \overline{\qhf}_2,
  \quad
  \qhS(\qhalpha) = \qhf_1 \qhbeta \qhS(\qhf_2)
\end{gather}
for all $h \in H$. The following equations are also useful:
\begin{equation}
  \label{eq:q-Hopf-delta-alpha}
  \Delta(\qhalpha) = \qhf^{-1} \mathbbm{a},
  \quad
  \Delta(\qhbeta) = \mathbbm{b} \qhf.
\end{equation}
For later use, we prove the following Lemmas~\ref{lem:q-Hopf-f-op-cop}--\ref{lem:q-Hopf-pR-qL}

\begin{lemma}
  \label{lem:q-Hopf-f-op-cop}
  Let $\qhf_{\op}$ and $\qhf_{\cop}$ be the elements $\qhf$ for the quasi-Hopf algebras $H^{\op}$ and $H^{\cop}$, respectively. These elements are given explicitly by
  \begin{equation}
    \label{eq:q-Hopf-f-op-cop}
    \qhf_{\op} = \overline{\qhS}(\overline{\qhf}_2) \otimes \overline{\qhS}(\overline{\qhf}_1),
    \quad
    \qhf_{\cop} = \overline{\qhS}(\qhf_1) \otimes \overline{\qhS}(\qhf_2).
  \end{equation}
\end{lemma}
\begin{proof}
  By~\eqref{eq:q-Hopf-pq-op},~\eqref{eq:q-Hopf-def-a},~\eqref{eq:q-Hopf-def-f}, and the definition of $H^{\op}$, we have
  \begin{align*}
    \qhf_{\op}
    = \Delta(\qhqr_2) \cdot (\overline{\qhphi}_{3(1)} \qhphi_2  \overline{\qhS}(\qhbeta)
    \overline{\qhS}(\overline{\qhphi}_2 \qhphi_1))
    \otimes \overline{\qhphi}_{3(2)} \qhphi_3 \overline{\qhS}(\qhbeta)
    \overline{\qhS}(\overline{\qhphi}_1))
    \cdot (\overline{\qhS} \otimes \overline{\qhS}) \Delta^{\cop}(\qhqr_{1}).
  \end{align*}
  Thus we compute\footnote{In this paper, the red dotted underlines $\qhano{\qquad}$ in an expression indicate which parts are changed to obtain the next expression.}:
  \begin{align*}
    (\qhS \otimes \qhS)(\qhf_{\op})
    & = \Delta^{\cop}(\qhqr_{1})
      \cdot (\qhano{\overline{\qhphi}_2 \qhphi_1 \qhbeta \qhS(\overline{\qhphi}_{3(1)} \qhphi_2)}
      \otimes \qhano{\overline{\qhphi}_1 \qhbeta \qhS(\overline{\qhphi}_{3(2)} \qhphi_3)})
      \cdot (\qhS \otimes \qhS) \Delta(\qhqr_2) \\
    {}^{\eqref{eq:q-Hopf-def-b}}
    & = \Delta^{\cop}(\qhqr_{1})
      \cdot \mathbbm{b}_{21}
      \cdot \qhano{(\qhS \otimes \qhS) \Delta(\qhqr_2)} \\
    {}^{\eqref{eq:q-Hopf-f-1}}
    & = \Delta^{\cop}(\qhqr_{1})
      \cdot \qhano{\mathbbm{b}_{21} \cdot \qhf_{21}}
      \cdot \Delta^{\cop}(\qhS(\qhqr_2)) \cdot \qhf_{21}^{-1} \\
    {}^{\eqref{eq:q-Hopf-delta-alpha}}
    & = \Delta^{\cop}(\qhqr_{1}) \cdot \Delta^{\cop}(\qhbeta) \cdot \Delta^{\cop}(\qhS(\qhqr_2)) \cdot \qhf_{21}^{-1} \\
    & = \Delta^{\cop}(\qhqr_{1} \qhbeta \qhS(\qhqr_2)) \cdot \qhf_{21}^{-1}
      \stackrel{\text{\eqref{eq:q-Hopf-def-6}}}{=} \qhf_{21}^{-1}.
  \end{align*}
  This shows the first equation of~\eqref{eq:q-Hopf-f-op-cop}. To prove the second one, we note:
  \begin{equation*}
    \qhf_{\cop}
    = (\overline{\qhS} \otimes \overline{\qhS})\Delta(\qhpl_2) \cdot (\overline{\qhS}(\qhphi_3 \overline{\qhphi}_2) \overline{\qhS}(\qhalpha) \qhphi_2 \overline{\qhphi}_{1(2)}
      \otimes \overline{\qhS}(\overline{\qhphi}_3) \overline{\qhS}(\qhalpha) \qhphi_1 \overline{\qhphi}_{1(1)})
      \cdot \Delta^{\cop}(\qhpl_1).
  \end{equation*}
  Thus we have
  \begin{align*}
    (\qhS \otimes \qhS)(\qhf_{\cop})
    & = \qhano{(\qhS \otimes \qhS) \Delta^{\cop}(\qhpl_1)}
      \cdot (\qhano{\qhS(\qhphi_2 \overline{\qhphi}_{1(2)}) \qhalpha \qhphi_3 \overline{\qhphi}_2}
      \otimes \qhano{\qhS(\qhphi_1 \overline{\qhphi}_{1(1)}) \qhalpha \overline{\qhphi}_3})
      \cdot \Delta(\qhpl_2) \\
    {}^\text{\eqref{eq:q-Hopf-def-a}, \eqref{eq:q-Hopf-f-1}}
    & = \qhf \cdot \Delta(\qhS(\qhpl_1)) \cdot \qhano{\qhf^{-1} \cdot \mathbbm{a}}
      \cdot \Delta(\qhpl_2) \\
    {}^{\eqref{eq:q-Hopf-delta-alpha}}
    & = \qhf \cdot \Delta(\qhS(\qhpl_1)) \cdot \Delta(\qhalpha) \cdot \Delta(\qhpl_2) \\
    & =
      \qhf \cdot \Delta(\qhS(\qhpl_1) \qhalpha \qhpl_2)
      \stackrel{\text{\eqref{eq:q-Hopf-def-6}}}{=} \qhf. \qedhere
  \end{align*}
\end{proof}

\begin{lemma}
  The following equations hold:
  \begin{align}
    \label{eq:q-Hopf-qR-phi}
    \qhqr_{1(1)} \overline{\qhphi}_1 \otimes \qhqr_{1(2)} \overline{\qhphi}_2 \otimes \qhqr_{2} \overline{\qhphi}_3
    & = \qhphi_1 \otimes \qhqr_1 \qhphi_{2(1)} \otimes \overline{\qhS}(\qhphi_3) \qhqr_2 \qhphi_{2(2)}, \\
    \label{eq:q-Hopf-pR-phi}
    \qhphi_1 \qhpr_{1(1)} \otimes \qhphi_2 \qhpr_{1(2)} \otimes \qhphi_3 \qhpr_{2}
    & = \overline{\qhphi}_1 \otimes \overline{\qhphi}_{2(1)} \qhpr_1 \otimes \overline{\qhphi}_{2(2)} \qhpr_2 \qhS(\overline{\qhphi}_3).
  \end{align}
\end{lemma}
\begin{proof}
  The first equation is verified as follows:
  \begin{align*}
    & \qhano{\qhqr_{1(1)}} \overline{\qhphi}_1
      \otimes \qhano{\qhqr_{1(2)}} \overline{\qhphi}_2
      \otimes \qhano{\qhqr_{2}} \overline{\qhphi}_3 \\
    {}^{\eqref{eq:q-Hopf-def-qR}}
    & = \qhano{\qhphi_{1(1)} \overline{\qhphi}_1}
      \otimes \qhano{\qhphi_{1(2)} \overline{\qhphi}_2}
      \otimes \overline{\qhS}(\qhS(\qhano{\qhphi_2 \overline{\qhphi}_3})
      \qhalpha \qhano{\qhphi_3}) \\
    {}^{\eqref{eq:q-Hopf-def-3}}
    & = \overline{\qhphi}{}'_1 \qhphi'_1 \otimes \overline{\qhphi}{}'_2 \qhphi_1 \qhphi'_{2(1)}
      \otimes \overline{\qhS}(\qhano{\qhS(\overline{\qhphi}{}'_{3(1)}} \qhphi_2 \qhphi{}'_{2(2)})
      \qhano{\qhalpha \overline{\qhphi}{}'_{3(2)}} \qhphi_3 \qhphi{}'_3) \\
    {}^{\eqref{eq:q-Hopf-def-5}, \eqref{eq:q-Hopf-phi-eps}}
    & = \qhphi'_1 \otimes \qhano{\qhphi_1} \qhphi'_{2(1)}
      \otimes \qhano{\overline{\qhS}(\qhalpha \qhphi{}_3} \qhphi'_3)
      \qhano{\qhphi_2} \qhphi'_{2(2)} \\
      {}^{\eqref{eq:q-Hopf-def-qR}}
    & = \qhphi'_1 \otimes \qhqr_1 \qhphi'_{2(1)} \otimes \overline{\qhS}(\qhphi'_3) \qhqr_2 \qhphi'_{2(2)},
  \end{align*}
  where $\qhphi'$ is a copy of $\qhphi$. The second one is obtained by applying~\eqref{eq:q-Hopf-qR-phi} to $H^{\op}$.
\end{proof}

\begin{lemma}
  \label{lem:q-Hopf-pR-qL}
  The following equations hold:
  \begin{gather}
    \label{eq:q-Hopf-pquv-1}
    \qhS(\qhpr_1) \qhql_1 \qhpr_{2(1)} \otimes \qhql_2 \qhpr_{2(2)}
    = \qhS(\qhpl_2) \qhf_1 \otimes \qhS(\qhpl_1) \qhf_2, \\
    \label{eq:q-Hopf-pquv-2}
    \qhqr_{2(1)} \qhpl_1 \overline{\qhS}(\qhqr_1) \otimes \qhqr_{2(2)} \qhpl_2
    = \overline{\qhS}(\qhql_2 \overline{\qhf}_2) \otimes \overline{\qhS}(\qhql_1 \overline{\qhf}_1), \\
    \label{eq:q-Hopf-pquv-3}
    \qhqr_1 \qhpl_{1(1)} \otimes \overline{\qhS}(\qhpl_2) \qhqr_2 \qhpl_{1(2)}
    = \overline{\qhS}(\qhf_2 \qhpr_2) \otimes \overline{\qhS}(\qhf_1 \qhpr_1), \\
    \label{eq:q-Hopf-pquv-4}
    \qhql_{1(1)} \qhpr_1 \otimes \qhql_{1(2)} \qhpr_2 \qhS(\qhql_2)
    = \overline{\qhf}_1 \qhS(\qhqr_2) \otimes \overline{\qhf}_2 \qhS(\qhqr_1).
  \end{gather}
\end{lemma}
\begin{proof}
  By \eqref{eq:q-Hopf-f-1} and \eqref{eq:q-Hopf-delta-alpha}, we have
  \begin{equation}
    \label{eq:q-Hopf-pR-qL-1}
    \Delta(\qhbeta \qhS(h))
    = \mathbbm{b} \qhf \Delta(\qhS(h))
    = \mathbbm{b} (\qhS(h_{(2)}) \otimes \qhS(h_{(1)})) \qhf
  \end{equation}
  for all $h \in H$. We now verify \eqref{eq:q-Hopf-pquv-1} as follows:
  \allowdisplaybreaks
  \begin{align*}
    & \qhS(\qhpr_1) \qhql_1 \qhpr_{2(1)} \otimes \qhql_2 \qhpr_{2(2)} \\
    {}^{\eqref{eq:q-Hopf-pR}, \eqref{eq:q-Hopf-qL}}
    & = \qhS(\overline{\qhphi}{}_1') \qhS(\overline{\qhphi}_1) 
      \qhalpha \overline{\qhphi}_2 (\overline{\qhphi}{}_2'
      \qhano{\qhbeta \qhS(\overline{\qhphi}{}_3')})_{(1)}
      \otimes \overline{\qhphi}_3 (\overline{\qhphi}{}_2'
      \qhano{\qhbeta \qhS(\overline{\qhphi}{}_3')})_{(2)} \\
    {}^{\eqref{eq:q-Hopf-pR-qL-1}}
    & = \qhS(\overline{\qhphi}_1 \overline{\qhphi}{}_1')
      \qhalpha \overline{\qhphi}_2 \overline{\qhphi}{}_{2(1)}' \qhano{\mathbbm{b}_1} \qhS(\overline{\qhphi}{}_{3(2)}') \qhf_1
      \otimes \overline{\qhphi}_3 \overline{\qhphi}{}_{2(2)}' \qhano{\mathbbm{b}_2} \qhS(\overline{\qhphi}{}_{3(1)}') \qhf_2 \\
    {}^{\eqref{eq:q-Hopf-def-pL}, \eqref{eq:q-Hopf-def-b}}
    & = \qhS(\overline{\qhphi}_1 \overline{\qhphi}{}_1')
      \qhalpha \overline{\qhphi}_2 \overline{\qhphi}{}_{2(1)}' \overline{\qhphi}{}''_1
      \qhbeta \qhS(\overline{\qhphi}{}''_{3(2)} \qhpl_2) \qhS(\overline{\qhphi}{}_{3(2)}') \qhf_1 \\
    & \qquad \qquad
      \otimes \overline{\qhphi}_3 \overline{\qhphi}{}_{2(2)}'
      \overline{\qhphi}{}''_2 \qhS(\overline{\qhphi}{}''_{3(1)} \qhpl_1)
      \qhS(\overline{\qhphi}{}_{3(1)}') \qhf_2 \\
    & = \qhS(\qhano{\overline{\qhphi}_1 \overline{\qhphi}{}_1'})
      \qhalpha \qhano{\overline{\qhphi}_2 \overline{\qhphi}{}_{2(1)}'
      \overline{\qhphi}{}''_1}
      \qhbeta \qhS((\qhano{\overline{\qhphi}{}_{3}' \overline{\qhphi}{}''_{3}})_{(2)} \qhpl_2)
      \qhf_1 \\
    & \qquad \qquad
      \otimes \qhano{\overline{\qhphi}_3 \overline{\qhphi}{}_{2(2)}'
      \overline{\qhphi}{}''_2}
      \qhS((\qhano{\overline{\qhphi}{}_{3}' \overline{\qhphi}{}''_{3}})_{(1)} \qhpl_1)
      \qhf_2 \\
    {}^{\eqref{eq:q-Hopf-def-3}}
    & = \qhano{\qhS(\overline{\qhphi}_{1(1)}} \overline{\qhphi}{}_1')
      \qhano{\qhalpha \overline{\qhphi}{}_{1(2)}} \overline{\qhphi}{}'_2
      \qhbeta \qhS((\overline{\qhphi}{}_{3} \overline{\qhphi}{}'_{3(2)})_{(2)} \qhpl_2) \qhf_1 \\
    & \qquad \qquad
      \otimes \overline{\qhphi}_2 \overline{\qhphi}{}'_{3(1)}
      \qhS((\overline{\qhphi}{}_{3} \overline{\qhphi}{}'_{3(2)})_{(1)} \qhpl_1) \qhf_2 \\
    {}^{\eqref{eq:q-Hopf-def-5}, \eqref{eq:q-Hopf-phi-eps}}
    & = \qhS(\overline{\qhphi}{}_1')
      \qhalpha \overline{\qhphi}{}_{1(2)}'
      \qhbeta \qhS(\overline{\qhphi}{}'_{3(2,2)} \qhpl_2) \qhf_1
      \otimes \overline{\qhphi}{}'_{3(1)}
      \qhS(\overline{\qhphi}{}'_{3(2,1)} \qhpl_1) \qhf_2 \\
    & = \qhS(\overline{\qhphi}{}_1')
      \qhalpha \overline{\qhphi}{}_{1(2)}'
      \qhbeta \qhS(\qhano{\overline{\qhphi}{}'_{3(2,2)} \qhpl_2}) \qhf_1
      \otimes 
      \qhS(\qhano{\overline{\qhphi}{}'_{3(2,1)}
      \qhpl_1 \overline{\qhS}(\overline{\qhphi}{}'_{3(1)})}) \qhf_2 \\
    {}^{\text{\eqref{eq:q-Hopf-pL}}}
    & = \qhano{\qhS(\overline{\qhphi}{}_1')
      \qhalpha \overline{\qhphi}{}_{1(2)}'
      \qhbeta \qhS}(\qhpl_2 \qhano{\overline{\qhphi}{}'_3}) \qhf_1
      \otimes \qhS(\qhpl_1) \qhf_2 \\
    {}^{\eqref{eq:q-Hopf-def-6}}
    & = \qhS(\qhpl_2) \qhf_1
      \otimes \qhS(\qhpl_1) \qhf_2,
  \end{align*}
  where $\qhphi'$ and $\qhphi''$ are copies of $\qhphi$. Equations~\eqref{eq:q-Hopf-pquv-2}, \eqref{eq:q-Hopf-pquv-3} and \eqref{eq:q-Hopf-pquv-4} are obtained by applying the equation \eqref{eq:q-Hopf-pquv-1} to $H^{\op}$, $H^{\cop}$ and $H^{\op, \cop}$, respectively.
\end{proof}

\subsection{Modules over a quasi-Hopf algebra}
\label{subsec:q-Hopf-and-monoidal}

Let $H$ be a quasi-Hopf algebra with bijective antipode. If $V$ and $W$ are left $H$-modules, then their tensor product $V \otimes W$ is a left $H$-module by the action given by $h \cdot (v \otimes w) = h_{(1)} v \otimes h_{(2)} w$ ($h \in H$, $v \in V$, $w \in W$). The vector space $\unitobj := k$ is a left $H$-module through the counit $\qhepsilon: H \to k$. The category ${}_H \Mod$ of left $H$-modules is a monoidal category with the associator $\Phi$, the left unit isomorphism $l$ and the right unit isomorphism $r$ given by
\begin{gather*}
  \Phi_{X,Y,Z}((x \otimes y) \otimes z) = \qhphi_1 x \otimes (\qhphi_2 y \otimes \qhphi_3 z), \\
  l_{X}(1 \otimes x) = x,
  \quad \text{and} \quad
  r_{X}(x \otimes 1) = x,
\end{gather*}
respectively, for $X, Y, Z \in {}_H \Mod$, $x \in X$, $y \in Y$ and $z \in Z$.

For a finite-dimensional left $H$-module $X$, we define the left $H$-module $X^{\vee}$ to be the vector space $X^{\vee} = X^*$ with the left $H$-module structure given by $h \cdot \xi = \xi \leftharpoonup \qhS(h)$ for $h \in H$ and $\xi \in X^{\vee}$. We fix a basis $\{ x_i \}$ of $X$ and let $\{ x^i \}$ be the dual basis of $\{ x_i \}$. We define linear maps $\eval_X$ and $\coev_X$ by
\begin{gather*}
  \eval_X: X^{\vee} \otimes X \to \unitobj,
  \quad \eval_X(\xi \otimes x) = \langle \xi, \qhalpha x \rangle, \\
  \coev_X: \unitobj \to X \otimes X^{\vee}, \quad \coev_X(1) = \qhbeta x_i \otimes x^i,
\end{gather*}
respectively, for $\xi \in X^{\vee}$ and $x \in X$, where we have used the Einstein convention to suppress the sum over $i$. Equations~\eqref{eq:q-Hopf-def-5} and~\eqref{eq:q-Hopf-def-6} imply that $(X^{\vee}, \eval_X, \coev_X)$ is a left dual object of $X$.

\begin{remark}
  \label{rem:epsilon-alpha}
  A representation-theoretical meaning of the condition \eqref{eq:q-Hopf-def-7} is explained as follows: The triple $(\unitobj, r_{\unitobj}^{}, r_{\unitobj}^{-1})$ is a left dual object of $\unitobj$. Thus, by the uniqueness of a left dual object (see Subsection~\ref{subsec:duality}), there is an isomorphism $f: \unitobj \to \unitobj^{\vee}$ in ${}_H \Mod$ such that $(\id_{\unitobj} \otimes f) r_{\unitobj}^{-1} = \coev_{\unitobj}$ and $r_{\unitobj} = \eval_{\unitobj} (f \otimes \id_{\unitobj})$. By the former equation, $f$ is given by $\langle f(c), c' \rangle = \qhepsilon(\qhbeta) c c'$ for $c, c' \in k$. Thus \eqref{eq:q-Hopf-def-7} is equivalent to that the canonical isomorphism $f: \unitobj \to \unitobj^{\vee}$ in ${}_H \Mod$ coincides with the canonical isomorphism $k \cong k^*$ of vector spaces.
\end{remark}

Let $X$ be a finite-dimensional left $H$-module with basis $\{ x_i \}$, and let $\{ x^i \}$ be the dual basis of $X^*$. As we have recalled in Subsection~\ref{subsec:duality}, there is a natural isomorphism
\begin{equation*}
  \mathbbm{Q} = \mathbbm{Q}_{V,W,X}: \Hom_H(V, X \otimes W) \to \Hom_H(X^{\vee} \otimes V, W)
\end{equation*}
for $V, W \in {}_H \Mod$. We express this isomorphism explicitly. Given a morphism $f: A \to B \otimes C$ in ${}_H \Mod$, where $A, B, C \in {}_H \Mod$, we write $f(a)$ for $a \in A$ symbolically as $f(a)_B \otimes f(a)_C$.

\begin{lemma}
  \label{lem:q-Hopf-reciprocity-Q}
  The isomorphism $\mathbbm{Q}$ and its inverse are given by
  \begin{gather}
    \label{eq:q-Hopf-reciprocity-Q1}
    \mathbbm{Q}(f)(\xi \otimes v) = \langle \xi, \qhql_1 f(v)_X \rangle \, \qhql_2 f(v)_W, \\
    \label{eq:q-Hopf-reciprocity-Q2}
    \mathbbm{Q}^{-1}(g)(v) = \qhS(\qhpl_1) x_i \otimes g(x^i \otimes \qhpl_2 v)
  \end{gather}
  for $f \in \Hom_H(V, X \otimes W)$, $g \in \Hom_H(X^{\vee} \otimes V, W)$, $v \in V$ and $\xi \in X^{\vee}$.
\end{lemma}
\begin{proof}
  For $f \in \Hom_H(V, X \otimes W)$, $\xi \in X^{\vee}$ and $v \in V$, we compute:
  \begin{align*}
    \mathbbm{Q}(f)(\xi \otimes v)
    & = ((\eval_{X} \otimes \id_{W}) \Phi_{X^{\vee}, X, W}^{-1} (\id_{X^{\vee}} \otimes f))(\xi \otimes v) \\
    & = (\eval_{X} \otimes \id_{W}) (\overline{\qhphi}_1 \xi \otimes \overline{\qhphi}_2 f(v)_X \otimes \overline{\qhphi}_3 f(v)_W) \\
    & = \langle \overline{\qhphi}_1 \xi, \qhalpha \overline{\qhphi}_2 f(v)_X \rangle
      \overline{\qhphi}_3 f(v)_W \\
    & = \langle \xi, \qhano{\qhS(\overline{\qhphi}_1) \qhalpha \overline{\qhphi}_2} f(v)_X \rangle
      \qhano{\overline{\qhphi}_3} f(v)_W \\
    {}^{\eqref{eq:q-Hopf-def-qL}}
    & = \langle \xi, \qhql_1 f(v)_X \rangle \, \qhql_2 f(v)_W.
  \end{align*}
  To verify the expression for $\mathbbm{Q}^{-1}$, we note that the equation $x_i \otimes (x^i \circ T) = T(x_i) \otimes x^i$ holds in $X \otimes X^{*}$ for all $T \in \End_k(X)$. In particular, we have $x_i \otimes h x^i = \qhS(h) x_i \otimes x^i$ for all $h \in H$. Thus we have
  \begin{align*}
    \mathbbm{Q}^{-1}(g)(v)
    & = ((\id_X \otimes g) \Phi_{X,X^{\vee},V} (\coev_X \otimes \id_{V}))(v) \\
    & = (\id_X \otimes g)(\qhphi_1 \qhbeta x_i \otimes \qhphi_2 x^i \otimes \qhphi_3 v) \\
    & = \qhphi_1 \qhbeta x_i \otimes g(\qhphi_2 x^i \otimes \qhphi_3 v) \\
    & = \qhano{\qhphi_1 \qhbeta \qhS(\qhphi_2)} x_i
      \otimes g(x^i \otimes \qhano{\qhphi_3} v) \\
    {}^{\eqref{eq:q-Hopf-def-pL}}
    & = \qhS(\qhpl_1) x_i \otimes g(x^i \otimes \qhpl_2 v)
  \end{align*}
  for $g \in \Hom_H(X^{\vee} \otimes V, W)$ and $v \in V$.
\end{proof}

The following lemma is proved in a similar manner:

\begin{lemma}
  \label{lem:q-Hopf-reciprocity-P}
  The natural isomorphism
  \begin{equation*}
    \mathbbm{P} = \mathbbm{P}_{V,W,X}: \Hom_H(V \otimes X, W) \to \Hom_H(V, W \otimes X^{\vee})
  \end{equation*}
  and its inverse are given by
  \begin{gather}
    \label{eq:q-Hopf-reciprocity-P1}
    \mathbbm{P}(f)(v) = f(\qhpr_1 v \otimes \qhpr_2 x_i) \otimes x^i, \\
    \label{eq:q-Hopf-reciprocity-P2}
    \mathbbm{P}^{-1}(g)(v \otimes x) = \qhqr_1 g(v)_W \langle g(v)_{X^{\vee}}, \qhS(\qhqr_2) x \rangle
  \end{gather}
  for $f \in \Hom_H(V \otimes X, W)$, $g \in \Hom_H(V, W \otimes X^{\vee})$, $v \in V$ and $x \in X$.
\end{lemma}

A left $H$-module (co)algebra is a synonym for a (co)algebra in the monoidal category ${}_H \Mod$. As an application of Lemmas~\ref{lem:q-Hopf-reciprocity-Q} and~\ref{lem:q-Hopf-reciprocity-P}, we give a description of the dual (co)module in ${}_H \Mod$. For this purpose, we introduce the following notation:
\begin{align}
  \label{eq:q-Hopf-def-UL}
  \qhUL & = \overline{\qhf}_1 \qhS(\qhqr_2) \otimes \overline{\qhf}_2 \qhS(\qhqr_1), \\
  \label{eq:q-Hopf-def-VR}
  \qhVR & = \qhS(\qhpl_2) \qhf_1 \otimes \qhS(\qhpl_1) \overline{\qhf}_2.
\end{align}

\begin{lemma}
  \label{lem:dual-module-left}
  \textup{(i)} Let $A$ be a left $H$-module algebra. If $X$ is a finite-dimensional left $A$-module in ${}_H \Mod$ with action $\triangleright$, then the right $A$-module structure of $X^{\vee}$ is given by
  \begin{equation}
    \label{eq:module-alg-dual-module-1}
    \langle \xi \triangleleft a, x \rangle
    = \langle \xi, (\qhUL_1 a) \triangleright (\qhUL_2 x) \rangle
    \quad (a \in A, \xi \in X^{\vee}, x \in X).
  \end{equation}
  \textup{(ii)} Let $C$ be a left $H$-module coalgebra. If $X$ is a finite-dimensional right $C$-comodule in ${}_H \Mod$ with coaction $x \mapsto x_{(0)} \otimes x_{(1)}$, then the left $C$-comodule structure of $X^{\vee}$ is given by
  \begin{equation}
    \label{eq:module-alg-dual-comodule-1}
    \xi \mapsto \langle \xi, \qhVR_1 x_{i(0)} \rangle \, \qhVR_2 x_{i(1)} \otimes x^i
    \quad (\xi \in X^{\vee}),
  \end{equation}
  where $\{ x_i \}$ is a basis of $X$ and $\{ x^i \}$ is the dual basis of $\{ x_i \}$.
\end{lemma}
\begin{proof}
  We write $\rho(a \otimes x) = a \triangleright x$. We have $h \rho(a \otimes x) = \rho(h_{(1)} a \otimes h_{(2)} x)$ for all $h \in H$, $a \in A$ and $x \in X$ since $\rho: A \otimes X \to X$ is a morphism in ${}_H \Mod$. By Lemmas~\ref{lem:q-Hopf-reciprocity-Q} and \ref{lem:q-Hopf-reciprocity-P}, we have
  \begin{align*}
    \langle \xi \triangleleft a, x \rangle
    & = \langle (\mathbbm{Q}_{A, X^{\vee}, X} \mathbbm{P}_{A, X, X}(\rho))(\xi \otimes a), x \rangle \\
    & = \langle \xi, \qhql_1 \rho(\qhpr_1 a \otimes \qhpr_2 x_i) \rangle \langle \qhql_2 x^i, x \rangle \\
    & = \langle \xi, \rho(\qhano{\qhql_{1(1)} \qhpr_1} a
      \otimes \qhano{\qhql_{1(2)} \qhpr_2 \qhS(\qhql_2)} x_i) \rangle \langle x^i, x \rangle \\
    {}^{\eqref{eq:q-Hopf-pquv-4}}
    & = \langle \xi, \rho(\qhano{\overline{\qhf}_1 \qhS(\qhqr_2)} a
      \otimes \qhano{\overline{\qhf}_2 \qhS(\qhqr_1)} x) \rangle \\
    {}^{\eqref{eq:q-Hopf-def-UL}}
    & = \langle \xi, (\qhUL_1 a) \triangleright (\qhUL_2 x) \rangle
  \end{align*}
  for $\xi \in X^{\vee}$, $a \in A$ and $x \in X$. Hence Part (i) is proved. To prove Part (ii), we let $\delta: X \to X \otimes C$ be the coaction of $C$ on $X$ and set $\delta^{\natural} = \mathbbm{Q}_{X,C,X}(\delta)$. Then the coaction of $C$ on $X^{\vee}$ is computed as follows:
  \begin{align*}
    \xi
    & \mapsto \mathbbm{P}_{X^{\vee},C,X}(\delta^{\natural})(\xi)
    = \delta^{\natural}(\qhpr_1 \xi \otimes \qhpr_2 x_i) \otimes x^i \\
    & = \langle \qhpr_1 \xi, \qhql_1 (\qhpr_2 x_i)_{(0)} \rangle
      \, \qhql_2 (\qhpr_2 x_i)_{(1)} \otimes x^i \\
    & = \langle \xi, \qhano{\qhS(\qhpr_1) \qhql_1 \qhqr_{2(1)}} x_{i(0)} \rangle
      \, \qhano{\qhql_2 \qhpr_{2(2)}} x_{i(1)} \otimes x^i \\
    {}^{\eqref{eq:q-Hopf-pquv-1}}
    & = \langle \xi, \qhano{\qhS(\qhpl_2) \qhf_1} x_{i(0)} \rangle
      \, \qhano{\qhS(\qhpl_1) \qhf_2} x_{i(1)} \otimes x^i \\
    {}^{\eqref{eq:q-Hopf-def-VR}}
    & = \langle \xi, \qhVR_1 x_{i(0)} \rangle \, \qhVR_2 x_{i(1)} \otimes x^i
  \end{align*}
  for all $\xi \in X^{\vee}$. Thus the proof of Part (ii) is done.
\end{proof}

For a finite-dimensional left $H$-module $X$, we also define the left $H$-module ${}^{\vee} \! X$ to be the vector space ${}^{\vee} \! X = X^*$ with the left $H$-module structure given by $h \cdot \xi = \xi \leftharpoonup \overline{\qhS}(h)$ for $h \in H$ and $\xi \in {}^{\vee} \! X$. We fix a basis $\{ x_i \}$ of $X$ and define $\eval'_X$ and $\coev'_X$ by
\begin{gather*}
  \eval'_X: X \otimes {}^{\vee} \! X \to \unitobj,
  \quad \eval'_X(x \otimes \xi) = \langle \xi, \overline{\qhS}(\qhalpha) x \rangle, \\
  \coev'_X: \unitobj \to {}^{\vee} \! X \otimes X,
  \quad \coev'_X(1) = x^i \otimes \overline{\qhS}(\qhbeta) x_i
\end{gather*}
for $x \in X$ and $\xi \in {}^{\vee} \! X$, where $\{ x^i \}$ is the dual basis of $\{ x_i \}$. Then $({}^{\vee} \! X, \eval'_X, \coev'_X)$ is a right dual object of $X$ in ${}_H \Mod$. We remark that a right dual object of $X$ is a left dual object of $X \in ({}_H \Mod)^{\rev}$ and the monoidal category $({}_H \Mod)^{\rev}$ can be identified with ${}_{H^{\cop}}\Mod$. We now define
\begin{align}
  \label{eq:q-Hopf-def-UR}
  \qhUR & = \overline{\qhS}(\qhql_2 \overline{\qhf}_2) \otimes \overline{\qhS}(\qhql_1 \overline{\qhf}_1), \\
  \label{eq:q-Hopf-def-VL}
  \qhVL & = \overline{\qhS}(\qhf_2 \qhpr_2) \otimes \overline{\qhS}(\qhf_1 \qhpr_1).
\end{align}
If we denote by $\qhUL_{\cop}$ and $\qhVR_{\cop}$ the elements $\qhUL$ and $\qhVR$ for $H^{\cop}$, respectively, then we have
\begin{equation}
  \label{eq:q-Hopf-UL-VR-cop}
  \qhUL_{\cop} = \qhUR_{21}
  \quad \text{and} \quad
  \qhVR_{\cop} = \qhVL_{21}
\end{equation}
by equations~\eqref{eq:q-Hopf-pq-cop} and \eqref{eq:q-Hopf-f-op-cop}. By applying Lemma~\ref{lem:dual-module-left} to $H^{\cop}$, we obtain the following lemma:

\begin{lemma}
  \label{lem:dual-module-right}
  \textup{(i)} Let $A$ be a left $H$-module algebra. If $X$ is a finite-dimensional right $A$-module in ${}_H \Mod$ with action $\triangleleft$, then the left $A$-module structure of ${}^{\vee} \! X$ is given by
  \begin{equation}
    \label{eq:module-alg-dual-module-2}
    \langle a \triangleright \xi, x \rangle
    = \langle \xi, (\qhUR_1 x) \triangleleft (\qhUR_2 a) \rangle
    \quad (a \in A, \xi \in {}^{\vee} \! X, x \in X).
  \end{equation}
  \textup{(ii)} Let $C$ be a left $H$-module coalgebra. If $X$ is a finite-dimensional left $C$-comodule in ${}_H \Mod$ with coaction $x \mapsto x_{(-1)} \otimes x_{(0)}$, then the right $C$-comodule structure of ${}^{\vee} \! X$ is given by
  \begin{equation}
    \label{eq:module-alg-dual-comodule-2}
    \xi \mapsto x^i \otimes
    \qhVL_1 x_{i(-1)} \,
    \langle \xi, \qhVL_2 x_{i(0)} \rangle
    \quad (\xi \in {}^{\vee} \! X),
  \end{equation}
  where $\{ x_i \}$ is a basis of $X$ and $\{ x^i \}$ is the dual basis of $\{ x_i \}$.
\end{lemma}

\section{Integral theory for quasi-Hopf algebras}
\label{sec:integrals}

\subsection{Integrals}

The notions of integrals and cointegrals for Hopf algebras play an important role in the Hopf algebra theory and its applications. The integral theory for quasi-Hopf algebras is established by Hausser and Nill in \cite{1999math4164H}. In this section, following \cite{1999math4164H,MR1995128,MR2862216,MR3929714}, we review basic results on (co)integrals for quasi-Hopf algebras.

The definition of integrals in a quasi-Hopf algebra is completely same as the case of Hopf algebras:

\begin{definition}
  Let $H$ be a quasi-Hopf algebra. A {\em left integral} in $H$ is an element $\Lambda \in H$ such that $h \Lambda = \qhepsilon(h) \Lambda$ for all $h \in H$. Analogously, a {\em right integral} in $H$ is an element $\Lambda \in H$ such that $\Lambda h = \qhepsilon(h) \Lambda$ for all $h \in H$.
\end{definition}

Hausser and Nill \cite{1999math4164H} defined a cointegral on a quasi-Hopf algebra based on the theory of {\em quasi-Hopf bimodules}. It turned out that quasi-Hopf bimodules are not indispensable to define cointegrals \cite{MR1995128,MR2862216}. However, since the definition of Hausser and Nill will be convenient and important in later sections, we also review basic results on quasi-Hopf bimodules.

\subsection{Quasi-Hopf bimodules}

Let $H$ be a quasi-Hopf algebra with bijective antipode. The category ${}_H \Mod_H$ of $H$-bimodules is identified with the category of left modules over the enveloping algebra $H^{e} := H \otimes H^{\op}$. Since $H^e$ is naturally a quasi-Hopf algebra, the category ${}_H \Mod_H$ is a monoidal category. To be precise, for $M, N \in {}_H \Mod_H$, their tensor product $M \hatotimes N \in {}_H \Mod_H$ is the vector space $M \otimes N$ endowed with the $H$-bimodule structure given by
\begin{equation*}
  h \cdot (m \otimes n) \cdot h' = h_{(1)}^{} m h'_{(1)} \otimes h_{(2)}^{} n h'_{(2)}
\end{equation*}
for $h, h' \in H$, $m \in M$ and $n \in N$. The unit object $\unitobj$ is the base field $k$ regarded as an $H$-bimodule by the counit of $H$. The associator of $({}_H \Mod_H, \hatotimes, \unitobj)$ is given by
\begin{align*}
  \Phi_{L,M,N}: (L \hatotimes M) \hatotimes N
  & \to (L \hatotimes M) \hatotimes N, \\
  (a \otimes b) \otimes c
  & \mapsto \qhphi_1 a \overline{\qhphi}_1 \otimes (\qhphi_2 b \overline{\qhphi}_2 \otimes \qhphi_3 c \overline{\qhphi}_3)
\end{align*}
for $L, M, N \in {}_H \Mod_H$, $a \in L$, $b \in M$ and $c \in N$. The left and the right unit isomorphisms of ${}_H \Mod_H$ are same as those of the monoidal category of vector spaces over $k$.

Equations~\eqref{eq:q-Hopf-def-1} and \eqref{eq:q-Hopf-def-2} imply that the $H$-bimodule $H$ is a coalgebra in the monoidal category $({}_H \Mod_H, \hatotimes, \unitobj)$ with respect to $\Delta$ and $\qhepsilon$. The category ${}_H^{} \Mod_H^H$ of {\em right quasi-Hopf bimodules} over $H$ is defined as the category of right $H$-comodules in ${}_H \Mod_H$ in the sense of Subsection \ref{subsec:modules-over-alg}. Given a left $H$-module $V$, we regard it as an $H$-bimodule by defining the right action of $H$ through the counit of $H$. Then the $H$-bimodule $V \hatotimes H$ is a right $H$-comodule in ${}_H \Mod_H$ by the coaction
\begin{equation}
  \label{eq:q-Hopf-bimod-free-coaction}
  V \hatotimes H
  \xrightarrow{\quad \id_{V} \hatotimes \Delta \quad}
  V \hatotimes (H \hatotimes H)
  \xrightarrow{\quad \Phi_{V,H,H}^{-1} \quad}
  (V \hatotimes H) \hatotimes H.
\end{equation}
Thus we have a functor
\begin{equation}
  \label{eq:q-Hopf-bimod-equiv-1}
  {}_H^{} \Mod \to {}_H^{} \Mod_H^H,
  \quad V \mapsto V \hatotimes H.
\end{equation}
For $M \in {}_H^{} \Mod_H^H$, we define the linear map $E_M: M \to M$ by
\begin{equation*}
  E_M(m) = \qhqr_1 m_{(0)} \qhbeta \qhS(\qhqr_2 m_{(1)}) \quad (m \in M),
\end{equation*}
where $m \mapsto m_{(0)} \otimes m_{(1)}$ is the coaction. The subspace $M^{\coinv{H}} := \Img(E_M)$ is called the space of {\em coinvariants}. Although $M^{\coinv H}$ is not a submodule of $M$, it is a left $H$-module by the action $\triangleright$ given by $h \triangleright m = E_M(h m)$ for $h \in H$ and $m \in M^{\coinv H}$. This construction extends to the functor
\begin{equation}
  \label{eq:q-Hopf-bimod-equiv-2}
  {}_H^{} \Mod_H^H \to {}_H^{} \Mod,
  \quad M \mapsto M^{\coinv H}.
\end{equation}
The fundamental theorem for quasi-Hopf bimodules \cite[Section 3]{1999math4164H} states that the functors \eqref{eq:q-Hopf-bimod-equiv-1} and \eqref{eq:q-Hopf-bimod-equiv-2} are mutually quasi-inverse to each other. Furthermore, they form an adjoint equivalence between ${}_H\Mod$ and ${}_H^{} \Mod_H^H$ with the unit and the counit given respectively by
\begin{gather}
  \label{eq:q-Hopf-bimod-FT-iso-1}
  V \to (V \hatotimes H)^{\coinv{H}}, \quad v \mapsto v \otimes 1 \quad (v \in V), \\
  \label{eq:q-Hopf-bimod-FT-iso-2}
  M^{\coinv{H}} \hatotimes H \to M, \quad m \otimes h \mapsto m h \quad (m \in M^{\coinv{H}}, h \in H).
\end{gather}

The category ${}_H^H \Mod_H^{}$ of {\em left quasi-Hopf bimodules} over $H$ is defined to be the category of left $H$-comodules in ${}_H \Mod_H$. This category can be identified with the category of right quasi-Hopf modules over $K = H^{\cop}$. Thus, by applying the fundamental theorem to $K$, we see that the functor
\begin{equation}
  \label{eq:q-Hopf-bimod-equiv-3}
  {}_H \Mod \to {}_H^H \Mod_H^{},
  \quad V \mapsto H \hatotimes V
\end{equation}
is an equivalence of categories. A quasi-inverse of~\eqref{eq:q-Hopf-bimod-equiv-3}, which we denote by
\begin{equation}
  \label{eq:q-Hopf-bimod-equiv-4}
  {}^{\coinv H}(-): {}_H^{H} \Mod_H^{} \to {}_H^{} \Mod,
\end{equation}
is given as follows: For $M \in {}^H_H \Mod_H^{}$, we define the linear map $E_M^{\cop}: M \to M$ by
\begin{equation*}
  E_M^{\cop}(m) = \qhql_2 m_{(0)} \overline{\qhS}(\qhbeta) \overline{\qhS}(\qhql_1 m_{(-1)})
  \quad (m \in M),
\end{equation*}
where $m \mapsto m_{(-1)} \otimes m_{(0)}$ is the coaction. Then ${}^{\coinv H} \! M := \mathrm{Im}(E_M^{\cop})$ as a vector space. The left $H$-module structure is given by $h \triangleright m = E_M^{\cop}(h m)$. The functors~\eqref{eq:q-Hopf-bimod-equiv-3} and \eqref{eq:q-Hopf-bimod-equiv-4} actually form an adjoint equivalence with the unit and the counit given respectively by
\begin{gather}
  V \to {}^{\coinv H}(H \hatotimes V),
  \quad v \mapsto 1 \otimes v
  \quad (v \in V), \\
  H \hatotimes {}^{\coinv H} \! M \to M,
  \quad h \otimes m \mapsto m h
  \quad (m \in {}^{\coinv H} \! M, h \in H).
\end{gather}

\subsection{Cointegrals}

Let $H$ be a quasi-Hopf algebra with bijective antipode. Since ${}_{H} \Mod_H$ is isomorphic to the category of left $H^{e}$-modules as a monoidal category, every finite-dimensional object of ${}_{H} \Mod_H$ admits a left dual object and a right dual object. Specifically, if $M$ is a finite-dimensional $H$-bimodule, then its left dual object $(M^{\vee}, \eval_M, \coev_M)$ is given as follows: As a vector space, $M^{\vee} = M^*$. We fix a basis $\{ m_i \}$ of $M$ and let $\{ m^i \}$ be the dual basis of $\{ m_i \}$. Then the $H$-bimodule structure of $M^{\vee}$, the evaluation morphism and the coevaluation morphism are given by
\begin{gather}
  \label{eq:H-bimod-left-dual-actions}
  h \cdot \xi \cdot h' = \overline{\qhS}(h') \rightharpoonup \xi \leftharpoonup \qhS(h), \\
  \label{eq:H-bimod-left-dual-eval}
  \eval_M: M^{\vee} \hatotimes M \to k, \quad
  \eval_M(\xi \otimes m) = \langle \xi, \upalpha m \overline{\qhS}(\qhbeta) \rangle, \\
  \label{eq:H-bimod-left-dual-coev}
  \coev_M: k \to M \hatotimes M^{\vee}, \quad
  \coev_M(1) = \qhbeta m_i \overline{\qhS}(\qhalpha) \otimes m^i,
\end{gather}
respectively, for $h, h' \in H$, $m \in M$ and $\xi \in M^{\vee}$, where the Einstein notation is used to suppress the sum over $i$. Analogously, the right dual object ${}^{\vee} \! M$ of $M$ is defined by ${}^{\vee} \! M = M^*$ as a vector space. The $H$-bimodule structure, the evaluation and the coevaluation are given by
\begin{gather}
  \label{eq:H-bimod-right-dual-actions}
  h \cdot \xi \cdot h' = \qhS(h') \rightharpoonup \xi \leftharpoonup \overline{\qhS}(h), \\
  \label{eq:H-bimod-right-dual-eval}
  \eval'_M: M \hatotimes {}^{\vee} \! M \to k, \quad
  \eval'_M(m \otimes \xi) = \langle \xi, \overline{\qhS}(\qhalpha) m \qhbeta \rangle, \\
  \label{eq:H-bimod-right-dual-coev}
  \coev'_M: k \to {}^{\vee} \! M \hatotimes M, \quad
  \coev'_M(1) = m^i \otimes \overline{\qhS}(\qhbeta) m_i \qhalpha,
\end{gather}
respectively, for $h, h' \in H$, $m \in M$ and $\xi \in {}^{\vee} \! M$.

\begin{lemma}
  \label{lem:q-Hopf-bimod-dual}
  \textup{(i)}   If $M$ is a finite-dimensional right quasi-Hopf bimodule over $H$, then the $H$-bimodule $M^{\vee}$ is a left quasi-Hopf bimodule by the coaction
  \begin{equation}
    \label{eq:H-bimod-left-dual-coactions}
    \xi \mapsto \langle \xi, \qhVR_1 m_{i(0)} \qhUR_1 \rangle \, \qhVR_2 m_{i(1)} \qhUR_2 \otimes m^i
    \quad (\xi \in M^{\vee}),
  \end{equation}
  where $\{ m_i \}$ is a basis of $M$, $\{ m^i \}$ is the dual basis of $M^*$.

  \textup{(ii)} If $M$ is a finite-dimensional left quasi-Hopf bimodule over $H$, then the $H$-bimodule ${}^{\vee} \! M$ is a right quasi-Hopf bimodule by the coaction
  \begin{equation}
    \label{eq:H-bimod-right-dual-coactions}
    \xi \mapsto m^i \otimes \qhVL_1 m_{i(-1)} \qhUL_1 \, \langle \xi, \qhVL_2 m_{i(0)} \qhUL_2 \rangle
    \quad (\xi \in {}^{\vee} \! M),
  \end{equation}
  where $\{ m_i \}$ is a basis of $M$, $\{ m^i \}$ is the dual basis of $M^*$.
\end{lemma}

See \eqref{eq:q-Hopf-def-UL}, \eqref{eq:q-Hopf-def-VR}, \eqref{eq:q-Hopf-def-UR} and \eqref{eq:q-Hopf-def-VL} for the definitions of the elements $\qhUL$, $\qhVR$, $\qhUR$ and $\qhVL$ of $H \otimes H$, respectively.

\begin{proof}
  The result follows from Lemmas \ref{lem:dual-module-left} and~\ref{lem:dual-module-right} applied to the coalgebra $C = H$ in ${}_H \Mod_H$ ($\cong {}_{H^e}\Mod$). One can also verify this lemma by a direct computation along the same way as \cite[Proposition 3.2]{MR2862216}.
\end{proof}

Now we suppose that $H$ is finite-dimensional. Then $H$ is a finite-dimensional left and right quasi-Hopf bimodule over $H$. Thus $H^{\vee}$ and ${}^{\vee}H$ are a left and a right quasi-Hopf bimodule over $H$, respectively.

\begin{definition}
  \label{def:cointegrals}
  We set $\cointl := ({}^{\vee} \! H)^{\coinv H}$ and $\cointr := {}^{\coinv H} (H^{\vee})$ and call them the spaces of {\em left cointegrals} and {\em right cointegrals} on $H$, respectively.
\end{definition}

The fundamental theorem gives an isomorphism ${}^{\vee} \! H \cong \cointl \hatotimes H$ in ${}_H^{} \Mod_{H}^{H}$. By counting dimensions, we see that $\cointl$ is one-dimensional. Thus there is an algebra map $\qhmu: H \to k$ such that $h \triangleright \lambda = \qhmu(h) \lambda$ for all $h \in H$ and $\lambda \in \cointl$.

\begin{definition}
  We call $\qhmu$ the {\em modular function} on $H$. We say that $H$ is {\em unimodular} if the modular function on $H$ is identical to the counit of $H$.
\end{definition}

By the fundamental theorem, there is also an isomorphism $H^{\vee} \cong H \hatotimes \cointr$ in ${}_H^{H} \Mod_H^{}$. Hence there is an algebra map $\qhmu': H \to k$ such that $h \triangleright \lambda = \qhmu'(h) \lambda$ for all $h \in H$ and $\lambda \in \cointr$. The following lemma shows that $\qhmu'$ coincides with the modular function.

\begin{lemma}
  $\cointl \cong \cointr$ as left $H$-modules.
\end{lemma}
\begin{proof}
  We have isomorphisms
  \begin{equation*}
    ({}^{\vee} \! \cointr \hatotimes \cointl) \hatotimes H
    \cong {}^{\vee} \! \cointr \hatotimes (\cointl \hatotimes H)
    \cong {}^{\vee} \! \cointr \hatotimes {}^{\vee} \! H
    \cong {}^{\vee} (H \hatotimes \cointr)
    \cong {}^{\vee} (H^{\vee})
    \cong H
  \end{equation*}
  of right quasi-Hopf bimodules. By the fundamental theorem, we have an isomorphism ${}^{\vee} \cointr \otimes \cointl \cong \unitobj$ of left $H$-modules. Thus $\cointl \cong \cointr$ as left $H$-modules.
\end{proof}

\subsection{Properties of cointegrals}

Let $H$ be a finite-dimensional quasi-Hopf algebra. We note that the antipode of $H$ is bijective if this is the case \cite{MR1995128,MR2037710}. Let $\Xi$ be the map \eqref{eq:q-Hopf-bimod-FT-iso-2} for the quasi-Hopf bimodule $M = {}^{\vee} \! H$. By \eqref{eq:H-bimod-right-dual-actions}, the map $\Xi$ is expressed as follows:
\begin{equation}
  \label{eq:Xi}
  \Xi: \cointl \hatotimes H \to {}^{\vee} \! H,
  \quad \lambda \otimes h \mapsto \qhS(h) \rightharpoonup \lambda
  \quad (\lambda \in \cointl, h \in H).
\end{equation}
The fundamental theorem for quasi-Hopf bimodules implies that this map is an isomorphism in the category ${}_H^{} \Mod^H_H$. In this subsection, we review some results on (co)integrals with an emphasis on the role of the map $\Xi$.

We fix a non-zero left cointegral $\lambda$ on $H$. Since the map $\Xi$ is bijective, and since the antipode $\qhS$ of $H$ is bijective, the map $H \to H^*$ given by $h \mapsto h \rightharpoonup \lambda$ ($h \in H$) is also bijective. This means that the algebra $H$ is a Frobenius algebra with Frobenius form $\lambda$.
We recall that the {\em Nakayama automorphism} of $H$ (with respect to $\lambda$) is the algebra automorphism $\nu: H \to H$ characterized by
\begin{equation}
  \label{eq:def-Nakayama-auto}
  \lambda(h h') = \lambda(h' \nu(h)) \quad (h, h' \in H). 
\end{equation}
The following theorem is due to Hausser and Nill \cite{1999math4164H}.

\begin{theorem}
  \label{thm:left-coint}
  Let $\lambda$ be a non-zero left cointegral on $H$. Then we have:
  \begin{enumerate}
  \item $H$ is a Frobenius algebra with Frobenius form $\lambda$. The Nakayama automorphism of $H$ with respect to $\lambda$ is given by
    \begin{equation*}
      \nu(h) = \qhS(\qhS(h) \leftharpoonup \qhmu)
      \quad (h \in H),
    \end{equation*}
    where $\qhmu$ is the modular function on $H$ and
    \begin{equation*}
      h \leftharpoonup \xi = \langle \xi, h_{(1)} \rangle \, h_{(2)} \quad (\xi \in H^*, h \in H).
    \end{equation*}
  \item The space of left integrals in $H$ is one-dimensional. If $\Lambda$ is a non-zero left integral in $H$, then we have $\lambda(\Lambda) \ne 0$ and $\Lambda h = \qhmu(h) \Lambda$ for all $h \in H$.
  \item The space of right integrals in $H$ is also one-dimensional. If $\Lambda$ is a non-zero right integral in $H$, then we have $\lambda(\Lambda) \ne 0$ and $h \Lambda = \overline{\qhmu}(h) \Lambda$ for all $h \in H$, where $\overline{\qhmu} = \qhmu \circ \qhS$.
  \end{enumerate}
\end{theorem}

We give some remarks on the role of the isomorphism $\Xi$ in the proof of this theorem:
Let $\lambda$ be a non-zero left cointegral on $H$. By the argument before the above theorem, we may say that the bijectivity of the map $\Xi$ given by~\eqref{eq:Xi} implies that $H$ is a Frobenius algebra with Frobenius form $\lambda$. The $H$-linearity of $\Xi$ implies Part (1) of the above theorem. Indeed, we have
\begin{gather*}
    \lambda(\overline{\qhS}(h) \qhS(h'))
    = (h \cdot \Xi(\lambda \otimes h'))(1)
    = \Xi(h \cdot (\lambda \otimes h'))(1) \\
    = \Xi(h_{(1)} \triangleright \lambda \otimes h_{(2)} h')(1)
    = \qhmu(h_{(1)}) \lambda(\qhS(h_{(2)} h'))
    = \lambda(\qhS(h') \qhS(h \leftharpoonup \qhmu))
\end{gather*}
for all $h, h' \in H$. If we replace $h$ and $h'$ with $\qhS(h)$ and $\overline{\qhS}(h')$, respectively, then we obtain $\lambda(h h') = \lambda(h' \qhS(\qhS(h) \leftharpoonup \qhmu))$ as desired.

According to Hausser and Nill \cite{1999math4164H}, Parts (2) and (3) of the above theorem follow from basic results on Frobenius algebras. We point out that the $H$-colinearity of $\Xi$ is not mentioned so far. This property of $\Xi$ actually implies:

\begin{lemma}
  \label{lem:left-coint-characterization-1}
  Let $\lambda$ be a left cointegral on $H$. Then the equation
  \begin{equation}
    \label{eq:left-cointegral-0}
    \qhVL_1 h_{(1)} \qhUL_1 \, \langle \lambda, \qhVL_2 h_{(2)} \qhUL_2 \qhS(h') \rangle
    = \qhmu(\overline{\qhphi}_{1})
    \langle \lambda, h \qhS(\overline{\qhphi}_2 h'_{(1)}) \rangle
    \, \overline{\qhphi}_3 h'_{(2)}
  \end{equation}
  holds for all $h, h' \in H$.
\end{lemma}
\begin{proof}
  We fix a left cointegral $\lambda$ on $H$ and elements $h, h' \in H$. Let $\delta$ and $\delta'$ be the right coactions of $H$ on $\cointl \hatotimes H$ and ${}^{\vee} \! H$, respectively. Since the map $\Xi$ preserves the coactions, we have $F = G$, where
  \begin{equation*}
    F = ((\Xi \otimes \id_H) \circ \delta)(\lambda \otimes h')
    \quad \text{and} \quad
    G = (\delta' \circ \Xi)(\lambda \otimes h').
  \end{equation*}
  We recall that $\delta$ is given by \eqref{eq:q-Hopf-bimod-free-coaction} with $V = \cointl$, and the $H$-bimodule structure of $\cointl$ is given by $x \lambda y = \qhmu(x) \qhepsilon(y) \lambda$ for $x, y \in H$. Thus we have
  \begin{align*}
    \delta(\lambda \otimes h')
    = (\overline{\qhphi}_1 \lambda \qhphi_1 \otimes \overline{\qhphi}_2 h'_{(1)} \qhphi_2)
    \otimes \overline{\qhphi}_3 h'_{(2)} \qhphi_3
    = (\qhmu(\overline{\qhphi}_1) \lambda \otimes \overline{\qhphi}_2 h'_{(1)})
    \otimes \overline{\qhphi}_3 h'_{(2)}.
  \end{align*}
  Now we define a linear map $T_h: H^* \otimes H \to H$ by $T_h(\xi \otimes a) = \xi(h) a$. Then,
  \begin{equation*}
    T_h(F)
    = \Xi(\qhmu(\overline{\qhphi}_1) \lambda \otimes \overline{\qhphi}_2 h'_{(1)})(h)
    \cdot \overline{\qhphi}_3 h'_{(2)}
    = \qhmu(\overline{\qhphi}_{1})
    \langle \lambda, h \qhS(\overline{\qhphi}_2 h'_{(1)}) \rangle
    \, \overline{\qhphi}_3 h'_{(2)}.
  \end{equation*}
  We fix a basis $\{ e_i \}$ of $H$ and let $\{ e^i \}$ be the dual basis of $\{ e_i \}$. Then we have
  \begin{align*}
    T_h(G) = T_h(\delta'(\qhS(h') \rightharpoonup \lambda))
    & = \langle e^i, h \rangle \qhVL_1 e_{i(1)} \qhUL_1 \, \langle \qhS(h') \rightharpoonup \lambda,
    \qhVL_2 e_{i(2)} \qhUL_2 \rangle \\
    & = \qhVL_1 h_{(1)} \qhUL_1 \, \langle \lambda, \qhVL_2 h_{(2)} \qhUL_2 \qhS(h') \rangle
  \end{align*}
  by Lemma~\ref{lem:q-Hopf-bimod-dual}. Now equation~\eqref{eq:left-cointegral-0} follows from $T_h(F) = T_h(G)$.
\end{proof}

\subsection{Characterizations of cointegrals}

If $\lambda \in H^*$ is a left cointegral on $H$, then we have
\begin{equation}
  \label{eq:left-cointegral-1}
  \qhVL_1 h_{(1)} \qhUL_1 \, \langle \lambda, \qhVL_2 h_{(2)} \qhUL_2\rangle
  = \qhmu(\overline{\qhphi}_{1})
  \langle \lambda, h \qhS(\overline{\qhphi}_2) \rangle
  \, \overline{\qhphi}_3
\end{equation}
for all $h \in H$ by~\eqref{eq:left-cointegral-0}. The first appearance of \eqref{eq:left-cointegral-1} seems to be the proof of \cite[Proposition 3.4]{MR1995128}. Later, equation~\eqref{eq:left-cointegral-1} has been used as a definition of left cointegrals on $H$ in some literature including \cite{MR2086073,MR2363502}. However, it is not trivial that equation \eqref{eq:left-cointegral-1} characterizes left cointegrals on $H$. One can prove such a characterization by putting some arguments of \cite{1999math4164H,MR1995128,MR2862216} together. For reader's convenience, we here give a self-contained proof of the fact that \eqref{eq:left-cointegral-1} characterizes left cointegrals on $H$ and give some more characterizations of cointegrals. We first begin with the following technical lemma:

\begin{lemma}
  \label{lem:left-coint-characterization-2}
  If $\lambda \in H^*$ satisfies \eqref{eq:left-cointegral-1}, then the equation
  \begin{equation}
    \label{eq:left-cointegral-2}
    \qhqr_1 h_{(1)} \qhpr_1 \langle \lambda, \qhqr_2 h_{(2)} \qhpr_2 \rangle
    = \qhmu(\overline{\qhphi}_{1}) \langle \lambda,
    \overline{\qhS}(\qhql_1) h \qhS(\overline{\qhphi}_2 \qhpl_1) \rangle
    \qhql_2 \overline{\qhphi}_3 \qhpl_2
  \end{equation}
  holds for all $h \in H$.
\end{lemma}
\begin{proof}
  Bulacu and Caenepeel \cite[Lemma 3.6]{MR2862216} showed that a left cointegral on $H$ satisfies~\eqref{eq:left-cointegral-2}, but their proof actually shows that \eqref{eq:left-cointegral-1} implies \eqref{eq:left-cointegral-2}. For the sake of completeness, we give a detailed proof. We first recall from \cite{1999math4164H} that the following equations hold:
  \begin{equation}
    \label{eq:HN-Eq-7.2,7.3}
    (\qhql_2 \otimes 1) \qhVL \Delta(\overline{\qhS}(\qhql_1)) = \qhqr,
    \quad
    \Delta(\qhS(\qhpl_1)) \qhUL (\qhpl_2 \otimes 1) = \qhpr.
  \end{equation}
  The first one is proved as follows:
  Set $\check{\qhf} = \overline{\qhS}(\qhf_2) \otimes \overline{\qhS}(\qhf_1)$.
  We have $\qhf_{\cop} = \check{\qhf}_{21}$ with notation of Lemma \ref{lem:q-Hopf-f-op-cop}. Thus, by applying \eqref{eq:q-Hopf-f-1} to $H^{\cop}$, we have
  \begin{equation}
    \label{eq:q-Hopf-f-1-cop}
    \check{\qhf} \cdot \Delta(\overline{\qhS}(h)) = (\overline{\qhS}(h_{(2)}) \otimes \overline{\qhS}(h_{(1)})) \cdot \check{\qhf}
  \end{equation}
  for all $h \in H$. We now compute:
  \begin{align*}
    & (\qhql_2 \otimes 1) \qhano{\qhVL} \Delta(\overline{\qhS}(\qhql_1)) \\
    {}^{\eqref{eq:q-Hopf-def-VL}}
    & = (\qhql_2 \otimes 1) \cdot (\overline{\qhS}(\qhpr_2) \otimes \overline{\qhS}(\qhpr_1))
      \cdot \qhano{\check{\qhf} \cdot \Delta(\overline{\qhS}(\qhql_1))} \\
    {}^{\eqref{eq:q-Hopf-f-1-cop}}
    & = (\qhql_2 \otimes 1) \cdot (\overline{\qhS}(\qhpr_2) \otimes \overline{\qhS}(\qhpr_1))
      \cdot (\overline{\qhS}(\qhql_{1(2)}) \otimes \overline{\qhS}(\qhql_{1(1)}))
      \cdot \check{\qhf} \\
    & = (\overline{\qhS} \otimes \overline{\qhS})
      (\qhano{\qhql_{1(2)} \qhpr_2 \qhS(\qhql_2)}
      \otimes \qhano{\qhql_{1(1)} \qhpr_1})
      \cdot \check{\qhf} \\
    {}^{\eqref{eq:q-Hopf-pquv-4}}
    & = (\overline{\qhS} \otimes \overline{\qhS})
      (\overline{\qhf}_2 \qhS(\qhqr_1) \otimes \overline{\qhf}_1 \qhS(\qhqr_2))
      \cdot \check{\qhf}
      = \qhqr.
  \end{align*}
  The second one is proved in a similar way. Now we suppose that $\lambda \in H^*$ satisfies~\eqref{eq:left-cointegral-1}. Then, for all $h \in H$, we have
  \begin{align*}
    & \qhmu(\overline{\qhphi}_{1}) \langle \lambda,
    \overline{\qhS}(\qhql_1) h \qhano{\qhS(\overline{\qhphi}_2 \qhpl_1)} \rangle
      \qhql_2 \overline{\qhphi}_3 \qhpl_2 \\
    & = \qhql_2 \cdot \qhano{\qhmu(\overline{\qhphi}_{1}) \langle \lambda,
      \overline{\qhS}(\qhql_1) h \qhS(\qhpl_1) \cdot \qhS(\overline{\qhphi}_2) \rangle
      \overline{\qhphi}_3} \cdot \qhpl_2 \\
    {}^{\eqref{eq:left-cointegral-1}}
    & = \qhql_2 \cdot \qhVL_1 (\overline{\qhS}(\qhql_1) h \qhS(\qhpl_1))_{(1)} \qhUL_1
      \, \langle \lambda, \qhVL_2 (\overline{\qhS}(\qhql_1) h \qhS(\qhpl_1))_{(2)} \qhUL_2 \rangle \cdot  \qhpl_2 \\
    & = \qhano{\qhql_2 \qhVL_1 \overline{\qhS}(\qhql_1)_{(1)}}
      h_{(1)} \qhano{\qhS(\qhpl_1)_{(1)} \qhUL_1 \qhpl_2}
      \, \langle \lambda, \qhano{\qhVL_2 \overline{\qhS}(\qhql_1)_{(2)}}
      h_{(2)} \qhano{\qhS(\qhpl_1)_{(2)} \qhUL_2} \rangle \\
    {}^{\eqref{eq:HN-Eq-7.2,7.3}}
    & = \qhqr_1 h_{(1)} \qhpr_1 \langle \lambda, \qhqr_2 h_{(2)} \qhpr_2 \rangle. \qedhere
  \end{align*}
\end{proof}

\begin{lemma}
  \label{lem:left-coint-characterization-3}
  Suppose that $\lambda \in H^*$ satisfies~\eqref{eq:left-cointegral-2}. Then we have
  \begin{equation}
    \label{eq:left-cointegral-4}
    \qhqr_1 \Lambda_{(1)} \qhpr_1 \langle \lambda, \qhqr_2 \Lambda_{(2)} \qhpr_2 \rangle
    = \qhmu(\qhbeta) \lambda(\Lambda) 1
  \end{equation}
  for all left integrals $\Lambda$ in $H$.
\end{lemma}
\begin{proof}
  If $\Lambda$ is a left integral in $H$, then we have
  \begin{align*}
    & \qhqr_1 \Lambda_{(1)} \qhpr_1 \langle \lambda, \qhqr_2 \Lambda_{(2)} \qhpr_2 \rangle \\
    {}^{\eqref{eq:left-cointegral-2}}
    & = \qhmu(\overline{\qhphi}_{1}) \langle \lambda,
      \qhano{\overline{\qhS}(\qhql_1) \Lambda \qhS(\overline{\qhphi}_2 \qhpl_1)} \rangle
      \qhql_2 \overline{\qhphi}_3 \qhpl_2 \\
    {}^{\text{(Theorem \ref{thm:left-coint})}}
    & = \qhmu(\overline{\qhphi}_{1}) \qhmu(\qhS(\overline{\qhphi}_2 \qhano{\qhpl_1}))
      \qhepsilon(\overline{\qhS}(\qhano{\qhql_1}))
      \langle \lambda, \Lambda \rangle
      \qhano{\qhql_2} \overline{\qhphi}_3 \qhano{\qhpl_2} \\
    {}^{\eqref{eq:q-Hopf-def-pR}, \eqref{eq:q-Hopf-def-pL}}
    & = \qhmu(\overline{\qhphi}_{1}) \qhmu(\qhphi_1 \qhbeta \qhS(\qhphi_2) \qhS(\overline{\qhphi}_2))
      \qhepsilon(\qhS(\overline{\qhphi}{}'_1) \qhalpha \overline{\qhphi}{}'_2)
      \langle \lambda, \Lambda \rangle
      \overline{\qhphi}{}'_3 \overline{\qhphi}_3 \qhphi_3 \\
    {}^{\eqref{eq:q-Hopf-def-4}, \eqref{eq:q-Hopf-def-7}}
    & = \qhmu(\overline{\qhphi}_{1} \qhphi_1) \qhmu(\qhbeta) \qhmu(\qhS(\overline{\qhphi}_2 \qhphi_{2}))
      \langle \lambda, \Lambda \rangle \overline{\qhphi}_3 \qhphi_3
    = \qhmu(\qhbeta) \langle \lambda, \Lambda \rangle. \qedhere
  \end{align*}
\end{proof}

We note that, for any algebra map $\gamma: H \to k$, both $\gamma(\qhalpha)$ and $\gamma(\qhbeta)$ are invertible elements of $k$. Indeed, by applying $\gamma$ to \eqref{eq:q-Hopf-def-6}, we obtain
\begin{equation*}
  \gamma(\qhalpha) \cdot \gamma(\qhbeta) \cdot \gamma(\qhphi_{1} \qhS(\qhphi_{2}) \qhphi_3) = 1.
\end{equation*}
Let $\lambda$ be a non-zero left cointegral on $H$. Then $H$ is a Frobenius algebra with Frobenius form $\lambda$. In other words, the map $H \to H^*$ given by $h \mapsto h \rightharpoonup \lambda$ ($h \in H$) is bijective. The inverse of this map is given explicitly as follows:

\begin{theorem}
  \label{thm:Frobenius-dual-basis}
  We fix a non-zero left cointegral $\lambda$ on $H$. Then the linear maps
  \begin{equation*}
    \Theta_{\qhL}, \Theta_{\qhR}: H \to H^*,
    \quad 
    \Theta_{\qhL}(h) = h \rightharpoonup \lambda,
    \quad
    \Theta_{\qhR}(h) = \lambda \leftharpoonup h
    \quad (h \in H)
  \end{equation*}
  are bijective. Let $\Lambda$ be a left integral in $H$ such that $\langle \lambda, \Lambda \rangle = \qhmu(\qhbeta)^{-1}$, and set
  \begin{equation}
    \label{eq:def-Lambda-tilde}
    \widetilde{\Lambda} = \qhqr_1 \Lambda_{(1)} \qhpr_1 \otimes \qhqr_2 \Lambda_{(2)} \qhpr_2.
  \end{equation}
  Then the inverses of $\Theta_{\qhL}$ and $\Theta_{\qhR}$ are given respectively by
  \begin{equation*}
    \Theta_{\qhL}^{-1}(\xi) = \qhS( \widetilde{\Lambda}_1 \leftharpoonup \qhmu) \, \langle \xi, \widetilde{\Lambda}_2 \rangle
    \quad \text{and} \quad
    \Theta_{\qhR}^{-1}(\xi) = \overline{\qhS}(\widetilde{\Lambda}_1) \, \langle \xi, \widetilde{\Lambda}_2 \rangle
    \quad (\xi \in H^*).
  \end{equation*}
\end{theorem}
\begin{proof}
  Since $H$ is a Frobenius algebra with Frobenius form $\lambda$, the maps $\Theta_{\qhL}$ and $\Theta_{\qhR}$ are bijective. Let $\nu$ be the Nakayama automorphism of $H$ with respect to $\lambda$. The defining formula~\eqref{eq:def-Nakayama-auto} of $\nu$ implies the equation $\Theta_{\qhR} = \Theta_{\qhL} \circ \nu$. Hence,
  \begin{equation}
    \label{eq:Frobenius-dual-basis-proof-1}
    \Theta_{\qhL}^{-1} = \nu \circ \Theta_{\qhR}^{-1}.
  \end{equation}
  To prove the formula for $\Theta_{\qhR}^{-1}$, we note:
  \begin{equation}
    \label{eq:Lambda-tilde-H}
    h \widetilde{\Lambda}_1 \otimes \widetilde{\Lambda}_2
    = \widetilde{\Lambda}_1 \otimes  \overline{\qhS}(h) \widetilde{\Lambda}_2
    \quad (h \in H).
  \end{equation}
  Indeed, for all $h \in H$, we have
  \begin{align*}
    h \widetilde{\Lambda}_1 \otimes \widetilde{\Lambda}_2
    & = \qhano{h \qhqr_1} \Lambda_{(1)} \qhpr_1 \otimes \qhano{\qhqr_2} \Lambda_{(2)} \qhpr_2 \\
    {}^{\eqref{eq:q-Hopf-qR}}
    & = \qhqr_1 h_{(1,1)} \Lambda_{(1)} \qhpr_1 \otimes \overline{\qhS}(h_{(2)}) \qhqr_2 h_{(1,2)} \Lambda_{(2)} \qhpr_2 \\
    {}^{(x \Lambda = \qhepsilon(x) \Lambda)}
    & = \qhepsilon(h_{(1)}) \qhqr_1 \Lambda_{(1)} \qhpr_1 \otimes \overline{\qhS}(h_{(2)}) \qhqr_2 \Lambda_{(2)} \qhpr_2 
      = \widetilde{\Lambda}_1 \otimes  \overline{\qhS}(h) \widetilde{\Lambda}_2.
  \end{align*}
  Set $\overline{\Theta}(\xi) = \overline{\qhS}(\widetilde{\Lambda}_1) \langle \xi, \widetilde{\Lambda}_2 \rangle$. By Lemmas~\ref{lem:left-coint-characterization-2} and~\ref{lem:left-coint-characterization-3}, equation \eqref{eq:left-cointegral-4} holds. Hence,
  \begin{gather*}
    \overline{\Theta} \Theta_{\qhR}(h)
    = \overline{\qhS}(\widetilde{\Lambda}_1) \langle \lambda, h \widetilde{\Lambda}_2 \rangle
    = \overline{\qhS}(\qhS(h) \widetilde{\Lambda}_1) \langle \lambda, \widetilde{\Lambda}_2 \rangle \\
    = \overline{\qhS}(\widetilde{\Lambda}_1) h \langle \lambda, \widetilde{\Lambda}_2 \rangle
    \stackrel{\eqref{eq:left-cointegral-4}}{=} h \overline{\qhS}(1) = h.
  \end{gather*}
  This implies $\Theta_{\qhR}^{-1} = \overline{\Theta}$. The expression for $\Theta_{\qhL}^{-1}$ is obtained by \eqref{eq:Frobenius-dual-basis-proof-1} and the explicit description of the Nakayama automorphism $\nu$ given in Theorem~\ref{thm:left-coint}.
\end{proof}

Now we give the following characterizations of cointegrals:

\begin{theorem}
  \label{thm:left-coint-characterization}
  For $\lambda \in H^*$, the following are equivalent:
  \begin{enumerate}
  \item $\lambda$ is a left cointegral on $H$.
  \item For all $h \in H$, the following equation holds:
    \begin{equation}
      \tag{$=$ \eqref{eq:left-cointegral-1}}
      \qhVL_1 h_{(1)} \qhUL_1
      \, \langle \lambda, \qhVL_2 h_{(2)} \qhUL_2 \rangle
      = \qhmu(\overline{\qhphi}_{1})
      \langle \lambda, h \qhS(\overline{\qhphi}_2) \rangle
      \, \overline{\qhphi}_3.
    \end{equation}
  \item For all $h \in H$, the following equation holds:
    \begin{equation}
      \tag{$=$ \eqref{eq:left-cointegral-2}}
      \qhqr_1 h_{(1)} \qhpr_1 \langle \lambda, \qhqr_2 h_{(2)} \qhpr_2 \rangle
      = \qhmu(\overline{\qhphi}_{1}) \langle \lambda,
      \overline{\qhS}(\qhql_1) h \qhS(\overline{\qhphi}_2 \qhpl_1) \rangle
      \qhql_2 \overline{\qhphi}_3 \qhpl_2.
    \end{equation}
  \item For all $h \in H$, the following equation holds:
    \begin{equation}
      \label{eq:left-cointegral-3}
      \qhVL_1 h_{(1)} \qhWL_1
      \, \langle \lambda, \qhVL_2 h_{(2)} \qhWL_2 \rangle
      = \langle \lambda, h \rangle \, 1,
    \end{equation}
    where $\qhWL$ is the element of $H^{\otimes 2}$ defined by
    \begin{equation}
      \label{eq:q-Hopf-def-WL}
      \qhWL = \qhmu(\overline{\qhphi}_1) \overline{\qhf}_1 \qhS(\qhqr_2 \overline{\qhphi}_3)
      \otimes \overline{\qhf}_2 \qhS((\qhqr_1 \leftharpoonup \qhmu) \overline{\qhphi}_2).
    \end{equation}
  \item For all left integrals $\Lambda$ in $H$, the following equation holds:
    \begin{equation}
      \tag{$=$ \eqref{eq:left-cointegral-4}}
      \qhqr_1 \Lambda_{(1)} \qhpr_1 \langle \lambda, \qhqr_2 \Lambda_{(2)} \qhpr_2 \rangle
      = \qhmu(\qhbeta) \lambda(\Lambda) 1.
    \end{equation}
  \end{enumerate}
\end{theorem}
\begin{proof}
  We prove this theorem in the following way:
  \begin{center}
    (1) $\Rightarrow$ (2) $\Rightarrow$ (3) $\Rightarrow$ (5) $\Rightarrow$ (1),
    \quad (2) $\Leftrightarrow$ (4),
  \end{center}
  Lemmas~\ref{lem:left-coint-characterization-1}, \ref{lem:left-coint-characterization-2} and~\ref{lem:left-coint-characterization-3} prove (1) $\Rightarrow$ (2), (2) $\Rightarrow$ (3) and (3) $\Rightarrow$ (5), respectively. We prove (5) $\Rightarrow$ (1). Suppose that $\lambda \in H^*$ satisfies \eqref{eq:left-cointegral-4} for any left integral $\Lambda$ in $H$. We fix a non-zero cointegral $\lambda_0$ on $H$ and define $\Theta_{\qhR}: H \to H^*$ by $\Theta_{\qhR}(h) = \lambda_0 \leftharpoonup h$ for $h \in H$. Then, by Theorem \ref{thm:Frobenius-dual-basis},  $\Theta_{\qhR}^{-1}(\lambda)$ is a scalar multiple of the unit $1 \in H$. Thus we have $\lambda \in \Theta_{\qhR}(k 1) = k \lambda_0 = \cointl$, that is, $\lambda$ is a left cointegral on $H$.

  (2) $\Leftrightarrow$ (4). To prove this, we require the following relations:
  \begin{align}
    \label{eq:left-coint-proof-UL-WL-1}
    \qhWL & = \qhmu(\qhphi_1) \qhS(\qhphi_2)_{(1)} \qhUL_1 \qhphi_3
    \otimes \qhS(\qhphi_2)_{(2)} \qhUL_2, \\
    \label{eq:left-coint-proof-UL-WL-2}
    \qhUL & = \qhmu(\overline{\qhphi}_1) \qhS(\overline{\qhphi}_2)_{(1)} \qhWL_1 \overline{\qhphi}_3
    \otimes \qhS(\overline{\qhphi}_2)_{(2)} \qhWL_2.
  \end{align}
  The first equation is proved as follows:
  \begin{align*}
    & \qhmu(\qhphi_1) \, \qhS(\qhphi_2)_{(1)} \qhano{\qhUL_1} \qhphi_3
      \otimes \qhS(\qhphi_2)_{(2)} \qhano{\qhUL_2} \\
    {}^{\eqref{eq:q-Hopf-def-UL}}
    & = \qhmu(\qhphi_1) \, \qhano{\qhS(\qhphi_2)_{(1)} \overline{\qhf}_1} \qhS(\qhqr_2) \qhphi_3
      \otimes \qhano{\qhS(\qhphi_2)_{(2)} \overline{\qhf}_2} \qhS(\qhqr_1) \\
    {}^{\eqref{eq:q-Hopf-f-1}}
    & = \qhmu(\qhphi_1) \, \overline{\qhf}_1 \qhano{\qhS(\qhphi_{2(2)}) \qhS(\qhqr_2) \qhphi_3}
      \otimes \overline{\qhf}_2 \qhano{\qhS(\qhphi_{2(1)}) \qhS(\qhqr_1)} \\
    & = \qhmu(\qhano{\qhphi_1}) \, \overline{\qhf}_1
      \qhS(\qhano{\overline{\qhS}(\qhphi_3) \qhqr_2\qhphi_{2(2)}})
      \otimes \overline{\qhf}_2 \qhS(\qhano{\qhqr_1 \qhphi_{2(1)}}) \\
    {}^{\eqref{eq:q-Hopf-qR-phi}}
    & = \qhmu(\qhqr_{1(1)} \overline{\qhphi}{}_1) \, \overline{\qhf}_1 \qhS(\qhqr_2 \overline{\qhphi}{}_3)
      \otimes \overline{\qhf}_2 \qhS(\qhqr_{1(2)} \overline{\qhphi}{}_2)
      \stackrel{\eqref{eq:q-Hopf-def-WL}}{=} \qhWL.
  \end{align*}
  The second one easily follows from the first one. Now we suppose that (2) holds. Then we have
  \begin{align*}
    \langle \lambda, h \rangle 1
    & = \qhmu(\qhphi_1) \cdot \qhano{\qhmu(\overline{\qhphi}_1) \langle \lambda, h \qhS(\qhphi_2) \cdot
      \qhS(\overline{\qhphi}_2) \rangle \overline{\qhphi}_3} \cdot \qhphi_3 \\
    {}^{\eqref{eq:left-cointegral-1}}
    & = \qhmu(\qhphi_1) \qhVL_1 (h \qhS(\qhphi_2))_{(1)} \qhUL_1
      \, \langle \lambda, \qhVL_2 (h \qhS(\qhphi_2))_{(2)} \qhUL_2 \rangle \qhphi_3 \\
    & = \qhano{\qhmu(\qhphi_1)} \qhVL_1 h_{(1)} \qhano{\qhS(\qhphi_2)_{(1)} \qhUL_1 \qhphi_3}
      \, \langle \lambda, \qhVL_2 h_{(2)} \qhano{\qhS(\qhphi_2)_{(2)} \qhUL_2} \rangle \\
    {}^{\eqref{eq:left-coint-proof-UL-WL-1}}
    & = \qhVL_1 h_{(1)} \qhWL_1
      \, \langle \lambda, \qhVL_2 h_{(2)} \qhWL_2 \rangle
  \end{align*}
  for all $h \in H$. Thus (4) holds. If, conversely, (4) holds, then we prove that (2) holds as follows:
  \begin{align*}
    & \qhmu(\overline{\qhphi}_{1})
      \qhano{\langle \lambda, h \qhS(\overline{\qhphi}_2) \rangle}
      \, \overline{\qhphi}_3 \\
    {}^{\eqref{eq:left-cointegral-3}}
    & = \qhVL_1 (h \qhS(\overline{\phi}_{2}))_{(1)} \qhWL_1
      \, \langle \lambda, \qhVL_2 (h \qhS(\overline{\phi}_{2}))_{(2)} \qhWL_2 \rangle
      \, \overline{\qhphi}_3 \\
    {}^{\eqref{eq:left-coint-proof-UL-WL-2}}
    & = \qhVL_1 h_{(1)} \qhUL_1
      \, \langle \lambda, \qhVL_2 h_{(2)} \qhUL_2 \rangle.
      \qedhere
  \end{align*}
\end{proof}

We note that a right cointegral on $H$ is just a left cointegral on the quasi-Hopf algebra $H^{\cop}$. By rephrasing the above theorem for $H^{\cop}$ by using \eqref{eq:q-Hopf-def-cop}, \eqref{eq:q-Hopf-pq-cop}, \eqref{eq:q-Hopf-f-op-cop} and \eqref{eq:q-Hopf-UL-VR-cop}, we obtain the following theorem:

\begin{theorem}
  \label{thm:right-coint-characterization}
  For $\lambda \in H^*$, the following are equivalent:
  \begin{enumerate}
  \item $\lambda$ is a right cointegral on $H$.
  \item For all $h \in H$, the following equation holds:
    \begin{equation}
      \label{eq:right-cointegral-1}
      \langle \lambda, \qhVR_1 h_{(1)} \qhUR_1 \rangle
      \, \qhVR_2 h_{(2)} \qhUR_2
      = \qhphi_1
      \langle \lambda, h \overline{\qhS}(\qhphi_2) \rangle
      \, \qhmu(\qhphi_3)
    \end{equation}
  \item For all $h \in H$, the following equation holds:
    \begin{equation}
      \label{eq:right-cointegral-2}
      \langle \lambda, \qhql_1 h_{(1)} \qhpl_1 \rangle \, \qhql_2 h_{(2)} \qhpl_2
      = \qhqr_1 \qhphi_1 \qhpr_1
      \langle \lambda, \qhS(\qhqr_2) h \overline{\qhS}(\qhphi_2 \qhpr_2) \rangle
      \qhmu(\qhphi_{3})
    \end{equation}
  \item For all $h \in H$, the following equation holds:
    \begin{equation}
      \label{eq:right-cointegral-3}
      \langle \lambda, \qhVR_1 h_{(1)} \qhWR_1 \rangle
      \, \qhVR_2 h_{(2)} \qhWR_2
      = \langle \lambda, h \rangle \, 1,
    \end{equation}
    where $\qhWR$ is the element of $H^{\otimes 2}$ defined by
    \begin{equation}
      \label{eq:q-Hopf-def-W-L}
      \qhWR =  \qhmu(\qhphi_3) \overline{\qhS}((\qhmu \rightharpoonup \qhql_2) \qhphi_2 \overline{\qhf}_2)
      \otimes \overline{\qhS}(\qhql_1 \qhphi_1 \overline{\qhf}_1)
    \end{equation}
  \item For all left integrals $\Lambda$ in $H$, the following equation holds:
    \begin{equation}
      \label{eq:right-cointegral-4}
      \langle \lambda, \qhql_1 \Lambda_{(1)} \qhpl_1 \rangle
      \qhql_2 \Lambda_{(2)} \qhpl_2
      = \qhmu(\overline{\qhS}(\qhbeta)) \langle \lambda, \Lambda \rangle.
    \end{equation}
  \end{enumerate}
\end{theorem}

\section{Yetter-Drinfeld category}
\label{sec:yetter-drinfeld}

\subsection{Yetter-Drinfeld modules of the first kind}

Throughout this section, $H$ is a quasi-Hopf algebra with bijective antipode. As in the case of ordinary Hopf algebras, a Yetter-Drinfeld module over $H$ is defined to be a left $H$-module $V$ equipped with a linear map $\delta: V \to H \otimes V$, denoted by $\delta(v) = v_{\YDA{-1}} \otimes v_{\YDA{0}}$, such that the family
\begin{equation*}
  \{ V \otimes M \to M \otimes V, \ v \otimes m \mapsto v_{\YDA{-1}} m \otimes v_{\YDA{0}} \}_{M \in {}_H \Mod}
\end{equation*}
of maps makes $V$ into an object of the Drinfeld center of ${}_H \Mod$. An explicit definition in this spirit is found, {\it e.g.}, in Majid \cite{MR1631648} and Bulacu-Caenepeel-Panaite \cite{MR2106925,MR2194347}. On the other hand, Schauenburg \cite{MR1897403} introduced another kind of Yetter-Drinfeld module over $H$ from the viewpoint of the fundamental theorem for quasi-Hopf bimodules. Although these two kinds of Yetter-Drinfeld modules are equivalent notions, each of them has advantages and disadvantages. In this paper, we use both of them. We first recall the definition of the first kind:

\begin{definition}[\cite{MR1631648,MR2106925,MR2194347}]
  A {\em Yetter-Drinfeld module of the first kind} is a left $H$-module $V$ endowed with a linear map
  \begin{equation*}
    \YDAdelta_V: V \to H \otimes V,
    \quad \YDAdelta_V(v) = v_{\YDA{-1}} \otimes v_{\YDA{0}}
    \quad (v \in V)
  \end{equation*}
  such that the following equations hold for all $h \in H$ and $v \in V$.
  \begin{gather}
    \label{eq:YD-1-1}
    \begin{aligned}
      & \qhphi_1 v_{\YDA{-1}} \otimes (\qhphi_2 v_{\YDA{0}})_{\YDA{-1}} \qhphi_3 \otimes (\qhphi_2 v_{\YDA{0}})_{\YDA{0}}
      \\
      & \qquad \qquad = \qhphi'_1(\qhphi_1 v)_{\YDA{-1}(1)} \qhphi_2 \otimes \qhphi'_2 (\qhphi_1 v)_{\YDA{-1}(2)} \qhphi_3 \otimes \qhphi'_3 (\qhphi_1 v)_{\YDA{0}},
    \end{aligned} \\
    \label{eq:YD-1-2}
    \qhepsilon(v_{\YDA{-1}}) v_{\YDA{0}} = v, \\
    \label{eq:YD-1-3}
    h_{(1)} v_{\YDA{-1}} \otimes h_{(2)} v_{\YDA{0}}
    = (h_{(1)} v)_{\YDA{-1}} h_{(2)} \otimes (h_{(1)} v)_{\YDA{0}},
  \end{gather}
  where $\qhphi'$ is a copy of $\qhphi$. The map $\YDAdelta_V$ is called the {\em coaction of the first kind}. We denote by ${}^H_H \YD_1$ the category of Yetter-Drinfeld modules of the first kind and $k$-linear maps that preserves the action and the coaction of $H$.
\end{definition}

The trivial $H$-module $\unitobj = k$ is a Yetter-Drinfeld module of the first kind by the coaction determined by $\YDAdelta_{\unitobj}(1) = 1 \otimes 1$. If $V$ and $W$ are Yetter-Drinfeld modules of the first kind, then their tensor product $H$-module $V \otimes W$ is a Yetter-Drinfeld module of the first kind by the coaction given by
\begin{equation}
  \label{eq:YD-1-tensor}
  \begin{aligned}
  \YDAdelta_{V \otimes W}(v \otimes w)
  & = \qhphi_1 (\overline{\qhphi}_{1} \qhphi'_{1} v)_{\YDA{-1}} \overline{\qhphi}_{2} (\qhphi'_{2} w)_{\YDA{-1}} \qhphi'_{3} \\
  & \qquad \qquad \otimes \qhphi_{2}(\overline{\qhphi}_{1} \qhphi'_{1} v)_{\YDA{0}}
  \otimes \qhphi_{3} \overline{\qhphi}_{3} (\qhphi'_{2} w)_{\YDA{0}}
\end{aligned}
\end{equation}
for $v \in V$ and $w \in W$ (see \cite[Proposition 2.2]{MR1631648}). The category ${}^H_H \YD_1$ is a monoidal category with this tensor product.

Given $V \in {}^H_H \YD_1$ and $X \in {}_H \Mod$, we define $\YDsigma_{V,X}: V \otimes X \to X \otimes V$ by $\YDsigma_{V,X}(v \otimes x) = v_{\YDA{-1}} x \otimes v_{\YDA{0}}$ for $v \in V$ and $x \in X$. The family $\YDsigma_V = \{ \YDsigma_{V,X} \}_{X}$ is natural in $X \in {}_H \Mod$ and the pair $(V, \YDsigma_V)$ is in fact an object of the Drinfeld center $\mathcal{Z}({}_H \Mod)$ of ${}_H \Mod$. Moreover, the assignment $(V, \YDAdelta_V) \mapsto (V, \YDsigma_V)$ is an isomorphism of monoidal categories from ${}^H_H \YD_1$ to $\mathcal{Z}({}_H \Mod)$.

\subsection{Yetter-Drinfeld modules of the second kind}

The category ${}_H^{} \Mod_H^H$ is naturally a left ${}_H \Mod_H$-module category. Hence ${}_H \Mod$ is also a left ${}_H \Mod_H$-module category in such a way that the category equivalence~\eqref{eq:q-Hopf-bimod-equiv-1} is ${}_H \Mod_H$-equivariant. Schauenburg \cite{MR1897403} gave an explicit description of such an action of ${}_H \Mod_H$ on ${}_H \Mod$. Given $M \in {}_H \Mod_H$, we denote by ${}_{\ad}M$ the vector space $M$ equipped with the left $H$-module structure given by
\begin{equation*}
  h \triangleright m = h_{(1)} m \qhS(h_{(2)})
  \quad (h \in H, m \in M).
\end{equation*}
Then the action $\ogreaterthan: {}_H \Mod_H \times {}_H \Mod \to {}_H \Mod$ is given by
\begin{equation*}
  M \ogreaterthan V = {}_{\ad} (M \hatotimes V)
  \quad (M \in {}_H \Mod_H, V \in {}_H \Mod).
\end{equation*}
To describe the associator $\Omega_{M,N,V}: (M \hatotimes N) \ogreaterthan V \to M \ogreaterthan (N \ogreaterthan V)$ for the action of ${}_H \Mod_H$, we introduce the following element:
\begin{equation}
  \label{eq:q-Hopf-def-omega}
  \qhomega = (1 \otimes 1 \otimes 1 \otimes \overline{\qhf}_2 \otimes \overline{\qhf}_1)
  \cdot (\id \otimes \Delta \otimes \qhS \otimes \qhS)(\chi)
  \cdot (\qhphi \otimes 1 \otimes 1) \in H^{\otimes 5},
\end{equation}
where $\chi$ is the element of $H^{\otimes 4}$ given by
\begin{equation}
  \label{eq:q-Hopf-def-chi}
  \begin{aligned}
    \chi
    & = (\id \otimes \Delta \otimes \id)(\qhphi^{-1})
    \cdot (1 \otimes \qhphi^{-1}) \cdot (\id \otimes \id \otimes \Delta)(\qhphi) \\
    {}^{\eqref{eq:q-Hopf-def-3}}
    & = (\qhphi \otimes 1) \cdot (\Delta \otimes \id \otimes \id)(\qhphi^{-1}).
  \end{aligned}
\end{equation}
The associator $\Omega$ is then given by
\begin{equation}
  \label{eq:def-Omega}
  \Omega_{M,N,V}(m \otimes n \otimes v)
  = \qhomega_1 m \qhomega_5 \otimes \qhomega_2 n \qhomega_4 \otimes \qhomega_3 v
\end{equation}
for $m \in M$, $n \in N$ and $v \in V$.

\begin{definition}[Schauenburg \cite{MR1897403}]
  We define the category ${}^H_H \YD_2$ to be the category of left $H$-comodules in the left ${}_H \Mod_H$-module category ${}_H \Mod$ (see Definition \ref{def:modules-in-module-cat}). An object of ${}^H_H \YD_2$ is referred to as a {\em Yetter-Drinfeld module of the second kind}. 
  Stated differently, a Yetter-Drinfeld module of the second kind is a left $H$-module $V$ endowed with a linear map $\YDBdelta_V: V \to H \otimes V$, expressed as $v \mapsto v_{\YDB{-1}} \otimes v_{\YDB{0}}$, such that the equations
  \begin{gather}
    \label{eq:YD-2-1}
    v_{\YDB{-1}} \otimes v_{\YDB{0} \YDB{-1}} \otimes v_{\YDB{0} \YDB{0}}
    = \qhomega_1 (v_{\YDB{-1}})_{(1)} \qhomega_5 \otimes \qhomega_2 (v_{\YDB{-1}})_{(2)} \qhomega_4 \otimes \qhomega_3 v_{\YDB{0}}, \\
    \label{eq:YD-2-2}
    \qhepsilon(v_{\YDB{-1}}) v_{\YDB{0}} = v, \\
    \label{eq:YD-2-3}
    \YDBdelta_V(h v) = h_{(1, 1)} v_{\YDB{-1}} \qhS(h_{(2)}) \otimes h_{(1, 2)} v_{\YDB{0}}
  \end{gather}
  hold for all $h \in H$ and $v \in V$.
\end{definition}

Schauenburg \cite{MR1897403} showed that there is an isomorphism $\Psi: {}^H_H \YD_1 \to {}^H_H \YD_2$ of categories. The isomorphism $\Psi$ is the identity on morphisms and keeps the underlying $H$-module unchanged. Given $V \in {}^H_H \YD_1$, the isomorphism $\Psi$ replaces the coaction of $V$ of the first kind with
\begin{equation}
  \label{eq:YD1-to-YD2}
  \YDBdelta_V(v) = (\qhpr_1 v)_{\YDA{-1}} \qhpr_2 \otimes (\qhpr_1 v)_{\YDA{0}}
  \quad (v \in V).
\end{equation}
The inverse of the isomorphism $\Psi$ replaces the second kind coaction of $V \in {}^H_H \YD_2$ with
\begin{equation}
  \label{eq:YD2-to-YD1}
  \YDAdelta_V(v)
  = \qhqr_{1(1)} v_{\YDB{-1}} \qhS(\qhqr_2)
  \otimes \qhqr_{1(2)} v_{\YDB{0}}
  \quad (v \in V).
\end{equation}
From now on, we identify ${}^H_H \YD_1$ with ${}^H_H \YD_2$, and denote both of them by ${}^H_H \YD$. An object of ${}^H_H \YD$ is thought of as a left $H$-module $V$ equipped with two kinds of coaction $\YDAdelta_V$ and $\YDBdelta_V$ related to each other by equations~\eqref{eq:YD1-to-YD2} and~\eqref{eq:YD2-to-YD1}.

\subsection{Induction to the Yetter-Drinfeld category, I}

Given $V \in {}_H \Mod$, we define $R(V) \in {}_H \Mod$ by $R(V) = H \ogreaterthan V$. The left $H$-module $R(V)$ is a Yetter-Drinfeld module of the second kind by the free left $H$-coaction given by
\begin{equation*}
  R(V) = H \ogreaterthan V
  \xrightarrow{\ \Delta \ogreaterthan \id_V \ }
  (H \hatotimes H) \ogreaterthan V
  \xrightarrow{\ \Omega_{H,H,V} \ }
  H \ogreaterthan (H \ogreaterthan V)
  = H \ogreaterthan R(V).
\end{equation*}
Specifically, the action of $H$ and the second kind coaction of $H$ on $R(V)$ are given respectively by
\begin{gather}
  \label{eq:Rad-ind-2nd-1}
  h \triangleright (a \otimes v)
  = h_{(1, 1)} a \qhS(h_{(2)}) \otimes h_{(1, 2)} v, \\
  \label{eq:Rad-ind-2nd-2}
  \YDBdelta_{R(V)}(a \otimes v)
  = \qhomega_1 a_{(1)} \qhomega_5 \otimes \qhomega_2 a_{(2)} \qhomega_4 \otimes \qhomega_3 v
\end{gather}
for $a, h \in H$ and $v \in V$. By~\eqref{eq:YD2-to-YD1}, the first kind coaction is given by
\begin{equation}
  \label{eq:Rad-ind-1st-coact}
  \YDAdelta_{R(V)}(a \otimes v)
  = \qhqr_{1(1)} \qhomega_1 a_{(1)} \qhomega_5 \qhS(\qhqr_2)
  \otimes \qhqr_{1(2)} \triangleright (\qhomega_2 a_{(2)} \qhomega_4 \otimes \qhomega_3 v).
\end{equation}
Now let $F: {}^H_H \YD \to {}_H \Mod$ be the forgetful functor. We recall that ${}^H_H \YD = {}^H_H \YD_2$ is the category of left $H$-comodules in ${}_H \Mod$. Hence the free $H$-comodule functor is right adjoint to $F$. Namely, we have:
 
\begin{theorem}
  \label{thm:Rad-ind-2nd}
  The assignment $V \mapsto R(V)$ extends to a functor from ${}_H \Mod$ to ${}^H_H \YD$. This functor is right adjoint to $F$ with the unit $\eta$ and the counit $\varepsilon$ given by
  \begin{gather}
    \label{eq:Rad-ind-2nd-unit}
    \eta_M: M \to R F(M),
    \quad m \mapsto m_{\YDB{-1}} \otimes m_{\YDB{0}}
    \quad (m \in M \in {}^H_H \YD), \\
    \label{eq:Rad-ind-2nd-counit}
    \varepsilon_V: F R(V) \to V,
    \quad a \otimes v \mapsto \qhepsilon(a) v
    \quad (a \in H, v \in V \in {}_H \Mod).
  \end{gather}
\end{theorem}

As we have recalled in Subsection~\ref{subsec:monoidal-cats}, the functor $R$ is a monoidal functor as a right adjoint of the strict monoidal functor $F$.
The structure morphisms
\begin{equation*}
  R^{(0)}: \unitobj \to R(\unitobj)
  \quad \text{and} \quad
  R^{(2)}_{X,Y}: R(X) \otimes R(Y) \to R(X \otimes Y)
  \quad (X, Y \in {}_H \Mod)
\end{equation*}
are given by $R^{(0)} = \eta_{\unitobj}$ and $R_{X,Y}^{(2)} = R(\varepsilon_{X} \otimes \varepsilon_{Y}) \circ \eta_{R(X) \otimes R(Y)}$, respectively, since $F$ is strict. By the explicit formulas for $\eta$ and $\varepsilon$ given by the above theorem, we now prove:

\begin{lemma}
  \label{lem:R-monoidal-struc}
  The morphism $R^{(0)}$ is determined by
  \begin{equation}
    \label{eq:induction-R-0}
    R^{(0)}(1) = \qhbeta \otimes 1.
  \end{equation}
  The natural transformation $R^{(2)}$ is given by
  \begin{equation}
    \label{eq:induction-R-2}
    \begin{aligned}
      & R^{(2)}_{X, Y}((a \otimes x) \otimes (b \otimes y)) \\
      & = \qhphi_1 \overline{\qhphi}_{1(1)} \qhqr_{1(1)} \overline{\qhphi}{}'_{1(1,1)} a \qhS(\qhqr_2 \overline{\qhphi}{}'_{1(2)})
      \overline{\qhphi}{}'_{2} \qhphi{}'_1 b
      \qhS(\overline{\qhphi}_{3} \overline{\qhphi}{}'_{3(2)} \qhphi{}'_3) \\
      & \qquad \qquad \otimes \qhphi_2 \overline{\qhphi}_{1(2)} \qhqr_{1(2)} \overline{\qhphi}{}'_{1(1,2)} x
      \otimes \qhphi_3 \overline{\qhphi}_2 \overline{\qhphi}{}'_{3(1)} \qhphi{}'_2 y
    \end{aligned}
  \end{equation}
  for $X, Y \in {}_H \Mod$, $a, b \in H$, $x \in X$ and $y \in Y$, where $\qhphi'$ is a copy of $\qhphi$.
\end{lemma}
\begin{proof}
  We note that the unit $\eta$ is given by the coaction of the second kind. Since the first kind coaction of the trivial Yetter-Drinfeld module $\unitobj = k$ is determined by $\YDAdelta_{\unitobj}(1) = 1 \otimes 1$, we have
  \begin{equation*}
    R^{(0)}(1) = \YDBdelta_{\unitobj}(1)
    \stackrel{\eqref{eq:YD1-to-YD2}}{=}
    \qhpr_2 \otimes \qhepsilon(\qhpr_1)
    \stackrel{\eqref{eq:q-Hopf-def-pR}}{=}
    \qhbeta \otimes 1.
  \end{equation*}
  Thus \eqref{eq:induction-R-0} is proved.

  We check that $R^{(2)}_{X,Y}$ is given by \eqref{eq:induction-R-2}. For simplicity, we write $m = m_H \otimes m_V \in H \otimes V$ for an element $m \in R(V)$. The following equations hold:
  \begin{gather}
    \label{eq:R-monoidal-struc-pf-1}
    m_{\YDA{-1}} \otimes \varepsilon_V(h \triangleright m_{\YDA{0}})
    = \qhqr_{1(1)} m_{H} \qhS(\qhqr_2) \otimes h \qhqr_{1(2)} m_{V}, \\
    \label{eq:R-monoidal-struc-pf-2}
    \begin{aligned}
      & \qhqr_{1(1)}(h_{(1)} \triangleright m)_{H} \qhS(\qhqr_2) h_{(2)}
      \otimes \qhqr_{1(2)} (h_{(1)} \triangleright m)_{V} \\
      & \qquad \qquad = h_{(1)} \qhqr_{1(1)} m_{H} \qhS(\qhqr_2) \otimes h_{(2)} \qhqr_{1(2)} m_{V},
    \end{aligned} \\
    \label{eq:R-monoidal-struc-pf-3}
    \qhqr_{1(1)}(\qhpr_{1} \triangleright m)_{H} \qhS(\qhqr_2) \qhpr_2
    \otimes \qhqr_{1(2)} (\qhpr_{1} \triangleright m)_{V} = m
  \end{gather}
  for $h \in H$ and $m \in R(V)$. Indeed, the first one is proved as follows:
  \begin{align*}
    & m_{\YDA{-1}} \otimes \varepsilon_V(h \triangleright m_{\YDA{0}}) \\
    \ {}^{\eqref{eq:Rad-ind-1st-coact}}
    & = \qhqr_{1(1)} \qhomega_1 m_{H(1)} \qhomega_5 \qhS(\qhqr_2)
      \otimes \qhano{\varepsilon_V(h \qhqr_{1(2)} \triangleright (\qhomega_2 m_{H(2)} \qhomega_4 \otimes \qhomega_3 m_{V}))} \\
    {}^{\eqref{eq:Rad-ind-2nd-counit}}
    & = \qhqr_{1(1)} \qhano{\qhomega_1 m_{H(1)} \qhomega_5} \qhS(\qhqr_2)
     \otimes h \qhqr_{1(2)} \cdot
      \qhano{\qhepsilon(\qhomega_2 m_{H(2)} \qhomega_4) \qhomega_3} m_{V} \\
    {}^{\eqref{eq:q-Hopf-def-omega}}
    & = \qhqr_{1(1)} m_{H} \qhS(\qhqr_2) \otimes h \qhqr_{1(2)} m_{V}.
  \end{align*}
  The second one is proved as follows:
  \begin{align*}
    & \qhano{\qhqr_{1(1)}(h_{(1)} \triangleright m)_{H} \qhS(\qhqr_2)} h_{(2)}
      \otimes \qhano{\qhqr_{1(2)} (h_{(1)} \triangleright m)_{V}} \\
    {}^{\eqref{eq:R-monoidal-struc-pf-1}}
    & = \qhano{(h_{(1)} \triangleright m)_{\YDA{-1}} h_{(2)}}
      \otimes \varepsilon_V(\qhano{(h_{(1)} \triangleright m)_{\YDA{0}}}) \\
    {}^{\eqref{eq:YD-1-3}}
    & = h_{(1)} \qhano{m_{\YDA{-1}}}
      \otimes \qhano{\varepsilon_V(h_{(2)} \triangleright m_{\YDA{0}})} \\
    {}^{\eqref{eq:R-monoidal-struc-pf-1}}
    & = h_{(1)} \qhqr_{1(1)} m_{H} \qhS(\qhqr_2) \otimes h_{(2)} \qhqr_{1(2)} m_{V}.
  \end{align*}
  The third one is proved as follows:
  \begin{align*}
    & \qhqr_{1(1)}\qhano{(\qhpr_{1} \triangleright m)_{H}} \qhS(\qhqr_2) \qhpr_2
    \otimes \qhqr_{1(2)} \qhano{(\qhpr_{1} \triangleright m)_{V}} \\
    {}^{\eqref{eq:Rad-ind-2nd-1}}
    & = \qhqr_{1(1)} \qhpr_{1(1,1)} m_{H} \qhS(\qhpr_{1(2)}) \qhS(\qhqr_2) \qhpr_2
    \otimes \qhqr_{1(2)} \qhpr_{1(1,2)} m_{V} \\
    & = (\qhano{\qhqr_{1} \qhpr_{1(1)}})_{(1)} m_{H}
      \qhS(\qhano{\overline{\qhS}(\qhpr_2) \qhqr_2 \qhpr_{1(2)}})
    \otimes (\qhano{\qhqr_{1} \qhpr_{1(1)}})_{(2)} m_{V} \\
    {}^{\eqref{eq:q-Hopf-pR-qR-2}}
    & = m_{H} \otimes m_{V} = m.
  \end{align*}
  For $v \in R(X)$ and $w \in R(Y)$, we have
  \begin{align*}
    & \YDBdelta_{R(X) \otimes R(Y)}(v \otimes w) \\
    {}^{\eqref{eq:YD1-to-YD2}}
    & = \qhano{(\qhpr_1 \cdot (v \otimes w))_{\YDA{-1}}} \qhpr_2
      \otimes \qhano{(\qhpr_1 \cdot (v \otimes w))_{\YDA{0}}} \\
    {}^{\eqref{eq:YD-1-tensor}}
    & = \qhphi_1 (\overline{\qhphi}_1 \qhano{\qhphi'_1 \qhpr_{1(1)}} \triangleright v)_{\YDA{-1}}
      \overline{\qhphi}_2 (\qhano{\qhphi'_2 \qhpr_{1(2)}} \triangleright w)_{\YDA{-1}}
      \qhphi'_3 \qhpr_2 \\
    & \qquad \otimes \qhphi_2 \triangleright
      (\overline{\qhphi}_1 \qhano{\qhphi'_1 \qhpr_{1(1)}} \triangleright v)_{\YDA{0}}
      \otimes \qhphi_3 \overline{\qhphi}_3
      \triangleright (\qhano{\qhphi'_2 \qhpr_{1(2)}} \triangleright w)_{\YDA{-1}} \\
    & = \underbrace{
      \qhphi_1 (\overline{\qhphi}_1 \mathbbm{t}_1 \triangleright v)_{\YDA{-1}}
      \overline{\qhphi}_2 (\mathbbm{t}_2 \triangleright w)_{\YDA{-1}} \mathbbm{t}_3
      }_{\in H}
      \otimes \underbrace{
      \qhphi_2 \triangleright (\overline{\qhphi}_1 \mathbbm{t}_1 \triangleright v)_{\YDA{0}}
      }_{\in R(X)}
      \otimes \underbrace{
      \qhphi_3 \overline{\qhphi}_3 \triangleright (\mathbbm{t}_2 \triangleright w)_{\YDA{0}}
      }_{\in R(Y)},
  \end{align*}
  where $\mathbbm{t} \in H^{\otimes 3}$ is given by
  \begin{align}
    \label{eq:R-monoidal-struc-pf-4}
    \mathbbm{t}
    := \qhphi'_1 \qhpr_{1(1)} \otimes \qhphi'_2 \qhpr_{1(2)} \otimes \qhphi'_3 \qhpr_2
    \stackrel{\eqref{eq:q-Hopf-pR-phi}}{=}
    \overline{\qhphi}{}'_1 \otimes \overline{\qhphi}{}'_{2(1)} \qhpr_1
    \otimes \overline{\qhphi}{}'_{2(2)} \qhpr_2 \qhS(\overline{\qhphi}{}'_3).
  \end{align}
  Set $\qhq = \qhq' = \qhqr$. We verify~\eqref{eq:induction-R-2} as follows:
  \allowdisplaybreaks
  \begin{align*}
    R^{(2)}_{X,Y}(v \otimes w)
    & = R(\varepsilon_X \otimes \varepsilon_Y) \YDBdelta_{R(X) \otimes R(Y)}(v \otimes w) \\
    & = \qhphi_1 \qhano{(\overline{\qhphi}_1 \mathbbm{t}_1 \triangleright v)_{\YDA{-1}}}
    \overline{\qhphi}_2 \qhano{(\mathbbm{t}_2 \triangleright w)_{\YDA{-1}}} \mathbbm{t}_3 \\
    & \qquad \qquad
      \otimes \qhano{\varepsilon_X(\qhphi_2 \triangleright (\overline{\qhphi}_1 \mathbbm{t}_1 \triangleright v)_{\YDA{0}})}
      \otimes \qhano{\varepsilon_Y(\qhphi_3 \overline{\qhphi}_3 \triangleright (\mathbbm{t}_2 \triangleright w)_{\YDA{0}})} \\
    {}^{\eqref{eq:R-monoidal-struc-pf-1}}
    & = \qhphi_1 \qhq_{1(1)} (\overline{\qhphi}_1 \qhano{\mathbbm{t}_1} \triangleright v)_{H} \qhS(\qhq_2)
    \overline{\qhphi}_2 \qhq'_{1(1)} (\qhano{\mathbbm{t}_2} \triangleright w)_{H} \qhS(\qhq'_2) \qhano{\mathbbm{t}_3} \\
    & \qquad \qquad \otimes \qhphi_2 \qhq_{1(2)} (\overline{\qhphi}{}_1 \qhano{\mathbbm{t}_1} \triangleright v)_{X}
    \otimes \qhphi_3 \overline{\qhphi}_3 \qhq'_{1(2)} (\qhano{\mathbbm{t}_2} \triangleright w)_{Y} \\
    {}^{\eqref{eq:R-monoidal-struc-pf-4}}
    & = \qhphi_1 \qhq_{1(1)} (\overline{\qhphi}_1 \overline{\qhphi}{}'_1 \triangleright v)_{H} \qhS(\qhq_2)
      \overline{\qhphi}_2
      \qhano{\qhq'_{1(1)} (\overline{\qhphi}{}'_{2(1)} \qhpr_1 \triangleright w)_{H} \qhS(\qhq'_2)
      \overline{\qhphi}{}'_{2(2)} \qhpr_2}
      \qhS(\overline{\qhphi}{}'_3) \\
    & \qquad \qquad \otimes \qhphi_2 \qhq_{1(2)} (\overline{\qhphi}{}_1 \overline{\qhphi}{}'_1 \triangleright v)_{X}
      \otimes \qhphi_3 \overline{\qhphi}_3
      \qhano{\qhq'_{1(2)} (\overline{\qhphi}{}'_{2(1)} \qhpr_1 \triangleright w)_{Y}} \\
    {}^{\eqref{eq:R-monoidal-struc-pf-2}}
    & = \qhphi_1 \qhq_{1(1)} (\overline{\qhphi}_1 \overline{\qhphi}{}'_1 \triangleright v)_{H} \qhS(\qhq_2)
      \overline{\qhphi}_2 \overline{\qhphi}{}'_{2(1)}
      \qhano{\qhq'_{1(1)} (\qhpr_1 \triangleright w)_{H} \qhS(\qhq'_2) \qhpr_2}
      \qhS(\overline{\qhphi}{}'_3) \\
    & \qquad \qquad \otimes \qhphi_2 \qhq_{1(2)} (\overline{\qhphi}{}_1 \overline{\qhphi}{}'_1 \triangleright v)_{X}
    \otimes \qhphi_3 \overline{\qhphi}_3 \overline{\qhphi}{}'_{2(2)} \qhano{\qhq'_{1(2)} (\qhpr_1 \triangleright w)_{Y}} \\
    {}^{\eqref{eq:R-monoidal-struc-pf-3}}
    & = \qhphi_1 \qhq_{1(1)}
      (\qhano{\overline{\qhphi}_1 \overline{\qhphi}{}'_1} \triangleright v)_{H} \qhS(\qhq_2)
      \qhano{\overline{\qhphi}_2 \overline{\qhphi}{}'_{2(1)}} w_{H}
    \qhS(\qhano{\overline{\qhphi}{}'_3}) \\
    & \qquad \qquad \otimes \qhphi_2 \qhq_{1(2)} (\qhano{\overline{\qhphi}{}_1 \overline{\qhphi}{}'_1} \triangleright v)_{X}
    \otimes \qhphi_3 \qhano{\overline{\qhphi}_3 \overline{\qhphi}{}'_{2(2)}} w_{Y} \\
    {}^{\eqref{eq:q-Hopf-def-3}}
    & = \qhphi_1 \qhano{\qhq_{1(1)} (\overline{\qhphi}_{1(1)} \overline{\qhphi}{}'_1
      \triangleright v)_{H} \qhS(\qhq_2) \overline{\qhphi}_{1(2)}}
      \overline{\qhphi}{}'_{2} \qhphi{}'_1 w_{H}
      \qhS(\overline{\qhphi}_{3} \overline{\qhphi}{}'_{3(2)} \qhphi{}'_3) \\
    & \qquad \qquad \otimes \qhphi_2
      \qhano{\qhq_{1(2)} (\overline{\qhphi}_{1(1)} \overline{\qhphi}{}'_1 \triangleright v)_{X}}
      \otimes \qhphi_3 \overline{\qhphi}_2 \overline{\qhphi}{}'_{3(1)} \qhphi{}'_2 w_{Y} \\
    {}^{\eqref{eq:R-monoidal-struc-pf-2}}
    & = \qhphi_1 \overline{\qhphi}_{1(1)} \qhq_{1(1)}
      \qhano{(\overline{\qhphi}{}'_1 \triangleright v)_{H}} \qhS(\qhq_2)
      \overline{\qhphi}{}'_{2} \qhphi{}'_1 w_{H}
      \qhS(\overline{\qhphi}_{3} \overline{\qhphi}{}'_{3(2)} \qhphi{}'_3) \\
    & \qquad \qquad \otimes \qhphi_2 \overline{\qhphi}_{1(2)} \qhq_{1(2)}
      \qhano{(\overline{\qhphi}{}'_1 \triangleright v)_{X}}
      \otimes \qhphi_3 \overline{\qhphi}_2 \overline{\qhphi}{}'_{3(1)} \qhphi{}'_2 w_{Y} \\
    {}^{\eqref{eq:Rad-ind-2nd-1}}
    & = \qhphi_1 \overline{\qhphi}_{1(1)} \qhq_{1(1)} \overline{\qhphi}{}'_{1(1,1)} v_{H} \qhS(\qhq_2 \overline{\qhphi}{}'_{1(2)})
    \overline{\qhphi}{}'_{2} \qhphi{}'_1 w_{H}
    \qhS(\overline{\qhphi}_{3} \overline{\qhphi}{}'_{3(2)} \qhphi{}'_3) \\
    & \qquad \qquad \otimes \qhphi_2 \overline{\qhphi}_{1(2)} \qhq_{1(2)} \overline{\qhphi}{}'_{1(1,2)} v_{X}
    \otimes \qhphi_3 \overline{\qhphi}_2 \overline{\qhphi}{}'_{3(1)} \qhphi{}'_2 w_{Y}.
    & \qedhere
  \end{align*}
\end{proof}

\subsection{The adjoint algebra}

Let $F \dashv R$ be the adjunction given in Theorem~\ref{thm:Rad-ind-2nd}. The Yetter-Drinfeld module $\mathbf{A} := R(\unitobj)$ is an algebra in ${}^H_H \YD$ as the image of the trivial algebra $\unitobj$ in ${}_H \Mod$ under a monoidal functor. We identify $\mathbf{A} = H$ as a vector space. Then the structure of $\mathbf{A}$ is given as follows:

\begin{theorem}
  \label{thm:BCP-alg}
  The action $\triangleright$, the coaction $\YDAdelta_{\mathbf{A}}$ of the first kind, the coaction $\YDBdelta_{\mathbf{A}}$ of the second kind, the multiplication $\star$, and the unit $1_{\mathbf{A}}$ of the algebra $\mathbf{A} \in {}^H_H \YD$ are given respectively by
  \begin{gather}
    \label{eq:BCP-alg-action}
    h \triangleright a = h_{(1)} a \qhS(h_{(2)}), \\
    \label{eq:BCP-alg-coaction-1}
    \YDAdelta_{\mathbf{A}}(a) = \qhphi_1 \qhphi'_{1(1)} a_{(1)} \overline{\qhf}_1 \qhS(\qhqr_2 \qhphi'_{2(2)}) \qhphi'_3 
    \otimes \qhphi_2 \qhphi'_{1(2)} a_{(2)} \overline{\qhf}_2 \qhS(\qhphi_3 \qhqr_1 \qhphi'_{2(1)}), \\
    \label{eq:BCP-alg-coaction-2}
    \YDBdelta_{\mathbf{A}}(a) = \qhphi_1 \overline{\qhphi}_{1(1)} a_{(1)} \overline{\qhf}_1 \qhS(\overline{\qhphi}_3)
    \otimes \qhphi_2 \overline{\qhphi}_{1(2)} a_{(2)} \overline{\qhf}_2 \qhS(\qhphi_3 \overline{\qhphi}_2), \\
    \label{eq:BCP-alg-mult}
    a \star b = \qhphi_1 a \qhS(\overline{\qhphi}_1 \qhphi_2) \qhalpha
    \overline{\qhphi}_2 \qhphi_{3(1)} b \qhS(\overline{\qhphi}_3 \qhphi_{3(2)}), \\
    \label{eq:BCP-alg-unit}
    1_{\mathbf{A}} = \qhbeta
  \end{gather}
  for $a, b \in \mathbf{A}$ and $h \in H$, where $\qhphi' = \qhphi$.
\end{theorem}

We call $\mathbf{A}$ the  {\em adjoint algebra}, since \eqref{eq:BCP-alg-action} is usually called the adjoint action of $H$ in the Hopf algebra theory. Theorem~\ref{thm:BCP-alg} shows that the adjoint algebra is identical to the algebra $H_0$ of \cite{MR2106925}. In other words, this theorem gives a category-theoretical origin of the algebra $H_0$ of \cite{MR2106925}. 

\begin{proof}[Proof of Theorem~\ref{thm:BCP-alg}]
  Equation~\eqref{eq:BCP-alg-action} is obvious from the definition of the functor $R$. It also follows from the definition of $R$ that the second kind coaction of $\mathbf{A}$ is given by
  \begin{align*}
    \YDBdelta_{\mathbf{A}}(a)
    & = \qhano{\qhomega_1} a_{(1)} \qhano{\qhomega_5}
      \otimes \qhano{\qhomega_2} a_{(2)} \qhano{\qhomega_4 \qhepsilon(\qhomega_3)} \\
    {}^{\eqref{eq:q-Hopf-def-omega}}
    & = \qhano{\chi_1} a_{(1)} \overline{\qhf}_1 \qhS(\qhano{\chi_4})
      \otimes \qhano{\chi_2} a_{(2)} \overline{\qhf}_2 \qhS(\qhano{\chi_3}) \\
    {}^{\eqref{eq:q-Hopf-def-chi}}
    & = \qhphi_1 \overline{\qhphi}_{1(1)} a_{(1)} \overline{\qhf}_1 \qhS(\overline{\qhphi}_3)
      \otimes \qhphi_2 \overline{\qhphi}_{1(2)} a_{(2)} \overline{\qhf}_2 \qhS(\qhphi_3 \overline{\qhphi}_2)
  \end{align*}
  for $a \in \mathbf{A}$. Thus~\eqref{eq:BCP-alg-coaction-2} is proved. We verify~\eqref{eq:BCP-alg-coaction-1} as follows: For $a \in \mathbf{A}$,
  \begin{align*}
    & \YDAdelta_{\mathbf{A}}(a) \\
    {}^{\eqref{eq:YD2-to-YD1},~\eqref{eq:BCP-alg-coaction-2}}
    & = \qhqr_{1(1)} \qhphi_1 \overline{\qhphi}_{1(1)}
      a_{(1)} \overline{\qhf}_1 \qhS(\overline{\qhphi}_3) \qhS(\qhqr_{2})
      \otimes \qhqr_{1(2)} \triangleright(\qhphi_2 \overline{\qhphi}_{1(2)}
      a_{(2)} \overline{\qhf}_2 \qhS(\qhphi_3 \overline{\qhphi}_2)) \\
    {}^{\eqref{eq:BCP-alg-action}}
    & = \qhano{\qhqr_{1(1)} \qhphi_1}
      \overline{\qhphi}_{1(1)} a_{(1)} \overline{\qhf}_1 \qhS(\qhqr_{2} \overline{\qhphi}_3)
      \otimes \qhano{\qhqr_{1(2,1)} \qhphi_2}
      \overline{\qhphi}_{1(2)} a_{(2)} \overline{\qhf}_2
      \qhS(\qhano{\qhqr_{1(2,2)} \qhphi_3} \overline{\qhphi}_2) \\
    {}^{\eqref{eq:q-Hopf-def-1}}
    & = \qhphi_1 \qhqr_{1(1,1)} \overline{\qhphi}_{1(1)} a_{(1)} \overline{\qhf}_1 \qhS(\qhqr_{2} \overline{\qhphi}_3)
      \otimes \qhphi_2 \qhqr_{1(1,2)} \overline{\qhphi}_{1(2)} a_{(2)} \overline{\qhf}_2 \qhS(\qhphi_3 \qhqr_{1(2)} \overline{\qhphi}_2) \\
    & = \qhphi_1 (\qhano{\qhqr_{1(1)} \overline{\qhphi}_{1}} a)_{(1)}
      \overline{\qhf}_1 \qhS(\qhano{\qhqr_{2} \overline{\qhphi}_3})
      \otimes \qhphi_2 (\qhano{\qhqr_{1(1)} \overline{\qhphi}_{1}} a)_{(2)}
      \overline{\qhf}_2 \qhS(\qhphi_3 \qhano{\qhqr_{1(2)} \overline{\qhphi}_2}) \\
      {}^{\eqref{eq:q-Hopf-qR-phi}}
    & = \qhphi_1 (\qhphi'_{1} a)_{(1)} \overline{\qhf}_1 \qhS(\overline{\qhS}(\qhphi'_3) \qhqr_{2} \qhphi'_{2(2)})
      \otimes \qhphi_2 (\qhphi'_{1} a)_{(2)} \overline{\qhf}_2 \qhS(\qhphi_3 \qhqr_{1} \qhphi'_{2(1)}) \\
    & = (\text{the right-hand side of~\eqref{eq:BCP-alg-coaction-1}}).
  \end{align*}
  Equation~\eqref{eq:BCP-alg-mult} is proved as follows: For $a, b \in \mathbf{A}$,
  \begin{align*}
    a \star b
    & = R^{(2)}_{\unitobj, \unitobj}(a \otimes b) \\
    {}^{\eqref{eq:induction-R-2}}
    & = \qhphi_1 \overline{\qhphi}_{1(1)} \qhqr_{1(1)} \overline{\qhphi}{}'_{1(1,1)} a \qhS(\qhqr_2 \overline{\qhphi}{}'_{1(2)})
      \overline{\qhphi}{}'_{2} \qhphi{}'_1 b
      \qhS(\overline{\qhphi}_{3} \overline{\qhphi}{}'_{3(2)} \qhphi{}'_3) \\
    & \qquad \qquad
      \qhepsilon(\qhphi_2 \overline{\qhphi}_{1(2)} \qhqr_{1(2)} \overline{\qhphi}{}'_{1(1,2)})
      \qhepsilon(\qhphi_3 \overline{\qhphi}_2 \overline{\qhphi}{}'_{3(1)} \qhphi{}'_2) \\
    & = \qhqr_{1} \overline{\qhphi}{}'_{1(1)} a \qhS(\qhqr_2 \overline{\qhphi}{}'_{1(2)})
      \overline{\qhphi}{}'_{2} b
      \qhS(\overline{\qhphi}{}'_{3}) \\
    {}^{\eqref{eq:q-Hopf-def-qR}}
    & = \qhano{\qhphi_{1} \overline{\qhphi}{}'_{1(1)}}
      a \qhS(\qhano{\qhphi_2 \overline{\qhphi}{}'_{1(2)}}) \qhalpha
      \qhano{\qhphi_3 \overline{\qhphi}{}'_{2}} b
      \qhS(\qhano{\overline{\qhphi}{}'_{3}}) \\
    {}^{\eqref{eq:q-Hopf-def-3}}
    & = \overline{\qhphi}{}'_1 \qhphi_1 a \qhS(\overline{\qhphi}{}'_{2(1)} \overline{\qhphi}_1 \qhphi_2) \qhalpha
      \overline{\qhphi}{}'_{2(2)} \overline{\qhphi}_2 \qhphi_{3(1)} b \qhS(\overline{\qhphi}{}'_3 \overline{\qhphi}_3 \qhphi_{3(2)}) \\
    & = \qhano{\overline{\qhphi}{}'_1} \qhphi_1 a \qhS(\overline{\qhphi}_1 \qhphi_2)
      \qhano{\qhS(\overline{\qhphi}{}'_{2(1)}) \qhalpha
      \overline{\qhphi}{}'_{2(2)}} \overline{\qhphi}_2 \qhphi_{3(1)}
      b \qhS(\qhano{\overline{\qhphi}{}'_3} \overline{\qhphi}_3 \qhphi_{3(2)}) \\
    {}^{\eqref{eq:q-Hopf-def-4}, \eqref{eq:q-Hopf-def-5}}
    & = (\text{the right-hand side of~\eqref{eq:BCP-alg-mult}}).
  \end{align*}
  The unit of $\mathbf{A}$ is $R^{(0)}: \unitobj \to R(\unitobj)$. Thus \eqref{eq:BCP-alg-unit} follows from \eqref{eq:induction-R-0}.
\end{proof}

\subsection{Remarks on the adjoint algebra}

Let $\mathcal{C}$ be a monoidal category such that the forgetful functor $U_{\mathcal{C}}: \mathcal{Z}(\mathcal{C}) \to \mathcal{C}$ admits a right adjoint $R_{\mathcal{C}}: \mathcal{C} \to \mathcal{Z}(\mathcal{C})$. The adjunction $U_{\mathcal{C}} \dashv R_{\mathcal{C}}$ and the algebra $\mathbf{A}_{\mathcal{C}} := R_{\mathcal{C}}(\unitobj)$ are important in the study of tensor categories and related areas \cite{MR3039775,MR3631720,MR3632104,MR3921367}.

If we identify ${}_H^H \YD$ with $\mathcal{Z}({}_H \Mod)$, then our $F$, $R$ and $\mathbf{A}$ are identified with $U_{\mathcal{C}}$, $R_{\mathcal{C}}$ and $\mathbf{A}_{\mathcal{C}}$ for $\mathcal{C} = {}_H \Mod$, respectively.
Thus it is worth to investigate $\mathbf{A}$ in detail and hence we give some remarks on $\mathbf{A}$ in this subsection.
The content of this subsection will not be mentioned in later sections, where we discuss categorical aspects of (co)integrals of quasi-Hopf algebras.

\subsubsection{Quantum commutativity of $\mathbf{A}$}

An algebra $A$ in a braided monoidal category $\mathcal{B}$ is said to be {\em quantum commutative}, or {\em commutative} for short, if the equation $m \circ \sigma_{A,A} = m$ holds, where $m$ is the multiplication of $A$ and $\sigma$ is the braiding of $\mathcal{B}$.

We have mentioned that $\mathbf{A}$ is identical to the algebra $H_0$ of \cite{MR2106925}. Thus, according to \cite[Proposition 4.2]{MR2106925}, the algebra $\mathbf{A}$ is commutative in the above sense. We now demonstrate that this fact follows from a general result in the theory of monoidal categories as follows:

\begin{theorem}
  \label{thm:BCP-alg-1}
  The algebra $\mathbf{A}$ is a commutative algebra in ${}^H_H \YD$.
\end{theorem}
\begin{proof}
  Let $\mathcal{C}$ be a monoidal category, and let $\mathbf{A}_{\mathcal{C}} \in \mathcal{Z}(\mathcal{C})$ be as in the beginning of this subsection. It is known that the algebra $\mathbf{A}_{\mathcal{C}}$ is commutative (a proof for fusion categories is found in \cite[Lemma 3.5]{MR3039775}, but the same proof can be applied for the general case). The proof is done by applying this result to the monoidal category $\mathcal{C} = {}_H \Mod$ and identifying $\mathcal{Z}(\mathcal{C})$ with ${}^H_H \YD$.
\end{proof}

\subsubsection{Class functions on $H$}

The vector space $\CF(H) := \Hom_H(\mathbf{A}, \unitobj)$ is called the {\em space of class functions} on $H$ as it coincides with the space of class functions in the usual sense when $H$ is a group algebra \cite{MR3631720}. By the argument of \cite{MR3631720}, the vector space $\CF(H)$ has a structure of an algebra that is isomorphic to the endomorphism algebra of $\mathbf{A} \in {}^H_H \YD$. By using our explicit description of the unit and the counit of the adjunction $F \dashv R$, we give the following expression of the algebra structure of $\CF(H)$:

\begin{theorem}
  \label{thm:BCP-alg-3}
  The vector space $\CF(H)$ is an algebra with respect to the multiplication $\star$ given by
  \begin{equation*}
    \langle \xi \star \zeta, a \rangle
    = \langle \xi, \qhphi_1 \overline{\qhphi}{}_{1(1)} a_{(1)} \overline{\qhf}{}_1 \qhS(\overline{\qhphi}{}_3) \rangle
    \langle \zeta, \qhphi_2 \overline{\qhphi}{}_{1(2)} a_{(2)} \overline{\qhf}{}_2 \qhS(\qhphi_2 \overline{\qhphi}{}_2) \rangle
  \end{equation*}
  for $\xi, \zeta \in \CF(H)$ and $a \in \mathbf{A}$. The algebra $\CF(H)$ is isomorphic to the endomorphism algebra of the Yetter-Drinfeld $H$-module $\mathbf{A}$.
\end{theorem}
\begin{proof}
  We first establish an isomorphism ${}^H_H \YD(\mathbf{A}, \mathbf{A}) \cong \CF(H)$ of vector spaces. Let $\eta$ be the unit of the adjunction $F \dashv R$ given by Theorem \ref{thm:Rad-ind-2nd}. Then the adjunction isomorphism of $F \dashv R$ is given by
  \begin{equation*}
    \Hom_H(F(M), X) \cong {}^H_H \YD(M, R(X)),
    \quad \xi \mapsto R(\xi) \circ \eta_M
  \end{equation*}
for $X \in {}_H \Mod$ and $M \in {}^H_H \YD$. Letting $X = \unitobj$ and $M = \mathbf{A}$ and identifying $R(\unitobj)$ with $\mathbf{A}$, we obtain the following isomorphism of vector spaces:
\begin{equation}
  \label{eq:CF-End-A-iso}
  \Psi: \CF(H) \to {}^H_H \YD(\mathbf{A}, \mathbf{A}),
  \quad \xi \mapsto R(\xi) \circ \eta_{\mathbf{A}}.
\end{equation}
Next, we show the following equation:
\begin{equation}
  \label{eq:CF-End-A-iso-2}
  \xi \star \zeta = \xi \circ F R(\zeta) \circ F(\eta_{\mathbf{A}})
  \quad (\xi, \zeta \in \CF(H)).
\end{equation}
Indeed, for all $a \in \mathbf{A}$, we have
\begin{align*}
  & (\xi \circ F R(\zeta) \circ F(\eta_{\mathbf{A}}))(a) \\
  {}^{\eqref{eq:Rad-ind-2nd-unit}} & = (\xi \otimes \zeta) \YDBdelta_{\mathbf{A}}(a) \\
  {}^{\eqref{eq:BCP-alg-coaction-2}} & = 
  \Big \langle \xi, \qhphi_1 \overline{\qhphi}_{1(1)} a_{(1)} \overline{\qhf}_1 \qhS(\overline{\qhphi}_3)
  \langle \zeta, \qhphi_2 \overline{\qhphi}_{1(2)} a_{(2)} \overline{\qhf}_2 \qhS(\qhphi_3 \overline{\qhphi}_2)  \rangle \Big \rangle \\
  & = \langle \xi \star \zeta, a \rangle.
\end{align*}
Since the right-hand side of \eqref{eq:CF-End-A-iso-2} is a composition of $H$-linear maps, $\xi \star \zeta$ belongs to $\CF(H)$. Finally, we have
\begin{align*}
  \Psi(\xi) \circ \Psi(\zeta)
  & = R(\xi) \circ \eta_{\mathbf{A}} \circ \Psi(\zeta) \\
  & = R(\xi) \circ R F(\Psi(\zeta)) \circ\eta_{\mathbf{A}} \\
  & = R(\xi) \circ R F (R(\zeta) \circ \eta_{\mathbf{A}}) \circ \eta_{\mathbf{A}} \\
  & = R(\xi \circ F R(\zeta) \circ F(\eta_{\mathbf{A}})) \circ  \eta_{\mathbf{A}} 
  = \Psi(\xi \star \zeta)
\end{align*}
for $\xi, \zeta \in \CF(H)$, where the second equality follows from the naturality of $\eta$. Since ${}^H_H \YD(\mathbf{A}, \mathbf{A})$ is an associative unital algebra by the composition, and since the map $\Psi$ is an isomorphism of vector spaces, $\CF(H)$ is an algebra with respect to $\star$ and is isomorphic to ${}^H_H \YD(\mathbf{A}, \mathbf{A})$ as an algebra through the map $\Psi$.
\end{proof}

\subsubsection{The category of $\mathbf{A}$-modules}

The adjoint algebra $\mathbf{A}$ acts on an object of the form $R(V)$, $V \in {}_H \Mod$, from the right by $R^{(2)}_{V, \unitobj}: R(V) \otimes \mathbf{A} \to R(V)$. We denote by $R(V)_{\mathbf{A}}$ the right $\mathbf{A}$-module in ${}^H_H \YD$ obtained in this way. This construction gives rise to a functor
\begin{equation*}
  K: {}_H \Mod \to {}^H_H \YD_{\mathbf{A}},
  \quad V \mapsto R(V)_{\mathbf{A}}
\end{equation*}
from ${}_H \Mod$ to the category ${}^H_H \YD_{\mathbf{A}}$ of right $\mathbf{A}$-modules in ${}^{H}_H \YD$. Here we aim to prove the following theorem:

\begin{theorem}
  \label{thm:BCP-alg-2}
  The functor $K$ is a category equivalence.
\end{theorem}

We explain that this theorem is in fact a consequence of a general result in the theory of Hopf monads (see \cite{MR2793022}).
Let $\mathcal{C}$ and $\mathcal{D}$ be monoidal categories, and let $U: \mathcal{C} \to \mathcal{D}$ be a strong monoidal functor admitting a right adjoint $T: \mathcal{D} \to \mathcal{C}$. The left Hopf operator $\mathbb{H}^{(\ell)}$ and the right Hopf operator $\mathbb{H}^{(r)}$ for the monoidal adjunction $U \dashv T$ are the natural transformations defined by
\begin{equation*}
  \mathbb{H}^{(\ell)}_{X, M}
  = T^{(2)}_{X, U(M)} \circ (\id_{T(X)} \otimes i_M)
  \quad \text{and} \quad
  \mathbb{H}^{(r)}_{X, M}
  = T^{(2)}_{U(M), X} \circ (i_M \otimes \id_{T(X)}),
\end{equation*}
respectively, for $X \in \mathcal{D}$ and $M \in \mathcal{C}$, where $i: \id_{\mathcal{C}} \to T U$ is the unit of $U \dashv T$. We say that the monoidal adjunction $U \dashv T$ is {\em co-Hopf} if the Hopf operators $\mathbb{H}^{(\ell)}$ and $\mathbb{H}^{(r)}$ are invertible. We note that $U \dashv T$ is co-Hopf if and only if the comonoidal adjunction
$(T^{\op} : \mathcal{D}^{\op} \to \mathcal{C}^{\op}) \dashv (U^{\op}: \mathcal{C}^{\op} \to \mathcal{D}^{\op})$
is a Hopf adjunction in the sense of \cite{MR2793022}. Thus any result on Hopf adjunctions can be translated into a result on co-Hopf adjunctions.

We apply the fundamental theorem for Hopf modules over a Hopf monad \cite[Theorem 6.11]{MR2793022} to the Hopf monad arising from the monoidal adjunction $F \dashv R$ of Theorem \ref{thm:Rad-ind-2nd}. We first prove the following theorem, which is of independent interest:

\begin{theorem}
  \label{thm:YD-co-Hopf-adj}
  The monoidal adjunction $F \dashv R$ is co-Hopf.
\end{theorem}
\begin{proof}
  \allowdisplaybreaks
  We recall that ${}_H \Mod_H^{H}$ is defined to be the category of right $H$-comodules in $({}_H \Mod_H, \hatotimes, \unitobj)$. Let ${}_H^{H} \Mod_H^H$ be the category of $H$-bicomodules in ${}_H \Mod_H$, and let $U: {}_H^{H} \Mod_H^H \to {}_H^{} \Mod_H^{H}$ be the forgetful functor. Since ${}^H_H \Mod_H^H$ is identified with the category of left $H$-comodules in the left ${}_H \Mod_H$-module category ${}_H^{} \Mod_H^{H}$, a right adjoint of $U$ is given by the free left $H$-comodule functor
  \begin{equation*}
    T: {}_H^{} \Mod_H^H \to {}_H^{H} \Mod_H^H,
    \quad M \mapsto H \hatotimes M.
  \end{equation*}
  We recall that ${}_H^H \YD = {}_H^H \YD_2$ is defined to be the category of $H$-comodules in the ${}_H \Mod_H$-module category ${}_H \Mod$. Since the functor \eqref{eq:q-Hopf-bimod-equiv-1} is an equivalence of ${}_H \Mod_H$-module categories, it induces an equivalence
  \begin{equation}
    \label{eq:q-Hopf-YD-HHMHH-equiv-1}
    {}_H^H \YD \to {}_H^H \Mod_H^H,
    \quad V \mapsto V \hatotimes H
  \end{equation}
  of categories \cite[Theorem 5.3]{MR1897403}. By the construction of this equivalence, it is trivial that the following diagram (of functors) is commutative:
  \begin{equation}
    \label{eq:q-Hopf-YD-HHMHH-equiv-2}
    \xymatrix@C=64pt{
      {}^H_H \YD \ar[r]_{\approx}^{\text{\eqref{eq:q-Hopf-YD-HHMHH-equiv-1}}} \ar[d]_{F}
      & {}^H_H \Mod^H_H \ar[d]^{U} \\
      {}_H \Mod \ar[r]_{\approx}^{\text{\eqref{eq:q-Hopf-bimod-equiv-1}}}
      & {}_H^{} \Mod^H_H \\
    }
  \end{equation}
  If $M, N \in {}_H \Mod_H^H$, then their tensor product $M \otimes_H N$ over $H$ is an object of ${}_H^{} \Mod^H_H$ by the usual $H$-bimodule structure and the coaction given by
  \begin{equation*}
    m \otimes_H n \mapsto (m_{(0)} \otimes_H n_{(0)}) \hatotimes m_{(1)} n_{(1)}
    \quad (m \in M, n \in N).
  \end{equation*}
  The category ${}_H^{} \Mod_H^H$ is a monoidal category with respect to this tensor product \cite{1999math4164H}. If, moreover, $M$ and $N$ belong to ${}_H^{H} \Mod_H^H$, then $M \otimes_H N \in {}_H \Mod_H^H$ is an object of ${}^H_H \Mod_H^H$ by the left coaction given by
  \begin{equation*}
     m \otimes_H n \mapsto m_{(-1)} n_{(-1)}  \hatotimes (m_{(0)} \otimes_H n_{(0)})
      \quad (m \in M, n \in N).
  \end{equation*}
  The category ${}_H^{H} \Mod_H^H$ is also a monoidal category with respect to this tensor product \cite{MR1897403}. The equivalence~\eqref{eq:q-Hopf-bimod-equiv-1} is in fact an equivalence of monoidal categories \cite{1999math4164H}. The equivalence \eqref{eq:q-Hopf-YD-HHMHH-equiv-1} is also a monoidal equivalence such that \eqref{eq:q-Hopf-YD-HHMHH-equiv-2} is in fact a commutative diagram of monoidal functors \cite[Section 9.3]{MR3929714}.

  By the above argument, the co-Hopfness of $F \dashv R$ is equivalent to that of $U \dashv T$. Since our expression of the monoidal structure of $R$ is quite complicated (see Lemma \ref{lem:R-monoidal-struc}), we shall show that $U \dashv T$ is co-Hopf. We first note that the unit $\nu$ and the counit $\varphi$ of $U \dashv T$ are given as follows:
  \begin{gather*}
    \nu_M: M \to T U(M),
    \quad m \mapsto m_{(-1)} \hatotimes m_{(0)}
    \quad (m \in M \in {}_H^H \Mod_H^H), \\
    \varphi_X: U T(X) \to X,
    \quad h \hatotimes x \mapsto \qhepsilon(h) x
    \quad (x \in X \in {}^{}_H \Mod_H^H, h \in H).
  \end{gather*}
  The monoidal structure $T^{(2)}$ and Hopf operators are given as follows:
  \begin{align*}
    T^{(2)}_{X, Y}: T(X) \otimes_H T(Y) & \to T(X \hatotimes Y), \\
    (h \otimes x) \otimes_H (h' \otimes y)
    & \mapsto h h' \otimes x \otimes y, \\
    \mathbb{H}_{X,M}^{(\ell)}: (H \hatotimes X) \otimes_H M & \to H \hatotimes (X \otimes_H U(M)), \\
    (h \hatotimes x) \otimes_H m & \mapsto h m_{(-1)} \hatotimes (x \otimes_H m_{(0)}), \\
    \mathbb{H}^{(r)}_{M,X}: M \otimes_H (H \hatotimes X) & \to H \hatotimes (U(M)  \otimes_H X), \\
    m \otimes_H (h \hatotimes x) & \mapsto m_{(-1)} h \hatotimes (m_{(0)} \otimes_H x) \phantom{,}
  \end{align*}
  for $h, h' \in H$, $X, Y \in {}_H^{} \Mod_{H}^{H}$, $M \in {}_{H}^{H} \Mod^{H}_{H}$, $x \in X$, $y \in Y$ and $m \in M$. To give the inverse of the left Hopf operator, we first define the linear map
  \begin{align*}
    \widetilde{\mathbb{G}}_{X,M}^{(\ell)}: H \otimes X \otimes M & \to (H \hatotimes X) \otimes_H M, \\
    h \otimes x \otimes m & \mapsto (h \qhS(m_{(-1)} \qhpl_1) \qhql_1 \hatotimes x \qhql_2) \otimes_H m_{(0)} \qhpl_2.
  \end{align*}
  For $h, h' \in H$, $x \in X$ and $m \in M$, we compute
  \begin{align*}
    \widetilde{\mathbb{G}}_{X,M}^{(\ell)} (h \otimes x \otimes h' m)
    & = (h \qhS(h'_{(1)} m_{(-1)} \qhpl_1) \qhql_1 \hatotimes x \qhql_2) \otimes_H h'_{(2)} m_{(0)} \qhpl_2 \\
    & = (h \qhS(m_{(-1)} \qhpl_1) \qhano{\qhS(h'_{(1)}) \qhql_1 h'_{(2,1)}}
      \hatotimes x \qhano{\qhql_2 h'_{(2,2)}}) \otimes_H m_{(0)} \qhpl_2 \\
    {}^{\eqref{eq:q-Hopf-qL}}
    & = (h \qhS(m_{(-1)} \qhpl_1) \qhql_1 \hatotimes x h' \qhql_2) \otimes_H m_{(0)} \qhpl_2 \\
    & = \widetilde{\mathbb{G}}_{X,M}^{(\ell)}(h \otimes x h' \otimes m).
  \end{align*}
  Hence the following linear map is well-defined:
  \begin{align*}
    \mathbb{G}^{(\ell)}_{X,M}: H \hatotimes (X \otimes_H M) & \to (H \hatotimes X) \otimes_H M, \\
    h \otimes (x \otimes_H m) & \mapsto \widetilde{\mathbb{G}}^{(\ell)}_{X,M}(h \otimes x \otimes m).
  \end{align*}
  The map $\mathbb{G}^{(\ell)} := \mathbb{G}_{X,M}^{(\ell)}$ is the inverse of $\mathbb{H}^{(\ell)} := \mathbb{H}_{X,M}^{(\ell)}$. To see this, we remark that the equation
  \begin{equation*}
    m_{(-1)} \otimes m_{(0,-1)} \otimes m_{(0,0)}
    = \qhphi_1 m_{(-1,1)} \overline{\qhphi}_1 \otimes \qhphi_2 m_{(-1,2)} \overline{\qhphi}_2 \otimes \qhphi_3 m_{(0)} \overline{\qhphi}_3
  \end{equation*}
  holds for all $m \in M$. By using this equation instead of \eqref{eq:q-Hopf-def-1}, one can verify that the equations
  \begin{gather}
    \label{eq:q-Hopf-bimod-pL}
    m_{(0,-1)} \qhpl_1 \overline{\qhS}(m_{(-1)}) \otimes m_{(0,0)} \qhpl_2
    = \qhpl_1 \otimes \qhpl_2 m, \\
    \label{eq:q-Hopf-bimod-qL}
    \qhS(m_{(-1)}) \qhql_1 m_{(0,-1)} \otimes m_{(0,0)} \qhql_{2(2)}
    = \qhql_1 \otimes m \qhql_2
  \end{gather}
  hold for all $m \in M$ in a similar way as \eqref{eq:q-Hopf-pL} and \eqref{eq:q-Hopf-qL}. Now we prove that $\mathbb{G}^{(\ell)}$ is the inverse of $\mathbb{H}^{(\ell)}$ as follows: For $h \in H$, $x \in X$ and $m \in M$,
  \begin{align*}
    & \mathbb{G}^{(\ell)} \mathbb{H}^{(\ell)} ((h \hatotimes x) \otimes_H m) \\
    & = (h m_{(-1)} \qhS(m_{(0,-1)} \qhpl_1) \qhql_1 \hatotimes x \qhql_2) \otimes_H m_{(0,0)} \qhpl_2 \\
    & = (h \qhS(\qhano{m_{(0,-1)} \qhpl_1 \overline{\qhS}(m_{(-1)})})
      \qhql_1 \hatotimes x \qhql_2) \otimes_H \qhano{m_{(0,0)} \qhpl_2} \\
    {}^{\eqref{eq:q-Hopf-bimod-pL}}
    & = (h \qhS(\qhpl_1) \qhql_1 \hatotimes x \qhql_2) \otimes_H \qhano{\qhpl_2} m \\
    & = (h \qhano{\qhS(\qhpl_1) \qhql_1 \qhpl_{2(1)}}
      \hatotimes x \qhano{\qhql_2 \qhpl_{2(2)}}) \otimes_H m
    \stackrel{\eqref{eq:q-Hopf-pL-qL-2}}{=}
    (h \hatotimes x) \otimes_H m, \\
    & \mathbb{H}^{(\ell)} \mathbb{G}^{(\ell)} (h \hatotimes (x \otimes_H m)) \\
    & = h \qhS(m_{(-1)} \qhpl_1) \qhql_1 m_{(0,-1)} \qhpl_{2(1)} \hatotimes (x \qhql_2 \otimes_H m_{(0,0)} \qhpl_{2(2)}) \\
    & = h \qhS(\qhpl_1) \qhano{\qhS(m_{(-1)}) \qhql_1 m_{(0,-1)}} \qhpl_{2(1)}
      \hatotimes (x \otimes_H \qhano{\qhql_2 m_{(0,0)}} \qhpl_{2(2)}) \\
    {}^{\eqref{eq:q-Hopf-bimod-qL}}
    & = h \qhano{\qhS(\qhpl_1) \qhql_1 \qhpl_{2(1)}}
      \hatotimes (x \otimes_H m \qhano{\qhql_2 \qhpl_{2(2)}})
      \stackrel{\eqref{eq:q-Hopf-pL-qL-2}}{=}
      h \hatotimes (x \otimes_H m).
  \end{align*}
  In a similar manner, one can verify that the linear map
  \begin{align*}
    \mathbb{G}^{(r)}_{M,X}: H \hatotimes (U(M)  \otimes_H X)
    & \to M \otimes_H (H \hatotimes X), \\
    h \hatotimes (m \otimes_H x) & \mapsto
    \qhql_2 m_{(0)} \otimes_H (\qhpl_1 \overline{\qhS}(\qhql_1 m_{(-1)}) h \hatotimes \qhpl_2 x)
  \end{align*}
  is well-defined and is the inverse of $\mathbb{H}^{(r)}_{M, X}$. The proof is done.
\end{proof}

\begin{proof}[Proof of Theorem~\ref{thm:BCP-alg-2}]
  We set $T = F \circ R$ and regard it as an endofunctor on $\mathcal{C} := ({}_H \Mod)^{\op,\rev}$. By Theorem \ref{thm:YD-co-Hopf-adj}, the comonoidal adjunction $R^{\op} \dashv F^{\op}$ is Hopf, and hence the functor $T$ has a natural structure of a Hopf monad on $\mathcal{C}$ \cite[Proposition 2.14]{MR2793022}. The category of $T$-modules and the category of left Hopf $T$-modules (see \cite[Subsection 6.5]{MR2793022} for the definition) are identified with $({}^H_H \YD)^{\op}$ and $\mathcal{H} := (({}^H_H \YD)_{\mathbf{A}})^{\op}$, respectively. The functor appearing (i) of \cite[Theorem 6.11]{MR2793022} agrees with the functor $K^{\op}: \mathcal{C} \to \mathcal{H}$. Since $\mathcal{C}$ is an abelian category, and since $T$ is exact and faithful, it is easy to check that the condition (ii) of \cite[Theorem 6.11]{MR2793022} holds. Thus, by \cite[Theorem 6.11]{MR2793022}, we see that $K^{\op}$ is a category equivalence. Hence so is $K$.
\end{proof}

\section{Categorical cointegrals}
\label{sec:categori-coint}

\subsection{Induction to the Yetter-Drinfeld category, II}

Let $H$ be a finite-\hspace{0pt}dimensional quasi-Hopf algebra. A {\em categorical cointegral} for a finite tensor category has been introduced in \cite{MR3632104,MR3921367}. The aim of this section is to investigate relations between cointegrals on $H$ and categorical cointegrals of the finite tensor category ${}_H \Mod_{\fd}$ of finite-dimensional left $H$-modules. Provided that $H$ is unimodular, the algebra $\mathbf{A} \in {}^H_H \YD$ of Theorem \ref{thm:BCP-alg}
 is known to be a Frobenius algebra in ${}^H_H \YD$ \cite{MR3632104}. We also describe the Frobenius structure of $\mathbf{A}$ in terms of integrals and cointegrals of $H$.

As in the previous section, we denote by $F: {}^H_H \YD \to {}_H \Mod$ the forgetful functor. The functor $F$ has a right adjoint $R$ given in Theorem \ref{thm:Rad-ind-2nd} (without the finiteness assumption on $H$). We now construct a left adjoint $L$ of $F$ as follows: Given $V \in {}_H \Mod$, we define $L(V) \in {}_H \Mod$ by $L(V) = H^{\vee} \ogreaterthan V$. We recall that $H^{\vee}$ is a left quasi-Hopf bimodule over $H$. Thus $L(V)$ is a Yetter-Drinfeld module over $H$ by the second kind coaction
\begin{align*}
\YDBdelta_{L(V)} := \Big(
  L(V) = H^{\vee} \ogreaterthan V
  & \xrightarrow{\makebox[6em]{$\scriptsize \delta \ogreaterthan \id_V$}}
  (H \hatotimes H^{\vee}) \ogreaterthan V \\
  & \xrightarrow{\makebox[6em]{$\Omega_{H,H^{\vee},V}$}}
  H \ogreaterthan (H^{\vee} \ogreaterthan V)
  = H \ogreaterthan L(V) \Big),
\end{align*}
where $\delta: H^{\vee} \to H \hatotimes H^{\vee}$ is the left coaction of $H$ given by Lemma~\ref{lem:q-Hopf-bimod-dual} (i) with $M = H$ and the natural isomorphism $\Omega$ is given by   \eqref{eq:def-Omega}. Although we will not need explicit formulas of the (co)action of $H$ on $L(V)$ in this paper, it could be worth to include them:

\begin{proposition}
We fix a basis $\{ h_i \}$ of $H$ and let $\{ h^i \}$ be the dual basis of $H^{\vee}$. The action and the second kind coaction of $H$ on $L(V)$ are given by
\begin{align*}
  h \triangleright (\xi \otimes v)
  & = (\overline{\qhS}(h_{(2)}) \rightharpoonup \xi \leftharpoonup \qhS(h_{(1,1)})) \otimes h_{(1,2)} v, \\
  \YDBdelta_{L(V)}(\xi \otimes v)
  & = \langle \xi, \qhVL_2 h_{i(2)} \qhUL_2 \rangle \, \qhomega_1 \qhVL_1 h_{i(1)} \qhUL_1 \qhomega_5 \\
  & \qquad \qquad \otimes (\overline{\qhS}(\qhomega_4) \rightharpoonup h^i \leftharpoonup \qhS(\qhomega_2)) \otimes \qhomega_3 v
\end{align*}
for $h \in H$, $\xi \in H^{\vee}$ and $v \in V$, where we have used the Einstein convention.
\end{proposition}
\begin{proof}
  The formula for the action follows from the definition of $\ogreaterthan$. We compute the second kind coaction as follows:
  \begin{align*}
    & \YDBdelta_{L(V)}(\xi \otimes v)
    = \Omega_{H,H^{\vee},V}(\delta \ogreaterthan \id_V)(\xi \otimes v) \\
    {}^{\eqref{eq:H-bimod-left-dual-coactions}}
    & = \Omega_{H,H^{\vee},V}(\langle \xi, \qhVR_1 h_{i(1)} \qhUR_1 \rangle \, \qhVR_2 h_{i(2)} \qhUR_2 \otimes h^i \otimes v) \\
    {}^{\eqref{eq:def-Omega}}
    & = \langle \xi, \qhVR_1 h_{i(1)} \qhUR_1 \rangle
    \, \qhomega_1 \qhVR_2 h_{i(2)} \qhUR_2 \qhomega_5 \otimes \qhomega_2 h^i \qhomega_4 \otimes \qhomega_3 v \\
    {}^{\eqref{eq:H-bimod-left-dual-actions}}
    & = \langle \xi, \qhVR_1 h_{i(1)} \qhUR_1 \rangle
    \, \qhomega_1 \qhVR_2 h_{i(2)} \qhUR_2 \qhomega_5
    \otimes (\overline{\qhS}(\qhomega_4) \rightharpoonup h^i \leftharpoonup \qhS(\qhomega_2))
    \otimes \qhomega_3 v. \qedhere
  \end{align*}
\end{proof}

Now we prove:

\begin{theorem}
  \label{thm:left-adj}
  \textup{(i)} The assignment $V \mapsto L(V)$ extends to a functor from ${}_H \Mod$ to ${}^H_H \YD$. This functor is left adjoint to the forgetful functor $F$ with the unit $\eta'$ and the counit $\varepsilon'$ given by
  \begin{gather}
    \label{eq:left-adj-unit}
    \eta'_V: V \to F L(V),
    \quad v \mapsto \qhepsilon \otimes v
    \quad (v \in V \in {}_H \Mod), \\
    \label{eq:left-adj-counit}
    \varepsilon'_M: L F(M) \to M,
    \quad \xi \otimes m \mapsto \langle \xi, \mathbbm{t}_1 m_{\YDB{-1}} \mathbbm{t}_3 \rangle \mathbbm{t}_2 m_{\YDB{0}}
    \quad (m \in M \in {}^{H}_H \YD),
  \end{gather}
  where
  \begin{equation*}
    \mathbbm{t} = \qhS(\overline{\qhphi}_1) \qhql_1 \overline{\qhphi}_{2(1)}
    \otimes \qhqr_1 \qhql_2 \overline{\qhphi}_{2(2)}
    \otimes \qhS(\qhqr_2 \overline{\qhphi}_3).
  \end{equation*}
  \textup{(ii)} For $V \in {}_H \Mod$, there is a natural isomorphism
  \begin{equation}
    \label{eq:left-adj-iso}
    R(\cointr \otimes V) \to L(V),
    \quad a \otimes (\lambda \otimes v)
    \mapsto (\overline{\qhS}(\overline{\qhphi}_1 a) \qhmu(\overline{\qhphi}_2) \rightharpoonup \lambda)
    \otimes \overline{\qhphi}_3 v
  \end{equation}
  of Yetter-Drinfeld $H$-modules.
\end{theorem}
\begin{proof}
  (i) The $H$-bimodule $H^{\vee}$ is an algebra in the monoidal category $({}_H \Mod_H, \hatotimes, \unitobj)$ as the left dual object of $H$. A Yetter-Drinfeld module $M$ is naturally a left $H^{\vee}$-module in ${}_H \Mod$ by the action
  \begin{align*}
    a_M := \Big(
    H^{\vee} \ogreaterthan M
    & \xrightarrow{\makebox[6em]{\scriptsize $\id_{H^{\vee}} \ogreaterthan \YDBdelta_{M}$}}
      H^{\vee} \ogreaterthan (H \ogreaterthan M) \\
    & \xrightarrow{\makebox[6em]{\scriptsize $(\Omega_{H^{\vee}, H, M})^{-1}$}}
      (H^{\vee} \hatotimes H) \ogreaterthan M
      \xrightarrow{\makebox[6em]{\scriptsize $\eval_H \ogreaterthan \id_M$}}
    M \Big),
  \end{align*}
  and the category ${}^H_H \YD$ is identified with the category ${}_{H^{\vee}}({}_H \Mod)$ of left $H^{\vee}$-modules in ${}_H \Mod$. Under this identification, a left adjoint of the forgetful functor $F$ is given by the free $H^{\vee}$-module functor, that is, the functor $L$. The unit of $L \dashv F$ is given by the unit of $H^{\vee}$, that is, the counit of $H$. Thus $\eta'$ is given as stated. The counit of $L \dashv F$ is given by $\varepsilon'_M = a_M$ for $M \in {}^H_H \YD$, where $a_M$ is given in the above. To verify the expression for $a_M$, we introduce the element $\tilde{\qhomega} \in H^{\otimes 5}$ defined by
  \begin{equation}
    \label{eq:ome-til}
    \tilde{\qhomega} = 
    \overline{\qhphi}_1 \overline{\chi}_1 \otimes \overline{\qhphi}_2 \overline{\chi}_{2(1)}
    \otimes \overline{\qhphi}_3 \overline{\chi}_{2(2)}
    \otimes \qhS(\overline{\chi}_3) \qhf_2 \otimes \qhS(\overline{\chi}_4) \qhf_1,
  \end{equation}
  where $\chi \in H^{\otimes 4}$ is given by~\eqref{eq:q-Hopf-def-chi}. The element $\tilde{\qhomega}$ is actually the inverse of the element $\qhomega \in H \otimes H \otimes H \otimes H^{\op} \otimes H^{\op}$. Hence the inverse of the natural isomorphism $\Omega$ is given by
  \begin{align*}
    \Omega_{M,N,V}^{-1}:
    M \ogreaterthan (N \ogreaterthan V)
    & \to (M \hatotimes N) \ogreaterthan V, \\
    m \otimes (n \otimes v)
      & \mapsto (\tilde{\qhomega}_1 m \tilde{\qhomega}_5 \otimes \tilde{\qhomega}_2 n \tilde{\qhomega}_4) \otimes \tilde{\qhomega}_3 v
  \end{align*}
  for $M, N \in {}_H \Mod_H$, $V \in {}_H \Mod$, $m \in M$, $n \in N$ and $v \in V$. Now let $M$ be a Yetter-Drinfeld $H$-module. For $\xi \in H^{\vee}$ and $m \in M$, we have:
  \begin{align*}
    a_M(\xi \otimes m)
    & = \eval_H(\tilde{\qhomega}_1 \xi \tilde{\qhomega}_5 \otimes \tilde{\qhomega}_2 m_{\YDB{-1}} \tilde{\qhomega}_4) \tilde{\qhomega}_3 m_{\YDB{0}} \\
    {}^{\eqref{eq:H-bimod-left-dual-actions}, \eqref{eq:H-bimod-left-dual-eval}}
    & = \langle \overline{\qhS}(\tilde{\qhomega}_5) \rightharpoonup \xi \leftharpoonup \qhS(\tilde{\qhomega}_1),
      \qhalpha \tilde{\qhomega}_2 m_{\YDB{-1}} \tilde{\qhomega}_4 \overline{\qhS}(\qhbeta) \rangle \tilde{\qhomega}_3 m_{\YDB{0}} \\
    & = \langle \xi, \mathbbm{t}'_1 m_{\YDB{-1}} \mathbbm{t}'_3 \rangle \mathbbm{t}'_2 m_{\YDB{0}},
  \end{align*}
  where $\mathbbm{t}' = \qhS(\tilde{\qhomega}_1) \qhalpha \tilde{\qhomega}_2 \otimes \tilde{\qhomega}_3 \otimes \tilde{\qhomega}_4 \overline{\qhS}(\qhbeta) \overline{\qhS}(\tilde{\qhomega}_5)$. We prove \eqref{eq:left-adj-counit} by showing $\mathbbm{t}' = \mathbbm{t}$ as follows:
  \begin{align*}
    \mathbbm{t}'
    \ {}^{\eqref{eq:ome-til}}
    & = \qhS(\overline{\qhphi}_1 \overline{\chi}_1) \qhalpha \overline{\qhphi}_2 \overline{\chi}_{2(1)}
      \otimes \overline{\qhphi}_3 \overline{\chi}_{2(2)}
      \otimes \qhS(\overline{\chi}_3) \qhano{\qhf_2 \overline{\qhS}(\qhf_1 \qhbeta)}
      \overline{\chi}_4 \\
    {}^{\eqref{eq:q-Hopf-f-4}}
    & = \qhS(\overline{\qhphi}_1 \overline{\chi}_1) \qhalpha \overline{\qhphi}_2 \overline{\chi}_{2(1)}
      \otimes \overline{\qhphi}_3 \overline{\chi}_{2(2)}
      \otimes \qhS(\overline{\chi}_3) \qhalpha \overline{\chi}_4 \\
    {}^{\eqref{eq:q-Hopf-def-chi}}
    & = \qhS(\qhano{\overline{\qhphi}_1 \qhphi_{1(1)}} \overline{\qhphi}{}'_{1})
      \qhalpha \qhano{\overline{\qhphi}_2 \qhphi_{1(2,1)}} \overline{\qhphi}{}'_{2(1)}
      \otimes \qhano{\overline{\qhphi}_3 \qhphi_{1(2,2)}} \overline{\qhphi}{}'_{2(2)}
      \otimes \qhS(\qhphi_2 \overline{\qhphi}{}'_{3}) \qhalpha \qhphi_3 \\
    {}^{\eqref{eq:q-Hopf-def-1}}
    & = \qhano{\qhS(\qhphi_{1(1,1)}} \overline{\qhphi}_1 \overline{\qhphi}{}'_{1})
      \qhano{\qhalpha \qhphi_{1(1,2)}} \overline{\qhphi}_2 \overline{\qhphi}{}'_{2(1)}
      \otimes \qhphi_{1(2)} \overline{\qhphi}_3 \overline{\qhphi}{}'_{2(2)}
      \otimes \qhS(\qhphi_2 \overline{\qhphi}{}'_{3}) \qhalpha \qhphi_3 \\
    {}^{\eqref{eq:q-Hopf-def-5}}
    & = \qhano{\qhS(\overline{\qhphi}_1} \overline{\qhphi}{}'_{1})
      \qhano{\qhalpha \overline{\qhphi}_2} \overline{\qhphi}{}'_{2(1)}
      \otimes \qhano{\qhphi_{1} \overline{\qhphi}_3} \overline{\qhphi}{}'_{2(2)}
      \otimes \qhS(\qhano{\qhphi_2} \overline{\qhphi}{}'_{3}) \qhano{\qhalpha \qhphi_3} \\
    {}^{\eqref{eq:q-Hopf-def-qR}, \eqref{eq:q-Hopf-def-qL}}
    & = \qhS(\overline{\qhphi}_1) \qhql_1 \overline{\qhphi}_{2(1)}
    \otimes \qhqr_1 \qhql_2 \overline{\qhphi}_{2(2)}
    \otimes \qhS(\qhqr_2 \overline{\qhphi}_3)
    = \mathbbm{t}.
  \end{align*}
  (ii) By the fundamental theorem for quasi-Hopf bimodules, the map
  \begin{equation*}
    \Xi_{\qhR}: H \hatotimes \cointr \to H^{\vee},
    \quad h \otimes \lambda \mapsto \overline{\qhS}(h) \rightharpoonup \lambda
  \end{equation*}
  is an isomorphism of left quasi-Hopf bimodules. Thus the map
  \begin{align*}
    \xi_V :=
    \Big( R(\cointr \otimes V)
    = H \ogreaterthan (\cointr \ogreaterthan V)
    & \xrightarrow{\makebox[6em]{\scriptsize $\Omega^{-1}_{H, \cointr, V}$}}
    (H \hatotimes \cointr) \ogreaterthan V \\
    & \xrightarrow{\makebox[6em]{\scriptsize $\Xi_{\qhR} \ogreaterthan \id_{V}$}}
    H^{\vee} \ogreaterthan V = L(V) \Big)
  \end{align*}
  is an isomorphism of Yetter-Drinfeld $H$-modules that is natural in the variable $V \in {}_H \Mod$. The map $\xi_V$ actually coincides with the map \eqref{eq:left-adj-iso}. Indeed, for $a \in H$, $\lambda \in \cointr$ and $v \in V$, we have
  \begin{align*}
    \xi_V(a \otimes (\lambda \otimes v))
    & = \Xi_{\qhR}(\tilde{\qhomega}_1 a \tilde{\qhomega}_5 \otimes \tilde{\qhomega}_{2} \lambda \tilde{\qhomega}_4) \otimes \tilde{\qhomega}_3 v \\
    & = \qhmu(\tilde{\qhomega}_{2} ) \qhepsilon(\tilde{\qhomega}_{4})
      (\overline{\qhS}(\tilde{\qhomega}_1 a \tilde{\qhomega}_5) \rightharpoonup \lambda) \otimes \tilde{\qhomega}_3 v \\
    & = (\overline{\qhS}(\overline{\qhphi}_1 a) \qhmu(\overline{\qhphi}_2) \rightharpoonup \lambda)
    \otimes \overline{\qhphi}_3 v.
  \end{align*}
  Here, the last equality follows from the following computation:
  \begin{align*}
    & (\id_H \otimes \id_H \otimes \id_H \otimes \qhepsilon \otimes \id_H)(\tilde{\qhomega}) \\
    & = \overline{\qhphi}_1 \overline{\chi}_1 \otimes \overline{\qhphi}_2 \overline{\chi}_{2(1)}
      \otimes \overline{\qhphi}_3 \overline{\chi}_{2(2)}
      \otimes \qhepsilon(\qhS(\overline{\chi}_3) \qhf_2) \otimes \qhS(\overline{\chi}_4) \qhf_1 \\
    & = \overline{\qhphi}_1 \otimes \overline{\qhphi}_2
      \otimes \overline{\qhphi}_3 \otimes 1_k \otimes 1_H. \qedhere
  \end{align*}
\end{proof}

\subsection{Categorical cointegrals of ${}_H \Mod$}

Let $H$ be a finite-dimensional quasi-Hopf algebra. Given an algebra map $\gamma: H \to k$, we define the Yetter-Drinfeld $H$-module $\mathbf{A}_{\gamma}$ as follows: As a vector space $\mathbf{A}_{\gamma} = H$. The action and the coaction (of the second kind) of $H$ are given by the following formulas:
\begin{equation*}
  h \triangleright a
  = h_{(1, 1)} a \qhS(h_{(2)}) \gamma(h_{(1, 2)}), \quad
  \YDBdelta_{\mathbf{A}_{\gamma}} (a)
  = \qhomega_1 a_{(1)} \qhomega_5 \otimes \qhomega_2 a_{(2)} \qhomega_4 \gamma(\qhomega_3)
\end{equation*}
for $a \in \mathbf{A}_{\gamma}$ and $h \in H$. We note that $\mathbf{A}_{\gamma}$ is isomorphic to $R(\gamma)$ as a Yetter-Drinfeld $H$-module, where $\gamma$ is regarded as a one-dimensional left $H$-module. Now we introduce the following definition:

\begin{definition}
  \label{def:categori-coint}
  We call $\catcoint := {}^H_{H} \YD(\mathbf{A}_{\qhmu}, \unitobj)$ the space of {\em categorical cointegrals}, where $\qhmu$ is the modular function on $H$.
\end{definition}

We first explain why this definition agrees with the original definition of categorical cointegrals introduced in \cite{MR3921367}. Let $\mathcal{C}$ be a finite tensor category, and let $\mathcal{Z}(\mathcal{C})$ be the Drinfeld center of $\mathcal{C}$. The forgetful functor $U_{\mathcal{C}}: \mathcal{Z}(\mathcal{C}) \to \mathcal{C}$ has a right adjoint, say $R_{\mathcal{C}}$. According to \cite[Lemma 4.1]{MR3921367}, there is a unique (up to isomorphisms) simple object $\mu_{\mathcal{C}} \in \mathcal{C}$ such that
\begin{equation*}
  \mathbf{I}_{\mathcal{C}} := \Hom_{\mathcal{Z}(\mathcal{C})}(R_{\mathcal{C}}(\mu_{\mathcal{C}}), \unitobj) \ne 0.
\end{equation*}
We call $\mathbf{I}_{\mathcal{C}}$ the space of {\em categorical cointegrals} \cite[Definition 4.3]{MR3921367} (see also the first equation in the proof of \cite[Theorem 4.8]{MR3921367}). Now we consider the case where $\mathcal{C} = {}_H \Mod_{\fd}$. Then $\mathcal{Z}(\mathcal{C})$ is identified with the category ${}^H_H \YD_{\fd}$ of finite-dimensional Yetter-Drinfeld $H$-modules and, under this identification, the restriction of the functor $R: {}_H \Mod \to {}^H_H \YD$ is right adjoint to $U_{\mathcal{C}}$. Moreover, since $\cointr \cong \qhmu$ as left $H$-modules, we have
\begin{equation}
  \label{eq:cat-coint-one-dim}
  \catcoint
  \cong {}^H_H \YD(R(\qhmu), \unitobj)
  \cong {}^H_H \YD(L(\unitobj), \unitobj)
  \cong \Hom_H(\unitobj, \unitobj) \cong k \ne 0
\end{equation}
by Theorem~\ref{thm:left-adj}. Thus $\mu_{\mathcal{C}} = \qhmu$ and categorical cointegrals of Definition~\ref{def:categori-coint} can be identified with categorical cointegrals of \cite{MR3921367}.

If $H$ is an ordinary Hopf algebra, then $\catcoint$ is identical to the space of {left} cointegrals on $H$ \cite[Example 4.4]{MR3921367} and hence it is spanned by $\lambda \circ \overline{\qhS}$, where $\lambda$ is a non-zero {right} cointegral on $H$. We now establish an analogous result for the quasi-Hopf algebra $H$ as follows:

\begin{theorem}
  \label{thm:q-Hopf-cat-coint}
  Let $\lambda$ be a non-zero right cointegral on $H$, and define
  \begin{equation}
    \label{eq:lambda-cat}
    \catlambda: \mathbf{A}_{\qhmu} \to \unitobj,
    \quad \catlambda(a) = \langle \lambda, \qhalpha \overline{\qhS}(a) \rangle
    \quad (a \in \mathbf{A}_{\qhmu}).
  \end{equation}
  Then $\catcoint$ is the one-dimensional vector space spanned by $\catlambda$.
\end{theorem}
\begin{proof}
  We have used an isomorphism $\mathbf{A}_{\qhmu} \cong R(\cointr)$ to establish \eqref{eq:cat-coint-one-dim}. We recall that $\cointr$ is the one-dimensional left $H$-module spanned by $\lambda$ and is isomorphic to $\qhmu$ as a left $H$-module. Thus we have an isomorphism
  \begin{equation}
    \label{eq:iso-A-mu-R-cointr}
    \mathbf{A}_{\qhmu} \to R(\cointr), \quad a \mapsto a \otimes \lambda
  \end{equation}
  of Yetter-Drinfeld $H$-modules. The element of $\catcoint$ corresponding to $1 \in k$ through the isomorphism \eqref{eq:cat-coint-one-dim} is given by
  \begin{equation*}
    l := \Big( \mathbf{A}_{\qhmu} \xrightarrow{\  \eqref{eq:iso-A-mu-R-cointr} \ }
    R(\cointr)
    \xrightarrow{\ \cong \ }
    R(\cointr \otimes \unitobj)
    \xrightarrow{\  \eqref{eq:left-adj-iso} \ }
    L(\unitobj)
    \xrightarrow{\  \varepsilon'_{\unitobj} \ } \unitobj \Big),
  \end{equation*}
  where $\varepsilon'$ is the counit of $L \dashv F$ given by~\eqref{eq:left-adj-counit}. To prove this theorem, it suffices to show that $l$ coincides with the map $\catlambda$ given by \eqref{eq:lambda-cat}. We first give a formula for $\varepsilon_{\unitobj}'$. Since $\YDBdelta_{\unitobj}(1) = \qhbeta \otimes 1$, we have
  \begin{gather*}
      \varepsilon_{\unitobj}'(\xi \otimes 1)
      \mathop{=}^{\eqref{eq:left-adj-counit}}
      \langle \xi,
        \qhS(\overline{\qhphi}_1) \qhano{\qhql_1} \overline{\qhphi}_{2(1)}
        \qhbeta \qhS(\qhano{\qhqr_2} \overline{\qhphi}_3) \rangle
        \qhepsilon(\qhano{\qhqr_1 \qhql_2} \overline{\qhphi}_{2(2)}) \\
      \mathop{=}^{\eqref{eq:q-Hopf-def-qL}, \eqref{eq:q-Hopf-def-qR}}
      \langle \xi,
      \qhano{\qhS(\overline{\qhphi}_1) \qhalpha \overline{\qhphi}_{2}
      \qhbeta \qhS(\overline{\qhphi}_3)} \qhalpha \rangle
      \mathop{=}^{\eqref{eq:q-Hopf-def-6}}
      \xi(\qhalpha)
    \end{gather*}
    for $\xi \in H^{\vee}$. Thus, for $a \in \mathbf{A}_{\qhmu}$, we have
  \begin{align*}
    l(a)
    & = \varepsilon_{\unitobj}'((\overline{\qhS}(\overline{\qhphi}_1 a) \qhmu(\overline{\qhphi}_2) \rightharpoonup \lambda) \otimes \qhepsilon(\overline{\qhphi}_3)) \\
    {}^{\eqref{eq:q-Hopf-phi-eps}}
    & = \varepsilon_{\unitobj}'((\overline{\qhS}(a) \rightharpoonup \lambda) \otimes 1)
    = \langle \overline{\qhS}(a) \rightharpoonup \lambda, \qhalpha \rangle = \catlambda(a). \qedhere
  \end{align*}
\end{proof}

\begin{remark}
\label{rem:left-right-coint}
  (a) Unlike the case of Hopf algebras, the map $\catlambda$ in the above theorem is not a left cointegral on $H$ in general. However, one can give a basis of $\catcoint$ by a non-zero left cointegral on $H$ as follows:
  By \cite[Proposition 4.3]{MR2862216}, the element $u := \qhmu(\qhVL_1) \qhS^2(\qhVL_2)$ is invertible and the map
  \begin{equation*}
    \mathbf{s}: \cointl \to \cointr, \quad \lambda^{\mathsf{L}} \mapsto (\lambda^{\mathsf{L}} \circ \overline{\qhS}) \leftharpoonup u^{-1}
    \quad (\lambda^{\mathsf{L}} \in \cointl)
  \end{equation*}
  is bijective. We fix a non-zero left cointegral $\lambda^{\mathsf{L}}$ on $H$ and set $\lambda' = \mathbf{s}(\lambda^{\mathsf{L}})$. Then $\lambda'$ is a non-zero right cointegral on $H$ and hence, by the above theorem,  $\catcoint$ is spanned by $\lambda'' := (\lambda' \leftharpoonup \qhalpha) \circ \overline{\qhS}$. Explicitly, this linear form is given by
  \begin{equation*}
    \langle \lambda'', a \rangle
    = \langle \lambda^{\mathsf{L}} \overline{\qhS}, u^{-1} \qhalpha \overline{\qhS}(a) \rangle
    = \langle \lambda^{\mathsf{L}}, \overline{\qhS}{}^2(a) \overline{\qhS}(u^{-1} \qhalpha) \rangle
    \quad (a \in \mathbf{A}_{\qhmu}).
  \end{equation*}
  (b) Suppose that $H$ is unimodular, that is, $\qhmu = \qhepsilon$. Let $\lambda$ be a non-zero right cointegral on $H$. By the unimodularity and Theorem~\ref{thm:left-coint}, we have
  \begin{equation}
    \label{eq:Nakayama-auto-unimo-case}
    \lambda(h h') = \lambda(h' \qhS^2(h))
  \end{equation}
  for all $h, h' \in H$. In particular, we have $\lambda \circ \qhS^2 = \lambda$. By the unimodularity, we also have $u = 1$ and therefore $\lambda \circ \overline{\qhS}$ ($= \lambda \circ \qhS$) is in fact a left cointegral on $H$ ({\it cf}. \cite[Corollary 4.4]{MR2862216}). Hence we obtain the following conclusion: Let $\lambda^{\mathsf{L}}$ be a non-zero left cointegral on $H$. Then $\catcoint$ is spanned by the map
  \begin{equation*}
    \mathbf{A}_{\qhmu} = \mathbf{A} \to \unitobj, \quad a \mapsto \langle \lambda^{\mathsf{L}}, a \qhS(\qhalpha) \rangle
    \quad (a \in \mathbf{A}_{\qhmu}).
  \end{equation*}
\end{remark}

\subsection{Frobenius structure of the adjoint algebra}

Let $H$ be a finite-dimensional unimodular Hopf algebra. Ishii and Masuoka observed that the adjoint algebra $\mathbf{A}$ of Theorem~\ref{thm:BCP-alg} is in fact a Frobenius algebra in ${}^H_H \YD$ and use this fact to construct an invariant of handlebody-knots \cite{MR3265394}. By the result of \cite{MR3632104}, the algebra $\mathbf{A}$ is still Frobenius even in the case where $H$ is a finite-dimensional unimodular quasi-Hopf algebra. However, an explicit Frobenius structure is not yet known. We answer this problem as follows:

\begin{theorem}
  \label{thm:Ishii-Masuoka-q-Hopf}
  Suppose that $H$ is a finite-dimensional unimodular quasi-Hopf algebra. We fix a non-zero right cointegral $\lambda$ on $H$, and let $\Lambda$ be the left integral in $H$ such that $\langle \lambda, \Lambda \rangle = 1$. Then the algebra $\mathbf{A} \in {}^H_H \YD$ of Theorem~\ref{thm:BCP-alg} is a self-dual object in ${}^H_H \YD$ with the evaluation morphism $e: \mathbf{A} \otimes \mathbf{A} \to \unitobj$ and the coevaluation morphism $d: \unitobj \to \mathbf{A} \otimes \mathbf{A}$ given by
  \begin{equation*}
    e(a \otimes b) = \langle \catlambda, a \star b \rangle
    \quad \text{and} \quad
    d(1) = \upbeta \triangleright \qhS(\overline{\qhf}_2 \qhql_1 \Lambda_{(1)} \qhpl_1) \otimes
     \qhql_2 \Lambda_{(2)} \qhpl_2 \overline{\qhf}_1,
  \end{equation*}
  respectively, for $a, b \in \mathbf{A}$, where $\catlambda = (\lambda \leftharpoonup \qhalpha) \circ \overline{\qhS}$. In particular, the algebra $\mathbf{A}$ is a commutative Frobenius algebra in ${}^H_H \YD$ with Frobenius form $\catlambda: \mathbf{A} \to \unitobj$.
\end{theorem}
\begin{proof}
  The monoidal category ${}^H_H \YD_{\fd}$ is rigid as it is isomorphic to the Drinfeld center of ${}_H \Mod_{\fd}$. The left dual object of $M \in {}^H_H \YD_{\fd}$ is, as a left $H$-module, identical to $M^{\vee}$. By Theorem~\ref{thm:q-Hopf-cat-coint}, the map $\catlambda$ belongs to ${}^H_H \YD(\mathbf{A}, \unitobj)$. Thus the map $e$ is a morphism in ${}^H_H \YD$. Since $\mathbf{A}$ is finite-dimensional, it has a left dual object in ${}^H_H \YD$. We now consider the morphism
  \begin{equation*}
    \widetilde{\Theta} := \Big( \mathbf{A} \xrightarrow{\  \id_{\mathbf{A}} \otimes \coev_{\mathbf{A}} \ } \mathbf{A} \otimes (\mathbf{A} \otimes \mathbf{A}\!^{\vee})
    \xrightarrow{\  \Phi^{-1}_{\mathbf{A},\mathbf{A},\mathbf{A}^{\vee}} \ } (\mathbf{A} \otimes \mathbf{A}) \otimes \mathbf{A}\!^{\vee}
    \xrightarrow{\  e \otimes \id_{\mathbf{A}^{\vee}} \ } \mathbf{A}\!^{\vee} \Big)
  \end{equation*}
  of Yetter-Drinfeld modules. We choose a basis $\{ a_i \}$ of $\mathbf{A}$ and let $\{ a^i \}$ be the dual basis. Using the Einstein convention, we compute
  \begin{align*}
    \langle \widetilde{\Theta}(h), h' \rangle
    & = \langle \catlambda, (\overline{\qhphi}_1 \triangleright h) \star (\overline{\qhphi}_2 \qhbeta \triangleright a_i) \rangle \langle \overline{\qhphi}_3 \triangleright a^i, h' \rangle \\
    & = \langle \catlambda, (\overline{\qhphi}_1 \triangleright h) \star (\overline{\qhphi}_2 \qhbeta \triangleright a_i) \rangle \langle a^i, \qhS(\overline{\qhphi}_3) \triangleright h' \rangle \\
    & = \langle \catlambda, (\overline{\qhphi}_1 \triangleright h) \star (\overline{\qhphi}_2 \qhbeta \qhS(\overline{\qhphi}_3) \triangleright h') \rangle \\
    {}^{\eqref{eq:q-Hopf-def-qR}, \eqref{eq:lambda-cat}}
    & = \langle \lambda \overline{\qhS}, ((\qhpr_1 \triangleright h) \star (\qhpr_2 \triangleright h')) \qhS(\qhalpha) \rangle \\
    {}^{\eqref{eq:BCP-alg-mult}}
    & = \langle \lambda \overline{\qhS}, \qhphi_1 (\qhano{\qhpr_1 \triangleright h})
      \qhS(\overline{\qhphi}_1 \qhphi_2) \qhalpha \overline{\qhphi}_2 \qhphi_{3(1)}
      (\qhano{\qhpr_2 \triangleright h'}) \qhS(\overline{\qhphi}_3 \qhphi_{3(2)}) \qhS(\qhalpha) \rangle \\
    {}^{\eqref{eq:BCP-alg-action}}
    & = \langle \lambda \overline{\qhS}, \qhphi_1 \qhpr_{1(1)} h \qhS(\qhpr_{1(2)}) \qhS(\overline{\qhphi}_1 \qhphi_2) \qhalpha \overline{\qhphi}_2 \qhphi_{3(1)} \qhpr_{2(1)} h'
      \qhano{\qhS(\qhpr_{2(2)}) \qhS(\overline{\qhphi}_3 \qhphi_{3(2)}) \qhS(\qhalpha)} \rangle \\
    {}^{\eqref{eq:Nakayama-auto-unimo-case}}
    & = \langle \lambda \overline{\qhS}, \overline{\qhS}(\qhalpha \overline{\qhphi}_3 \qhphi_{3(2)} \qhpr_{2(2)}) \qhphi_1 \qhpr_{1(1)} h \qhS(\qhpr_{1(2)}) \qhS(\overline{\qhphi}_1 \qhphi_2) \qhalpha \overline{\qhphi}_2 \qhphi_{3(1)} \qhpr_{2(1)} h' \rangle \\
    & = \langle \lambda \overline{\qhS}, \overline{\qhS}(\qhano{\qhS(\qhphi_1 \qhpr_{1(1)}) \qhalpha \overline{\qhphi}_3 \qhphi_{3(2)} \qhpr_{2(2)}}) h \qhano{\qhS(\overline{\qhphi}_1 \qhphi_2 \qhpr_{1(2)}) \qhalpha \overline{\qhphi}_2 \qhphi_{3(1)} \qhpr_{2(1)}} h' \rangle \\
    {}^{\eqref{eq:q-Hopf-def-f}}
    & = \langle \lambda \overline{\qhS}, \overline{\qhS}(\qhf_2) h \qhf_1 h' \rangle
    = \langle \lambda, \overline{\qhS}(h') \overline{\qhS}(\overline{\qhS}(\qhf_2) h \qhf_1) \rangle
  \end{align*}
  for $h, h' \in H$. Hence, we finally obtain
  \begin{equation*}
    \widetilde{\Theta}(h) = (\overline{\qhS}(\overline{\qhS}(\qhf_2) h \qhf_1)
    \rightharpoonup \lambda) \circ \overline{\qhS}
    = \Theta_{\qhL}(\overline{\qhS}(\overline{\qhS}(\qhf_2) h \qhf_1)) \circ \overline{\qhS}
  \end{equation*}
  for $h \in H$, where $\Theta_{\qhL}(x) = x \rightharpoonup \lambda$ ($x \in H$). We note that $\lambda$ is a {\em left} cointegral on $H^{\cop}$. Thus, by applying Theorem~\ref{thm:Frobenius-dual-basis} to $H^{\cop}$, we see that the inverse of $\Theta_{\qhL}$ is given by
  \begin{equation*}
    \Theta_{\qhL}^{-1}(\xi) = \overline{\qhS}(\widetilde{\Lambda}_1) \langle \xi, \widetilde{\Lambda}_2 \rangle
  \end{equation*}
  for $\xi \in H^{*}$, where $\widetilde{\Lambda} = \qhql_2 \Lambda_{(2)} \qhpl_2 \otimes \qhql_1 \Lambda_{(1)} \qhpl_1$. By \eqref{eq:Lambda-tilde-H} applied to $H^{\cop}$, we have
  \begin{equation}
    \label{eq:Lambda-tilde-H-cop}
    h \widetilde{\Lambda}_1 \otimes \widetilde{\Lambda}_2
    = \widetilde{\Lambda}_1 \otimes \qhS(h) \widetilde{\Lambda}_2
  \end{equation}
  for all $h \in H$. Since $\Theta_{\qhL}$, $\qhS$ and $\qhf$ are invertible, $\widetilde{\Theta}$ is invertible. Explicitly, the inverse of $\widetilde{\Theta}$ is given by
  \begin{align*}
    \widetilde{\Theta}{}^{-1}(\xi)
    & = \overline{\qhS}(\overline{\qhf}_2)
      \cdot \qhS(\Theta_{\qhL}^{-1}(\xi \circ \qhS)) \cdot \overline{\qhf}_1 \\
    & = \overline{\qhS}(\overline{\qhf}_2)
      \cdot \qhS(\overline{\qhS}(\widetilde{\Lambda}_1)
      \langle \xi \circ \qhS, \widetilde{\Lambda}_2 \rangle) \cdot \overline{\qhf}_1 \\
    & = \overline{\qhS}(\overline{\qhf}_2) \widetilde{\Lambda}_1 \overline{\qhf}_1
      \langle \xi, \qhS(\widetilde{\Lambda}_2) \rangle \\
      {}^{\eqref{eq:Lambda-tilde-H-cop}}
    & = \widetilde{\Lambda}_1 \overline{\qhf}_1
      \langle \xi, \qhS(\overline{\qhf}_2 \widetilde{\Lambda}_2) \rangle
  \end{align*}
  for $\xi \in \mathbf{A}\!^{\vee}$. Thus $\mathbf{A}$ is a self-dual object with evaluation $e$ and the coevaluation
  \begin{align*}
    (\id_{\mathbf{A}} \otimes \widetilde{\Theta}^{-1}) \coev_{\mathbf{A}}(1)
    & = \upbeta \triangleright a_i \otimes \widetilde{\Theta}^{-1}(a^i) \\
    & = \upbeta \triangleright a_i \otimes
      \widetilde{\Lambda}_1 \overline{\qhf}_1
      \langle a^i, \qhS(\overline{\qhf}_2 \widetilde{\Lambda}_2) \rangle \\
    & = \upbeta \triangleright \qhS(\overline{\qhf}_2 \widetilde{\Lambda}_2) \otimes
      \widetilde{\Lambda}_1 \overline{\qhf}_1.
      \qedhere
  \end{align*}
\end{proof}

\section{Twisted module trace}
\label{sec:modified-trace}

\subsection{Twisted module trace on tensor ideals}
\label{subsec:modified-trace}

Let $\mathcal{C}$ be a rigid monoidal category, and let $X^{\vee}$ denote the left dual object of $X \in \mathcal{C}$. Then the assignment $X \mapsto X^{\vee \vee} := (X^{\vee})^{\vee}$ canonically extends to a monoidal autoequivalence on $\mathcal{C}$. A {\em pivotal structure} of $\mathcal{C}$ is an isomorphism of monoidal functors from $\id_{\mathcal{C}}$ to $(-)^{\vee\vee}$. A {\em pivotal monoidal category} is a rigid monoidal category equipped with a pivotal structure. Now let $\mathcal{C}$ be a pivotal monoidal category with pivotal structure $\mathfrak{j}: \id_{\mathcal{C}} \to (-)^{\vee\vee}$. For objects $V, W, X \in \mathcal{C}$, the (right) {\em partial pivotal trace} over $X$ is the map
\begin{gather*}
  \ptrace^{\mathcal{C}}_{V, W | X}: \Hom_{\mathcal{C}}(V \otimes X, W \otimes X) \to \Hom_{\mathcal{C}}(V, W), \\
  f \mapsto (\id_W \otimes \eval_{X^{\vee}}(\mathfrak{j}_{X} \otimes \id_{X^{\vee}})) \Phi_{W,X,X^{\vee}} (f \otimes \id_{X^{\vee}}) \Phi_{V,X,X^{\vee}}^{-1} (\id_V \otimes \coev_X).
\end{gather*}
Graphically, the partial pivotal trace is expressed as follows:
\begin{equation*}
  \ptrace^{\mathcal{C}}_{V,W|X}
  \left( \xy /r1pc/:
    (0,2)="T1", (0,-2)="T0",
    "T1"; "T0" **\dir{} ?(.5) *+!{\makebox[3em]{$f$}} *\frm{-}="BX1",
    "BX1"!U!L(.6); \makecoord{p}{"T1"} **\dir{-} ?>="L1",
    "BX1"!U!R(.6); \makecoord{p}{"T1"} **\dir{-} ?>="L2",
    "BX1"!D!L(.6); \makecoord{p}{"T0"} **\dir{-} ?>="L3",
    "BX1"!D!R(.6); \makecoord{p}{"T0"} **\dir{-} ?>="L4",
    "L1" *+!D{V}, "L2" *+!D{X},
    "L3" *+!U{W}, "L4" *+!U{X},
    \endxy
  \right)
  = \xy /r1pc/:
  (0,2)="T1", (0,-2)="T0",
  "T1"; "T0" **\dir{} ?(.4) *+!{\makebox[3em]{$f$}} *\frm{-}="BX1",
  "BX1"!U!L(.6); \makecoord{p}{"T1"} **\dir{-} ?>="L1",
  "BX1"!D!L(.6); \makecoord{p}{"T0"} **\dir{-} ?>="L3",
  "BX1"!U!R(.6); p+(2,0) \mycap{.75} ?>="P1",
  "BX1"!D!R(.6); p-(0,.5) **\dir{-}
  ?> *+!U{\mathfrak{j}_X} *\frm{-}="BX2",
  "BX2"!D; p+(2,0) \mycap{-.75}
  ?>; "P1" **\dir{-} ?(.5)="L5",
  "L1" *+!D{V}, "L3" *+!U{W}, "L5" *+!L{X^{\vee}}
  \endxy
\end{equation*}

The partial pivotal trace is widely used to construct invariants of knots and closed 3-manifolds so-called quantum invariants \cite{MR1292673}. Although such a construction of quantum invariants works in a quite general setting, most of interesting quantum invariants originate from semisimple $k$-linear pivotal monoidal categories. In fact, the partial pivotal trace often vanishes in the non-semisimple case. To construct a meaningful invariant from a non-semisimple category, some authors considered a modification of the partial pivotal trace \cite{MR2998839,2018arXiv180100321B,2018arXiv180907991C,2018arXiv180900499G,2018arXiv180901122F}.

We suppose that $\mathcal{C}$ is a $k$-linear pivotal monoidal category with finite-\hspace{0pt}dimensional Hom-spaces. In this paper, we mention a recent result of Fontalvo Orozco and Gainutdinov \cite{2018arXiv180901122F}. They introduced the notion of a {\em module trace} over a right $\mathcal{C}$-module category $\mathcal{M}$ equipped with a $\mathcal{C}$-module endofunctor $\Sigma$. For simplicity, we consider the case where $\mathcal{M}$ is a {\em tensor ideal} of $\mathcal{C}$ in the sense of \cite[Section 2]{MR2998839}. We moreover restrict ourselves to the case where $\Sigma = D \otimes (-)$ for some object $D \in \mathcal{C}$. Then the notion of a module trace \cite{2018arXiv180901122F}, which we call a $D$-twisted module trace, is defined as follows:

\begin{definition}
  \label{def:D-tw-trace}
  {\rm (a)} Let $\mathcal{M}$ be a $k$-linear category, and let $\Sigma: \mathcal{M} \to \mathcal{M}$ be a $k$-linear endofunctor on $\mathcal{M}$. A {\em $\Sigma$-twisted trace} \cite[Definition 2.1]{2016arXiv160503523B} on $\mathcal{M}$ is a family $\mathsf{t}_{\bullet} = \{ \mathsf{t}_{P}: \Hom_{\mathcal{C}}(P, \Sigma(P)) \to k \}_{P \in \mathcal{M}}$ of $k$-linear maps such that the equation
  \begin{equation}
    \label{eq:tw-trace-cyclicity}
    \mathsf{t}_{P} \Big( P \xrightarrow{\ g \ } Q \xrightarrow{\ f \ } \Sigma(P) \Big)
    = \mathsf{t}_{Q} \Big( Q \xrightarrow{\ f \ } \Sigma(P) \xrightarrow{\ \Sigma(g) \ } \Sigma(Q) \Big)
  \end{equation}
  holds for all morphisms $g: P \to Q$ and $f: Q \to \Sigma(P)$ in $\mathcal{M}$. We denote by $\HH_{0}(\mathcal{M}, \Sigma)$ the class of $\Sigma$-twisted traces on $\mathcal{M}$.

  {\rm (b)} Let $\mathcal{I}$ be a tensor ideal of $\mathcal{C}$, and let $D$ be an object of $\mathcal{C}$. By the definition of tensor ideals, we can define the functor $\Sigma: \mathcal{I} \to \mathcal{I}$ by $\Sigma(P) = D \otimes P$ for $P \in \mathcal{I}$. A {\em $D$-twisted module trace on $\mathcal{I}$} is a $\Sigma$-twisted trace $\mathsf{t}_{\bullet}$ on $\mathcal{I}$ with this $\Sigma$ such that the {\it module trace condition}
  \begin{equation}
    \label{eq:module-trace-property}
    \mathsf{t}_{P \otimes V}(f)
    = \mathsf{t}_{P}\Big( \ptrace^{\mathcal{C}}_{P, \Sigma(P)|V}(\Phi_{D, P, V}^{-1} \circ f) \Big)
  \end{equation}
  holds for all objects $P \in \mathcal{I}$, $V \in \mathcal{C}$ and morphisms $f: P \otimes V \to \Sigma(P \otimes V)$.
\end{definition}

One of the main results of \cite{2018arXiv180901122F} classifies $D$-twisted module traces on $\mathcal{I}$ in the case where $\mathcal{C} = {}_H \Mod_{\fd}$ for some finite-dimensional pivotal Hopf algebra $H$, $\mathcal{I}$ is the class of projective objects of $\mathcal{C}$, and $D$ is the one-dimensional left $H$-module associated to the modular function $\qhmu$ on $H$. According to \cite[Corollary 6.1]{2018arXiv180901122F}, the space of such a trace is identified with the space of ``$\qhmu$-symmetrized'' cointegrals on $H$. The aim of this section is to give the same description of such a trace in the case where $H$ is a finite-dimensional quasi-Hopf algebra.

\subsection{Pivotal quasi-Hopf algebras}

We recall that a pivotal Hopf algebra is a Hopf algebra $H$ equipped with a grouplike element $\qhg \in H$ such that $\qhg h \qhg^{-1} = \qhS^2(h)$ for all $h \in H$. If $H$ is a pivotal Hopf algebra, then ${}_H \Mod_{\fd}$ is a pivotal monoidal category by the pivotal structure given by $\qhg$. The definition of a pivotal quasi-Hopf algebra is a little more complicated than the ordinary case because of the fact that the canonical isomorphism $(V \otimes W)^{\vee\vee} \cong V^{\vee\vee} \otimes W^{\vee\vee}$ is non-trivial in the quasi-Hopf case.

\begin{definition}[{{\it cf}. \cite{MR2501184,MR4009567}}]
  Let $H$ be a quasi-Hopf algebra. A {\em pivotal element} of $H$ is an invertible element $\qhg \in H$ such that the equations
  \begin{equation}
    \label{eq:q-Hopf-pivot}
    \qhg h \qhg^{-1} = \qhS^2(h)
    \quad \text{and} \quad
    \Delta(\qhg)
    = \overline{\qhf}_1 \qhS(\qhf_2) \qhg
    \otimes \overline{\qhf}_2 \qhS(\qhf_1) \qhg
  \end{equation}
  hold for all element $h \in H$. A {\em pivotal quasi-Hopf algebra} is a quasi-Hopf algebra equipped with a pivotal element.
\end{definition}

The category $\Vect := {}_{k} \Mod_{\fd}$ of finite-dimensional vector spaces over $k$ is a pivotal monoidal category with the canonical pivotal structure $\mathfrak{j}_V: V \to V^{**}$ given by $\langle \mathfrak{j}_{V}(v), \xi \rangle = \langle \xi, v \rangle$ for $v \in V$ and $\xi \in V^*$. Let $f: V \otimes X \to W \otimes X$ be a morphism in $\Vect$. We write $f(m) = f(m)_W \otimes f(m)_X$ for $m \in V \otimes X$. The partial pivotal trace over $X$ (with respect to the canonical pivotal structure $\mathfrak{j}$) is given by
\begin{equation}
  \label{eq:partial-tr-Vec}
  \ptrace^{\Vect}_{V,W|X}(f)(v) = f(v \otimes x_i)_W \, \langle x^i, f(v \otimes x_i)_X \rangle
\end{equation}
for $v \in V$, where $\{ x_i \}$ is a basis of $X$, $\{ x^i \}$ is the dual basis of $\{ x_i \}$ and the Einstein summation convention is used to suppress the sum.

Let $H$ be a quasi-Hopf algebra. Given a pivotal element $\qhg$ of $H$, we define the natural isomorphism $\mathfrak{g}_V: V \to V^{\vee\vee}$ ($V \in {}_H \Mod_{\fd}$) by $\mathfrak{g}_V(v) = \mathfrak{j}_V(\qhg v)$ for $v \in V$. Then $\mathfrak{g} = \{ \mathfrak{g}_V \}$ is a pivotal structure on ${}_H \Mod_{\fd}$. Moreover, if $H$ is finite-dimensional, then every pivotal structure of ${}_H \Mod_{\fd}$ is obtained in this way from a pivotal element of $H$ \cite[Proposition 4.2]{MR2501184}.

Now we fix a pivotal element $\qhg$ of $H$ and compute the partial pivotal trace in ${}_H \mathcal{M}_{\fd}$ with respect to the pivotal structure associated to $\qhg$. In the following lemma, given left $H$-modules $X
$ and $Y$, we regard $X \otimes Y$ as a left $H \otimes H$-module by $(h \otimes h') (x \otimes y) = h x \otimes h' y$ for $h, h' \in H$, $x \in X$ and $y \in Y$.

\begin{lemma}
  \label{lem:partial-tr-H-mod}
  The partial pivotal trace of $f: V \otimes X \to W \otimes X$ over $X$ in $\mathcal{C} := {}_H \Mod_{\fd}$ is given by
  \begin{equation}
    \label{eq:partial-tr-H-mod-1}
    \ptrace^{\mathcal{C}}_{V,W|X}(f) = \ptrace^{\Vect}_{V,W|X}(\widetilde{f}),
  \end{equation}
  where $\widetilde{f}: V \otimes X \to W \otimes X$ is the linear map defined by
  \begin{equation}
    \label{eq:partial-tr-H-mod-2}
    \widetilde{f}(m)
    = (1 \otimes \qhg)\qhqr f(\qhpr m)
    \quad (m \in V \otimes X).
  \end{equation}
\end{lemma}
\begin{proof}
  We recall that there is a canonical isomorphism
  \begin{equation*}
    \Hom_{H}(V \otimes X, W \otimes X) \cong \Hom_{H}(V, (W \otimes X) \otimes X^{\vee}).
  \end{equation*}
  Let $f^{\sharp}: V \to (W \otimes X) \otimes X^{\vee}$ be the morphism in ${}_H \Mod_{\fd}$ corresponding to $f: V \otimes X \to W \otimes X$ via this isomorphism. Although an explicit formula of $f^{\sharp}$ has been given in Lemma~\ref{lem:q-Hopf-reciprocity-P}, we also express $f^{\sharp}$ by $f^{\sharp}(v) = (f^{\sharp}(v)_W \otimes f^{\sharp}(v)_{X}) \otimes f^{\sharp}(v)_{X^{\vee}}$ for $v \in V$.   We fix a basis $\{ x_i \}$ of $X$ and let $\{ x^i \}$ be the dual basis to $\{ x_i \}$. By the definition of the partial pivotal trace, we compute as follows: For $v \in V$,
  \begin{align*}
    \ptrace^{\mathcal{C}}_{V,W|X}(f)(v)
    & = (\id_W \otimes \eval_{X^{\vee}}(\mathfrak{g}_X \otimes X^{\vee})) \Phi_{W,X,X^{\vee}}^{-1} f^{\sharp}(v) \\
    & = \qhphi_1 f^{\sharp}(v)_{W}
      \langle \mathfrak{g}_X(\qhphi_2 f^{\sharp}(v)_X),
      \qhalpha \qhphi_3 f^{\sharp}(v)_{X^{\vee}} \rangle \\
    & = \qhphi_1 f^{\sharp}(v)_{W}
      \langle f^{\sharp}(v)_{X^{\vee}},
      \qhano{\qhS(\qhalpha \qhphi_3) \qhg} \qhphi_2 f^{\sharp}(v)_X \rangle \\
    {}^{\eqref{eq:q-Hopf-pivot}}
    & = \qhano{\qhphi_1} f^{\sharp}(v)_{W}
      \langle f^{\sharp}(v)_{X^{\vee}},
      \qhg \qhano{\overline{\qhS}(\qhalpha \qhphi_3) \qhphi_2} f^{\sharp}(v)_X \rangle \\
    {}^{\eqref{eq:q-Hopf-def-pR}}
    & = \qhqr_1 f^{\sharp}(v)_{W}
      \langle f^{\sharp}(v)_{X^{\vee}},
      \qhg \qhqr_2 f^{\sharp}(v)_X \rangle \\
    {}^{\eqref{eq:q-Hopf-reciprocity-P1}}
    & = \qhqr_1 f(\qhpr_1 v \otimes \qhpr_2 x_i)_{W}
      \langle x^i, \qhg \qhqr_2 f(\qhpr_1 v \otimes \qhpr_2 x_i)_{X} \rangle \\
    {}^{\eqref{eq:partial-tr-H-mod-2}}
    & = \widetilde{f}(v \otimes x_i)_{W}
      \langle x^i, \widetilde{f}(v \otimes x_i)_{X} \rangle \\
    {}^{\eqref{eq:partial-tr-Vec}}
    & = \ptrace^{\Vect}_{V,W|X}(\widetilde{f})(v). \qedhere
  \end{align*}
\end{proof}

\subsection{Twisted trace on the full subcategory of projective objects}
\label{subsec:tw-mod-tr}

Let $A$ be a finite-dimensional algebra. We denote by ${}_A \Proj_{\fd}$ the full subcategory of projective objects of ${}_A \Mod_{\fd}$. Given an algebra automorphism $\chi: A \to A$ and a left $A$-module $V$, we define $\chi^{*}(V)$ to be the vector space $V$ equipped with the left action $\cdot_{\chi}$ of $A$ given by $a \cdot_{\chi} v = \chi(a) v$ for $a \in H$ and $v \in V$. It is easy to see that the assignment $V \mapsto \chi^{*} (V)$ extends to a $k$-linear autoequivalence on ${}_A \Mod_{\fd}$ preserving the full subcategory ${}_A \Proj_{\fd}$. We say that a linear form $t: A \to k$ is {\em $\chi$-symmetric} if $t(a b) = t(\chi(b) a)$ for all $a, b \in A$. According to \cite{2018arXiv180901122F}, we denote by $\HH_0(A, \chi)$ the space of $\chi$-symmetric linear forms on $A$. Given a $\chi$-symmetric linear form $t$ on $A$ and $P \in {}_A \Proj_{\fd}$, we define the linear map $\mathsf{t}_P: \Hom_{A}(P, \chi^{*}(P)) \to k$ as follows: By the dual basis lemma, there are a natural number $n$ and $A$-linear maps $a_i: A \to P$ and $b_i: P \to A$ ($i = 1, \dotsc, n$) such that $\sum_{i = 1}^n a_i b_i = \id_P$. Choose such a system $\{ a_i, b_i \}_{i = 1}^n$ and define
\begin{equation}
  \label{eq:def-tP-f}
  \mathsf{t}_P(f) = \sum_{i = 1}^n t \Big( \chi^{*}(b_i) f a_i(1) \Big)
\end{equation}
for $f \in \Hom_{A}(P, \chi^{*}(P))$. Then the map $\mathsf{t}_P$ does not depend on the choice of the system $\{ a_i, b_i \}_{i = 1}^n$. The family $\mathsf{t}_{\bullet} = \{ \mathsf{t}_P \}$ of $k$-linear maps is a $\chi^{*}$-twisted trace on ${}_A \Proj_{\fd}$ and this construction gives an isomorphism
\begin{equation*}
  \HH_0(A, \chi) \to \HH_0({}_A \Proj_{\fd}, \chi^{*}), \quad t \mapsto \mathsf{t}_{\bullet}
\end{equation*}
of vector spaces; see \cite{2018arXiv180901122F}.

\subsection{Twisted module trace for quasi-Hopf algebras}

Let $H$ be a finite-dimensional pivotal quasi-Hopf algebra with pivotal element $\qhg$, and let $\chi$ be the algebra automorphism of $H$ given by $\chi(h) = h \leftharpoonup \qhmu$ for $h \in H$, where $\qhmu$ is the modular function on $H$. By abuse of notation, we write
\begin{equation*}
  \HH_0(H, \qhmu)
  := \HH_0(H, \chi)
  \quad \text{and} \quad
  \HH_0({}_H \Proj_{\fd}, \qhmu)
  := \HH_0({}_H \Proj_{\fd}, \chi^{*}).
\end{equation*}
We note that $\mathcal{P} := {}_H \Proj_{\fd}$ is a tensor ideal of ${}_H \Mod_{\fd}$. Again by abuse of notation, we denote by $\qhmu$ the one-dimensional left $H$-module associated to $\qhmu$. Let $\HHm_0(\mathcal{P}, \qhmu)$ be the class of $\qhmu$-twisted module traces on $\mathcal{P}$ in the sense of Definition \ref{def:D-tw-trace} (b). Since the autoequivalence $\chi^{*}$ on $\mathcal{P}$ is identified with $\qhmu \otimes (-)$, the class $\HHm_0(\mathcal{P}, \qhmu)$ is the subspace of $\HH_0(\mathcal{P}, \qhmu)$ consisting of elements satisfying the module trace condition \eqref{eq:module-trace-property}. By the reduction lemma \cite[Lemma 2.9]{2018arXiv180901122F}, to check \eqref{eq:module-trace-property}, it is enough to verify that the equation
\begin{equation}
  \label{eq:tw-trace-progen}
  \mathsf{t}_{H \otimes H}(f)
  = \mathsf{t}_{H}\Big( \ptrace^{\mathcal{C}}_{H, \qhmu \otimes H|H}(\Phi_{\qhmu, H, H}^{-1} \circ f) \Big)
\end{equation}
holds for all $f \in \Hom_{H}(H \otimes H, \qhmu \otimes (H \otimes H))$. Thus we have
\begin{equation*}
  \HHm_0(\mathcal{P}, \qhmu)
  = \{ \mathsf{t}_{\bullet} \in \HH_{0}(\mathcal{P}, \qhmu)
  \mid \text{$\mathsf{t}_{\bullet}$ satisfies~\eqref{eq:tw-trace-progen}} \}.
\end{equation*}

\begin{definition}
  We define
  \begin{equation*}
    \musymint = \{ (\lambda \circ \overline{\qhS}) \leftharpoonup \qhg
    \mid \text{$\lambda$ is a left cointegral on $H$} \}
  \end{equation*}
  and call $\musymint$ the space of {\em $\qhmu$-symmetrized cointegrals on $H$}. 
\end{definition}

Let $\lambda$ be a left cointegral on $H$, and set $t = (\lambda \circ \overline{\qhS}) \leftharpoonup \qhg$. Then we have
\begin{align*}
  \langle t, h h' \rangle
  & = \langle \lambda, \overline{\qhS}(\qhg h h') \rangle
    = \langle \lambda, \qhano{\overline{\qhS}(h')} \overline{\qhS}(\qhg h) \rangle \\
  {}^{\text{(Theorem~\ref{thm:left-coint})}}
  & = \langle \lambda, \overline{\qhS}(\qhg h) \qhS(h' \leftharpoonup \qhmu) \rangle \\
  & = \langle \lambda, \overline{\qhS}(\qhano{\qhS^2(h' \leftharpoonup \qhmu)} \qhg h) \rangle \\
  {}^{\eqref{eq:q-Hopf-pivot}}
  & = \langle \lambda, \overline{\qhS}(\qhg (h' \leftharpoonup \qhmu) h) \rangle
    = \langle t, (h' \leftharpoonup \qhmu) h \rangle
\end{align*}
for all $h, h' \in H$. Thus $\musymint$ is a subspace of $\HH_0(H, \qhmu)$. As we have recalled from \cite{2018arXiv180901122F} in the previous subsection, there is an isomorphism between $\HH_0(H, \qhmu)$ and $\HH_0({}_H \Proj_{\fd}, \qhmu)$. Now the main result of this section is stated as follows:

\begin{theorem}
  \label{thm:q-Hopf-modified-trace}
  The isomorphism $\HH_0(H, \qhmu) \cong \HH_0({}_H \Proj_{\fd}, \qhmu)$ restricts to
  \begin{equation*}
    \musymint \cong \HHm_0({}_H \Proj_{\fd}, \qhmu).
  \end{equation*}
\end{theorem}

We fix an element $t \in \HH_0(H, \qhmu)$ and then define $\mathsf{t}_{\bullet} \in \HH_0({}_H \Proj_{\fd}, \qhmu)$ from $t$ as in the previous subsection. To prove the above theorem, we first derive a necessary and sufficient condition for $\mathsf{t}_{\bullet}$ satisfying \eqref{eq:tw-trace-progen} in terms of the linear form $t$. The following lemma is useful:

\begin{lemma}
  For a left $H$-module $X$, we define the linear map $\theta_X$ by
  \begin{equation}
    \label{eq:q-Hopf-theta}
    \theta_X: H \otimes X_0 \to H \otimes X,
    \quad h \otimes x \mapsto h_{(1)} \qhpr_1 \otimes h_{(2)} \qhpr_2 x
    \quad (h \in H, x \in X),
  \end{equation}
  where $X_0$ is the vector space $X$ regarded as a left $H$-module by $\qhepsilon$. Then the family $\theta = \{ \theta_X \}$ is a natural isomorphism. The inverse of $\theta_X$ is given by
  \begin{equation}
    \label{eq:q-Hopf-theta-inv}
    \theta_X^{-1}(h \otimes x) = \qhqr_1 h_{(1)} \otimes \qhS(\qhqr_2 h_{(2)}) x
    \quad (h \in H, x \in X).
  \end{equation}
\end{lemma}
\begin{proof}
  A left-right switched version is found at \cite[p.3356]{MR1897403}. One can also verify this lemma directly by using equations~\eqref{eq:q-Hopf-pR}, \eqref{eq:q-Hopf-qR}, \eqref{eq:q-Hopf-pR-qR-1} and \eqref{eq:q-Hopf-pR-qR-2}.
\end{proof}

Set $M = \qhmu \otimes (H \otimes H_0)$ for simplicity. We may identify $M$ with the vector space $H \otimes H$ equipped with the left $H$-action given by $h \cdot (x \otimes y) = \langle \qhmu, h_{(1)} \rangle h_{(2)} x \otimes y$ ($h, x, y \in H$). There are isomorphisms
\begin{equation*}
  \Hom_H(H \otimes H_0, M)
  \cong \Hom_k(H, M)
  \cong H^* \otimes M
  \cong H^* \otimes H \otimes H
\end{equation*}
of vector spaces. Let $\Gamma_{\xi,x,y} \in \Hom_H(H \otimes H_0, M)$ be the element corresponding to $\xi \otimes x \otimes y \in H^* \otimes H \otimes H$ via the above isomorphism. Explicitly, it is given by
\begin{equation}
  \label{eq:Gamma-xi-x-y}
  \Gamma_{\xi, x, y}(h \otimes h')
  = \langle \xi, h' \rangle \langle \qhmu, h_{(1)} \rangle h_{(2)} x \otimes y
  \quad (h, h' \in H).
\end{equation}
By the above lemma, we have
\begin{align*}
  & \Hom_H(H \otimes H, \qhmu \otimes (H \otimes H)) \\
  & = \{ (\id_{\qhmu} \otimes \theta_H) \circ f \circ \theta_H^{-1}
    \mid f \in \Hom_H(H \otimes H_0, M) \} \\
  & = \mathrm{span}_k \{ \Gamma_{\xi, x, y}' \mid \xi \in H^*, x, y \in H \},
\end{align*}
where $\Gamma_{\xi,x,y}' := (\id_{\qhmu} \otimes \theta_H) \circ \Gamma_{\xi, x, y} \circ \theta_H^{-1}$. Thus $\mathsf{t}_{\bullet}$ satisfies~\eqref{eq:tw-trace-progen} if and only if the equation
\begin{equation}
  \label{eq:tw-trace-progen-2}
  \mathsf{t}_{H \otimes H}(\Gamma_{\xi,x,y}')
  = \mathsf{t}_{H}\Big( \ptrace^{\mathcal{C}}_{H, \qhmu \otimes H|H}(\Phi_{\qhmu, H, H}^{-1} \circ \Gamma_{\xi,x,y}') \Big)
\end{equation}
holds for all $\xi \in H^*$ and $x, y \in H$. We now compute the left and the right hand sides of~\eqref{eq:tw-trace-progen-2}.

\begin{claim}
  \label{claim:trace-proof-1}
  The left-hand side of \eqref{eq:tw-trace-progen-2} is equal to $t(x) \xi(y)$.
\end{claim}
\begin{proof}
  Let $\{ h_i \}_{i = 1}^n$ be a basis of $H$, and let $\{ h^i \}$ be the dual basis of $\{ h_i \}$. For each $i$, we define linear maps $a_i$ and $b_i$ by
\begin{align*}
  a_i & : H \to H \otimes H_0,
  \quad a_i(h) = h \otimes h_i, \\
  b_i & : H \otimes H_0 \to H,
  \quad b_i(h \otimes h') = \langle h^i, h' \rangle h
\end{align*}
for $h, h' \in H$, respectively. Then $a_i$ and $b_i$ are $H$-linear maps and the equation $\sum_{i = 1}^n a_i b_i = \id_{H\otimes H_0}$ holds. Since $\mathsf{t}_{\bullet}$ satisfies \eqref{eq:tw-trace-cyclicity} with $\Sigma = \qhmu \otimes (-)$, we have
  \begin{gather*}
    \mathsf{t}_{H \otimes H}(\Gamma'_{\xi,x,y})
    = \mathsf{t}_{H \otimes H}(\Sigma(\theta_H) \circ \Gamma_{\xi,x,y} \circ \theta_H^{-1})
    \stackrel{\eqref{eq:tw-trace-cyclicity}}{=} \mathsf{t}_{H \otimes H_0}(\Gamma_{\xi,x,y}) \\
    \stackrel{\eqref{eq:def-tP-f}}{=} \sum_{i = 1}^n t(b_i \Gamma_{\xi,x,y} a_i(1))
    \stackrel{\eqref{eq:Gamma-xi-x-y}}{=} \sum_{i = 1}^n t(x \langle h^i, y \rangle \langle \xi, h_i \rangle)
    = t(x) \xi(y). \qedhere
  \end{gather*}
\end{proof}

\begin{claim}
  \label{claim:trace-proof-2}
  The right-hand side of \eqref{eq:tw-trace-progen-2} is equal to
  \begin{equation}
    \label{eq:tw-trace-progen-2-RHS}
    \qhmu(\overline{\qhphi}_1) \langle \xi, \qhg \qhqr_2 \overline{\qhphi}_3 x_{(2)} \qhpr_2 y \rangle
    \langle t, (\qhqr_1 \leftharpoonup \qhmu) \overline{\qhphi}_2 x_{(1)} \qhpr_1 \rangle.
  \end{equation}
\end{claim}
\begin{proof}
  For simplicity, we set $f := \Phi_{\qhmu, H, H}^{-1} \circ \Gamma'_{\xi,x,y}$ and $r := \ptrace^{\mathcal{C}}_{H, \qhmu \otimes H|H}(f)$. By the defining formula~\eqref{eq:def-tP-f} of $\mathsf{t}_{\bullet}$, the right-hand side of \eqref{eq:tw-trace-progen-2} is equal to $t(r(1))$. To compute $r(1)$, we note:
  \begin{equation*}
    \theta_H^{-1}(\qhpr_1 \otimes \qhpr_2 h)
    \stackrel{\eqref{eq:q-Hopf-theta-inv}}{=}
    \qhqr_1 \qhpr_{1(1)} \otimes \qhS(\qhqr_2 \qhpr_{1(2)}) \qhpr_2 h
    \stackrel{\eqref{eq:q-Hopf-pR-qR-2}}{=} 1 \otimes h
    \quad (h \in H).
  \end{equation*}
  Hence, for all $h \in H$, we have
  \begin{align*}
    f(\qhpr_1 \otimes \qhpr_2 h)
    & = (\Phi_{\qhmu,H,H}^{-1} (\id_{\qhmu} \otimes \theta_H) \Gamma_{\xi, x, y})
      (1 \otimes h) \\
    & = \langle \xi, h \rangle
      \, \Phi_{\qhmu,H,H}^{-1} (\id_{\qhmu} \otimes \theta_H)
      (\underbrace{1 \otimes (x \otimes y)}_{\mathclap{\in \qhmu \otimes (H \otimes H_0)}}) \\
    & = \langle \xi, h \rangle
      \, \Phi_{\qhmu,H,H}^{-1}
      (1 \otimes (x_{(1)} \qhpr_1 \otimes x_{(2)} \qhpr_2 y)) \\
    & = \langle \qhmu, \overline{\qhphi}_1 \rangle \langle \xi, h \rangle
      \, ((1 \otimes \overline{\qhphi}_2 x_{(1)} \qhpr_1)
      \otimes \overline{\qhphi}_3 x_{(2)} \qhpr_2 y).
  \end{align*}
  Now we define $\widetilde{f}$ by \eqref{eq:partial-tr-H-mod-2} with $V = H$ and $W = \qhmu \otimes H$. Then,
  \begin{align*}
    \widetilde{f}(1 \otimes h)
    & = \langle \qhmu, \overline{\qhphi}_1 \rangle \langle \xi, h \rangle
      (1 \otimes \qhg) \qhqr \cdot ((1 \otimes \overline{\qhphi}_2 x_{(1)} \qhpr_1)
      \otimes \overline{\qhphi}_3 x_{(2)} \qhpr_2 y) \\
    & = \langle \qhmu, \qhqr_{1(1)} \overline{\qhphi}_1 \rangle \langle \xi, h \rangle
      \, ((\underbrace{1 \otimes \qhqr_{1(2)} \overline{\qhphi}_2 x_{(1)} \qhpr_1}_{\in \qhmu \otimes H})
      \otimes \underbrace{\qhg \qhqr_2 \overline{\qhphi}_3 x_{(2)} \qhpr_2 y}_{\in H}).
  \end{align*}
  We fix a basis $\{ h_i \}$ of $H$ and let $\{ h^i \}$ be the dual basis of $\{ h_i \}$. We identify $\qhmu \otimes H$ with $H$ as a vector space. By Lemma~\ref{lem:partial-tr-H-mod}, we compute $r(1)$ as follows:
  \begin{align*}
    r(1) & = \ptrace^{\Vect}_{H,H|H}(\widetilde{f})(1)
      \quad \text{(by identifying $\qhmu \otimes H$ with $H$ as a vector space)} \\
    {}^{\eqref{eq:partial-tr-Vec}}
    & = \langle \qhmu, \qhqr_{1(1)} \overline{\qhphi}_1 \rangle \langle \xi, h_i \rangle
      \qhqr_{1(2)} \overline{\qhphi}_2 x_{(1)} \qhpr_1
      \langle h^i, \qhg \qhqr_2 \overline{\qhphi}_3 x_{(2)} \qhpr_2 y \rangle \\
    & = \langle \qhmu, \qhqr_{1(1)} \overline{\qhphi}_1 \rangle \langle \xi, \qhg \qhqr_2 \overline{\qhphi}_3 x_{(2)} \qhpr_2 y \rangle
      \qhqr_{1(2)} \overline{\qhphi}_2 x_{(1)} \qhpr_1.
  \end{align*}
  Hence the right-hand side of \eqref{eq:tw-trace-progen-2} coincides with~\eqref{eq:tw-trace-progen-2-RHS}
\end{proof}

Claims \ref{claim:trace-proof-1} and \ref{claim:trace-proof-2} show that $\mathsf{t}_{\bullet}$ satisfies~\eqref{eq:tw-trace-progen-2} if and only if the equation
\begin{equation}
  \label{eq:tw-trace-progen-3}
  \langle t, h \rangle 1
  = \qhmu(\overline{\qhphi}_1)
  \langle t, (\qhqr_1 \leftharpoonup \qhmu) \overline{\qhphi}_2 h_{(1)} \qhpr_1 \rangle
  \, \qhg \qhqr_2 \overline{\qhphi}_3 h_{(2)} \qhpr_2
\end{equation}
holds for all $h \in H$. Now we prove Theorem~\ref{thm:q-Hopf-modified-trace}.

\begin{proof}[Proof of Theorem~\ref{thm:q-Hopf-modified-trace}]
  By the above argument, it is sufficient to show
  \begin{equation}
    \label{eq:q-Hopf-modified-trace-pf-1}
    \musymint = \{ t \in \HH_0(H, \qhmu) \mid \text{$t$ satisfies~\eqref{eq:tw-trace-progen-3}} \}
  \end{equation}
  to prove this theorem. Suppose that $t$ belongs to $\musymint$. Then, by definition, $t = (\lambda \circ \overline{\qhS}) \leftharpoonup \qhg$ for some left cointegral $\lambda$ on $H$. Let $h$ be an arbitrary element of $H$, and set $h' = \overline{\qhS}(\qhg h)$. Then we have
  \begin{align*}
    \overline{\qhS}(\qhf_1) h'_{(2)} \overline{\qhf}_2 \otimes
    \overline{\qhS}(\qhf_2) h'_{(1)} \overline{\qhf}_1
    & = \qhano{\overline{\qhS}(\qhf_1) \overline{\qhS}(\qhg h)_{(2)}} \overline{\qhf}_2 \otimes
      \qhano{\overline{\qhS}(\qhf_2) \overline{\qhS}(\qhg h)_{(1)}} \overline{\qhf}_1 \\
    {}^{\eqref{eq:q-Hopf-f-1}}
    & = \overline{\qhS}(\qhg_{(1)} h_{(1)}) \overline{\qhS}(\qhf_1) \overline{\qhf}_2 \otimes
      \overline{\qhS}(\qhg_{(2)} h_{(2)}) \overline{\qhS}(\qhf_2) \overline{\qhf}_1 \\
    & = \overline{\qhS}(\qhano{\qhS(\overline{\qhf}_2) \qhf_1 \qhg_{(1)}} h_{(1)}) \otimes
      \overline{\qhS}(\qhano{\qhS(\overline{\qhf}_1) \qhf_2 \qhg_{(2)}} h_{(2)}) \\
    {}^{\eqref{eq:q-Hopf-pivot}}
    & = \overline{\qhS}(\qhg h_{(1)}) \otimes \overline{\qhS}(\qhg h_{(2)}).
  \end{align*}
  We now show that $t$ satisfies~\eqref{eq:tw-trace-progen-3} as follows:
  \begin{align*}
    & \qhmu(\overline{\qhphi}_1)
      \langle t, (\qhqr_1 \leftharpoonup \qhmu) \overline{\qhphi}_2 h_{(1)} \qhpr_1 \rangle
      \, \qhg \qhqr_2 \overline{\qhphi}_3 h_{(2)} \qhpr_2 \\
    & = \qhmu(\overline{\qhphi}_1)
      \langle \lambda, \overline{\qhS}(\qhano{\qhg (\qhqr_1 \leftharpoonup \qhmu) \overline{\qhphi}_2} h_{(1)} \qhpr_1) \rangle
      \, \qhano{\qhg \qhqr_2 \overline{\qhphi}_3} h_{(2)} \qhpr_2 \\
    {}^{\eqref{eq:q-Hopf-pivot}}
    & = \qhmu(\overline{\qhphi}_1)
      \langle \lambda, \overline{\qhS}(\qhS^2((\qhqr_1 \leftharpoonup \qhmu) \overline{\qhphi}_2)
      \qhg h_{(1)} \qhpr_1) \rangle
      \, \qhS^2(\qhqr_2 \overline{\qhphi}_3) \qhg h_{(2)} \qhpr_2 \\
    & = \qhmu(\overline{\qhphi}_1)
      \langle \lambda, \overline{\qhS}(\qhpr_1) \overline{\qhS}(\qhg h_{(1)})
      \qhS((\qhqr_1 \leftharpoonup \qhmu) \overline{\qhphi}_2) \rangle
      \, \qhS^2(\qhqr_2 \overline{\qhphi}_3) \qhS \overline{\qhS}(\qhg h_{(2)}) \qhpr_2 \\
    {}^{(h' := \overline{\qhS}(\qhg h))}
    & = \qhmu(\overline{\qhphi}_1)
      \langle \lambda, \overline{\qhS}(\qhpr_1) \overline{\qhS}(\qhf_1) h'_{(2)} \overline{\qhf}_2
      \qhS((\qhqr_1 \leftharpoonup \qhmu) \overline{\qhphi}_2) \rangle
      \, \qhS^2(\qhqr_2 \overline{\qhphi}_3) \qhS(\overline{\qhS}(\qhf_2) h'_{(1)} \overline{\qhf}_1) \qhpr_2 \\
    & = \qhano{\qhmu(\overline{\qhphi}_1)}
      \qhS(\qhano{\overline{\qhS}(\qhf_2 \qhpr_2)} h'_{(1)}
      \qhano{\overline{\qhf}_1 \qhS(\qhqr_2 \overline{\qhphi}_3)}) \,
      \langle \lambda, \qhano{\overline{\qhS}(\qhf_1 \qhpr_1)} h'_{(2)}
      \qhano{\overline{\qhf}_2 \qhS((\qhqr_1 \leftharpoonup \qhmu) \overline{\qhphi}_2)} \rangle \\
    {}^{\eqref{eq:q-Hopf-def-VL}, \eqref{eq:q-Hopf-def-WL}}
    & = \qhS(\qhVL_1 h'_{(1)} \qhWL_1)
      \langle \lambda, \qhVL_2 h'_{(2)} \qhWL_2 \rangle \\
    {}^{\eqref{eq:left-cointegral-3}}
    & = \qhS(\langle \lambda, h' \rangle 1) = t(h)1.
  \end{align*}
  Hence ``$\subset$'' of \eqref{eq:q-Hopf-modified-trace-pf-1} is proved. We shall prove the converse inclusion. Let $t \in \HH_0(H, \qhmu)$ be an element satisfying~\eqref{eq:tw-trace-progen-3}, and set $\lambda = (t \leftharpoonup \qhg^{-1}) \circ \qhS$. Then, for all $h \in H$, we have
  \begin{align*}
    \langle \lambda, h \rangle 1
    & = \qhano{\langle t, \qhg^{-1} \qhS(h) \rangle} \overline{\qhS}(\qhano{1}) \\
    {}^{\eqref{eq:tw-trace-progen-3}}
    & = \qhmu(\overline{\qhphi}_1) \langle t, (\qhqr_1 \leftharpoonup \qhmu)
      \overline{\qhphi}_2 (\qhano{\qhg^{-1}} \qhS(h))_{(1)} \qhpr_1 \rangle
      \, \overline{\qhS}(\qhg \qhqr_2 \overline{\qhphi}_3 (\qhano{\qhg^{-1}} \qhS(h))_{(2)} \qhpr_2) \\
    {}^{\eqref{eq:q-Hopf-pivot}}
    & = \qhmu(\overline{\qhphi}_1) \langle t,
      (\qhqr_1 \leftharpoonup \qhmu) \overline{\qhphi}_2 \qhg^{-1}
      \qhS(\overline{\qhf}{}_2) \qhano{\qhf_1 \qhS(h)_{(1)}}
      \qhpr_1 \rangle
      \, \overline{\qhS}(\qhg \qhqr_2 \overline{\qhphi}_3 \qhg^{-1}
      \qhS(\overline{\qhf}{}_1) \qhano{\qhf_2 \qhS(h)_{(2)}} \qhpr_2) \\
    {}^{\eqref{eq:q-Hopf-f-1}}
    & = \qhmu(\overline{\qhphi}_1) \langle t,
      \qhano{(\qhqr_1 \leftharpoonup \qhmu) \overline{\qhphi}_2 \qhg^{-1}}
      \qhS(\overline{\qhf}{}_2) \qhS(h_{(2)}) \qhf_1
      \qhpr_1 \rangle
      \, \overline{\qhS}(\qhano{\qhg \qhqr_2 \overline{\qhphi}_3 \qhg^{-1}}
      \qhS(\overline{\qhf}{}_1) \qhS(h_{(1)}) \qhf_2 \qhpr_2) \\
    {}^{\eqref{eq:q-Hopf-pivot}}
    & = \qhmu(\overline{\qhphi}_1) \langle t, \qhg^{-1} \qhS^2(\qhqr_1 \leftharpoonup \qhmu)
      \qhS^2(\overline{\qhphi}_2) \qhS(\overline{\qhf}{}_2) \qhS(h_{(2)}) \qhf_1 \qhpr_1 \rangle
      \, \overline{\qhS}(\qhS^2(\qhqr_2 \overline{\qhphi}_3)
      \qhS(\overline{\qhf}{}_2) \qhS(h_{(1)}) \qhf_2 \qhpr_2) \\
    & = \qhano{\qhmu(\overline{\qhphi}_1)} \langle \lambda,
      \qhano{\overline{\qhS}(\qhf_1 \qhpr_1)} h_{(2)}
      \qhano{\overline{\qhf}{}_2 \qhS(\overline{\qhphi}_2) \qhS(\qhqr_1 \leftharpoonup \qhmu)}
      \rangle \, \qhano{\overline{\qhS}(\qhf_2 \qhpr_2)} h_{(1)}
      \qhano{\qhS(\qhqr_2 \overline{\qhphi}_3)} \\
    {}^{\eqref{eq:q-Hopf-def-VL}, \eqref{eq:q-Hopf-def-WL}}
    & = \qhVL_1 h_{(1)} \qhWL_1
      \langle \lambda, \qhVL_2 h_{(2)} \qhWL_2 \rangle.
  \end{align*}
  Thus, by Theorem~\ref{thm:left-coint-characterization}, the linear form $\lambda$ is a left cointegral on $H$. Since $t = (\lambda \circ \overline{\qhS}) \leftharpoonup \qhg$, the linear form $t$ belongs to $\musymint$. The proof is done.
\end{proof}

\subsection{Remarks}

\subsubsection{Twisted left module traces for quasi-Hopf algebras}

It would be worth to include a left-right switched version of Theorem~\ref{thm:q-Hopf-modified-trace}. Let $\mathcal{C}$ be a $k$-linear pivotal monoidal category with pivotal structure $\mathfrak{j}$.
We note that a left dual object of $X \in \mathcal{C}^{\rev}$ is just a right dual object of $X$ in $\mathcal{C}$ and $\mathcal{C}^{\rev}$ is a pivotal monoidal category by
\begin{equation*}
  \mathfrak{j}_X^{\rev}: X \xrightarrow{\quad \cong \quad} {}^{\vee\vee} \! X^{\vee\vee} \xrightarrow{\quad {}^{\vee\vee}(\mathfrak{j}_{X}^{-1}) \quad} {}^{\vee\vee} \! X
  \quad (X \in \mathcal{C}).
\end{equation*}
In this section, we only considered the {\em right} partial pivotal trace. The {\em left} partial pivotal trace is defined by
\begin{gather*}
  \widetilde{\ptrace}{}^{\mathcal{C}}_{X|V,W}
  := \left(
    \Hom_{\mathcal{C}}(X \otimes V, X \otimes W)
    \xrightarrow{\quad \displaystyle \ptrace^{\mathcal{C}^{\rev}}_{V,W|X} \quad}
    \Hom_{\mathcal{C}}(V, W)
  \right)
\end{gather*}
for $V, W, X \in \mathcal{C}$. Let $\mathcal{I}$ be a tensor ideal of $\mathcal{C}$, and let $D$ be an object of $\mathcal{C}$. A $D$-twisted module trace on $\mathcal{I}$ of Definition \ref{def:D-tw-trace} should be called a $D$-twisted {\em right} module trace on $\mathcal{I}$. We define a {\em $D$-twisted left module trace on $\mathcal{I}$} to be a $\Sigma$-twisted trace $\mathsf{t}_{\bullet}$ on $\mathcal{I}$, where $\Sigma = (-) \otimes D$, such that the equation
\begin{equation*}
  \mathsf{t}_{X \otimes P}(f) = \mathsf{t}_{P}\Big(
  \widetilde{\ptrace}{}_{X | P, P \otimes D}^{\mathcal{C}}(\Phi_{X, P, D} \circ f) \Big)
\end{equation*}
holds for all objects $P \in \mathcal{I}$, $X \in \mathcal{C}$ and morphisms $f: X \otimes P \to (X \otimes P) \otimes D$ in $\mathcal{I}$.

Now let $H$ be a finite-dimensional pivotal quasi-Hopf algebra with pivotal element $\qhg$. We define the algebra automorphism $\widetilde{\chi}: H \to H$ by $\widetilde{\chi}(h) = \qhmu \rightharpoonup h$ for $h \in H$ and set
\begin{equation*}
\widetilde{\HH}_0(H, \qhmu) = \HH_0(H, \widetilde{\chi})
\quad \text{and} \quad
\widetilde{\HH}_0({}_H\Proj_{\fd}, \qhmu) = \HH_0({}_H \Proj_{\fd}, \widetilde{\chi}{}^*)
\end{equation*}
with notation of Subsection \ref{subsec:tw-mod-tr}. We denote by $\widetilde{\HH}{}^{\mathsf{mod}}_0({}_H\Proj_{\fd}, \qhmu)$ the class of $\qhmu$-twisted left module traces on the tensor ideal ${}_H\Proj_{\fd}$.

\begin{theorem}
  \label{thm:q-Hopf-modified-trace-left}
  The isomorphism $\widetilde{\HH}_0({}_H\Proj_{\fd}, \qhmu) \cong \widetilde{\HH}_0(H, \qhmu)$ given in Subsection \ref{subsec:tw-mod-tr} restricts to the isomorphism
  \begin{equation*}
    \widetilde{\HH}{}^{\mathsf{mod}}_0({}_H\Proj_{\fd}, \qhmu)
     \cong \{ (\lambda \circ \qhS) \leftharpoonup \qhg^{-1} \mid \text{$\lambda$ is a right cointegral on $H$} \}.
  \end{equation*}
\end{theorem}
\begin{proof}
  $\widetilde{\HH}{}^{\mathsf{mod}}_0({}_H\Proj_{\fd}, \qhmu)$ can be identified with the set of $\qhmu$-twisted {\em right} module traces on the tensor ideal of projective objects of the $k$-linear pivotal monoidal category $({}_H \Mod_{\fd})^{\rev}$. The quasi-Hopf algebra $H^{\cop}$ is also pivotal with pivotal element $\qhg^{-1}$. Indeed, we have  $\qhg^{-1} h \qhg = \overline{\qhS}{}^2(h)$ for all $h \in H$ and
  \begin{align*}
    \Delta^{\cop}(\qhg^{-1})
    \ {}^{\eqref{eq:q-Hopf-pivot}}
    & = \qhg^{-1} \qhS(\overline{\qhf}_1) \qhf_2 \otimes \qhg^{-1} \qhS(\overline{\qhf}_2) \qhf_1 \\
    {}^{\eqref{eq:q-Hopf-pivot}}
    & = \overline{\qhS}(\overline{\qhf}_1) \overline{\qhS}{}^2(\qhf_2) \qhg^{-1}
    \otimes \overline{\qhS}(\overline{\qhf}_2) \overline{\qhS}{}^2(\qhf_1) \qhg^{-1} \\
    {}^{\eqref{eq:q-Hopf-f-op-cop}}
    & = \overline{\mathfrak{f}}_1 \overline{\qhS}(\mathfrak{f}_2) \qhg^{-1}
    \otimes \overline{\mathfrak{f}}_2 \overline{\qhS}(\mathfrak{f}_1) \qhg^{-1},
  \end{align*}
  where $\mathfrak{f} = \qhf_{\cop}$. The $k$-linear pivotal monoidal category $({}_H \Mod_{\fd})^{\rev}$ can be identified with ${}_{H^{\cop}}\Mod_{\fd}$. A left cointegral on $H^{\cop}$ is nothing but a right cointegral on $H$. The proof is done by applying Theorem \ref{thm:q-Hopf-modified-trace} to $H^{\cop}$.
\end{proof}

\subsubsection{Unimodular case}

Let $H$ be a finite-dimensional pivotal quasi-Hopf algebra.
In \cite{2018arXiv181210445B}, modified traces on ${}_H \Proj_{\fd}$ are studied in the case where $H$ is unimodular. A modified trace on ${}_H \Proj_{\fd}$ is, in our terminology, a $\Sigma$-twisted module trace on ${}_H \Proj_{\fd}$ with $\Sigma$ the identity functor.

We suppose that $H$ is unimodular. Then, by \cite[Corollary 4.4]{MR2862216}, a linear map $\lambda: H \to k$ is a left cointegral on $H$ if and only if $\lambda \circ \overline{\qhS}$ is a right cointegral on $H$. Thus we have
\begin{equation*}
\musymint = \{ \lambda \leftharpoonup \qhg
\mid \text{$\lambda$ is a right cointegral on $H$} \},
\end{equation*}
which coincides with the space of {\em symmetrized right cointegrals} on $H$ in the sense of \cite{2018arXiv181210445B}. Hence Theorem \ref{thm:q-Hopf-modified-trace} reduces to the description of the space of modified traces given in \cite[Theorem 1.1]{2018arXiv181210445B}.

\subsubsection{Two-sided module traces}

Let $H$ be a finite-dimensional unimodular pivotal quasi-Hopf algebra with pivotal element $\qhg$. Since both $\qhmu \otimes (-)$ and $(-) \otimes \qhmu$ are the identity functors on ${}_H \Mod_{\fd}$ in this case, we may ask whether a left/right $\qhmu$-twisted module trace on ${}_H \Proj_{\fd}$ is a right/left $\qhmu$-twisted module trace on ${}_H \Proj_{\fd}$. By a {\em two-sided module trace} on ${}_H \Proj_{\fd}$, we mean a right $\qhmu$-twisted module trace on ${}_H \Proj_{\fd}$ as well as a left $\qhmu$-twisted module trace on ${}_H \Proj_{\fd}$.

We fix a non-zero right cointegral $\lambda$ on $H$. According to \cite{2018arXiv180100321B}, we say that $H$ is {\em unibalanced} if $\lambda \leftharpoonup \qhg = \lambda^{\qhL} \leftharpoonup \qhg^{-1}$ for some left cointegral $\lambda^{\qhL}$ on $H$. As a generalization of a result of \cite{2018arXiv180100321B} for finite-dimensional unimodular pivotal Hopf algebras, we prove:

\begin{theorem}[{\it cf.} {\cite[Theorem 1]{2018arXiv180100321B}}]
  A non-zero two-sided module trace on ${}_H \Proj_{\fd}$ exists if and only if $H$ is unibalanced.
\end{theorem}   
\begin{proof}
  Let $\lambda$ be a non-zero right cointegral on $H$.
  By Theorems~\ref{thm:q-Hopf-modified-trace} and \ref{thm:q-Hopf-modified-trace-left} and the fact that the spaces $\cointl$ and $\cointr$ of cointegrals are one-dimensional, a non-zero two-sided module trace on ${}_H \Proj_{\fd}$ exists if and only if
  \begin{equation}
    \label{eq:q-Hopf-two-sided-mod-tr}
    \lambda \qhS \leftharpoonup \qhg^{-1} \in \{ \lambda^{\qhL} \overline{\qhS} \leftharpoonup \qhg \mid \lambda^{\mathsf{L}} \in \cointl \}.
  \end{equation}
  Let $\lambda^{\qhL}$ and $\lambda^{\qhR}$ be a left and a right cointegral on $H$, respectively.
  By \cite[Corollary 4.4]{MR2862216} (see also Remark~\ref{rem:left-right-coint} (b)), $\lambda^{\qhL} \overline{\qhS}$ and $\lambda^{\qhR} \qhS$ are a right and a left cointegral on $H$, respectively.
  By this observation, it is easy to check that \eqref{eq:q-Hopf-two-sided-mod-tr} is equivalent to that $H$ is unibalanced. The proof is done.
\end{proof}

Suppose that $H$ is a finite-dimensional unimodular pivotal Hopf algebra with pivotal element $\qhg$. There is a unique grouplike element $g_H \in H$ such that the equation $h_{(1)} \langle \lambda, h_{(2)} \rangle = \lambda(h) g_H$ holds for all $h \in H$ and all right cointegrals $\lambda$ on $H$. It has been known that $H$ is unibalanced if and only if $g_H = \qhg^2$ \cite[Lemma 4.8]{2018arXiv180100321B}. We expect that a similar result holds in the general case where $H$ is a unimodular pivotal quasi-Hopf algebra.


\begin{thebibliography}{DMNO13}

\bibitem[BBG18]{2018arXiv180100321B}
A.~{Beliakova}, C.~{Blanchet}, and A.~M. {Gainutdinov}.
\newblock {Modified trace is a symmetrised integral}.
\newblock {\tt arXiv:1801.00321}.

\bibitem[BC03]{MR1995128}
D.~Bulacu and S.~Caenepeel.
\newblock Integrals for (dual) quasi-{H}opf algebras. {A}pplications.
\newblock {\em J. Algebra}, 266(2):552--583, 2003.

\bibitem[BC12]{MR2862216}
D.~Bulacu and S.~Caenepeel.
\newblock On integrals and cointegrals for quasi-{H}opf algebras.
\newblock {\em J. Algebra}, 351:390--425, 2012.

\bibitem[BCP05]{MR2106925}
D.~Bulacu, S.~Caenepeel, and Florin Panaite.
\newblock More properties of {Y}etter-{D}rinfeld modules over quasi-{H}opf
  algebras.
\newblock In {\em Hopf algebras in noncommutative geometry and physics}, volume
  239 of {\em Lecture Notes in Pure and Appl. Math.}, pages 89--112. Dekker,
  New York, 2005.

\bibitem[BCP06]{MR2194347}
D.~Bulacu, S.~Caenepeel, and F.~Panaite.
\newblock Yetter-{D}rinfeld categories for quasi-{H}opf algebras.
\newblock {\em Comm. Algebra}, 34(1):1--35, 2006.

\bibitem[BCPO19]{MR3929714}
D.~Bulacu, S.~Caenepeel, F.~Panaite and F. Oystaeyen.
\newblock Quasi-Hopf algebras. A categorical approach.
\newblock Cambridge University Press, Cambridge, 2019.

\bibitem[BCT09]{MR2501184}
D.~Bulacu, S.~Caenepeel, and B.~Torrecillas.
\newblock Involutory quasi-{H}opf algebras.
\newblock {\em Algebr. Represent. Theory}, 12(2-5):257--285, 2009.

\bibitem[BGR20]{2018arXiv181210445B}
J.~Berger, A.~M. Gainutdinov and I.~Runkel.
\newblock Modified traces for quasi-Hopf algebras.
\newblock {\em J. Algebra} (548), pp. 96--119, 2020.

\bibitem[BKW16]{2016arXiv160503523B}
  A.~Beliakova, K.~K.~Putyra and S.~M.~Wehrli.
  Quantum link homology via trace functor I. {\em Invent. math.}, 215, 383--492, 2019.

\bibitem[BLV11]{MR2793022}
A.~Brugui{\`e}res, S.~Lack, and A.~Virelizier.
\newblock Hopf monads on monoidal categories.
\newblock {\em Adv. Math.}, 227(2):745--800, 2011.

\bibitem[BT04]{MR2086073}
D.~Bulacu and B.~Torrecillas.
\newblock Factorizable quasi-{H}opf algebras---applications.
\newblock {\em J. Pure Appl. Algebra}, 194(1-2):39--84, 2004.

\bibitem[BT08]{MR2363502}
D.~Bulacu and B.~Torrecillas.
\newblock The representation-theoretic rank of the doubles of quasi-quantum
  groups
\newblock {\em J. Pure Appl. Algebra}, 212(4):919--940, 2008.

\bibitem[BT20]{MR4009567}
D.~Bulacu and B.~Torrecillas.
On sovereign, balanced and ribbon quasi-Hopf algebras.
J. Pure Appl. Algebra 224 (2020), no. 3, 1064--1091.

\bibitem[CGPT20]{2018arXiv180907991C}
  F.~{Costantino}, N.~{Geer}, B.~{Patureau-Mirand}, and V.~{Turaev}.
  Kuperberg and Turaev--Viro invariants in unimodular categories.
  {\it Pacific J. Math.}, 306 (2), 421--450, 2020.

\bibitem[CGR20]{2017arXiv171207260C}
  T.~{Creutzig}, A.~M. {Gainutdinov}, and I.~{Runkel}.
  \newblock {A quasi-Hopf algebra for the triplet vertex operator algebra}.
  {\em Commun. Contemp. Math.}, 22 (3), 1950024, 71 pp, 2020.

\bibitem[GKP11]{MR2803849}
N. Geer, J. Kujawa, B. Patureau-Mirand, Generalized trace and modified
dimension functions on ribbon categories, Selecta Math.17 (2011).

\bibitem[GLO18]{2018arXiv180902116G}
  A. M. Gainutdinov, S. Lentner and T. Ohrmann.
  Modularization of small quantum groups. {\tt arXiv:1809.02116}.

\bibitem[GPV13]{MR2998839}
  N. Geer, B. Patureau-Mirand, A. Virelizier, Traces on ideals in pivotal
  categories, Quantum Topology 4 (2013) 91-124.

\bibitem[DMNO13]{MR3039775}
  A.~Davydov, M.~M{\"u}ger, D.~Nikshych, and V.~Ostrik.
  \newblock The {W}itt group of non-degenerate braided fusion categories.
  \newblock {\em J. Reine Angew. Math.}, 677:135--177, 2013.

\bibitem[Dri89]{MR1047964}
V.~G. Drinfeld.
\newblock Quasi-{H}opf algebras.
\newblock {\em Algebra i Analiz}, 1(6):114--148, 1989.

\bibitem[EGNO15]{MR3242743}
P.~Etingof, S.~Gelaki, D.~Nikshych, and V.~Ostrik.
\newblock {\em Tensor categories}, volume 205 of {\em Mathematical Surveys and
  Monographs}.
\newblock American Mathematical Society, Providence, RI, 2015.

\bibitem[FG18]{2018arXiv180901122F}
A.~F. {Fontalvo Orozco} and A.~M. {Gainutdinov}.
\newblock {Module traces and Hopf group-coalgebras}.
\newblock {\tt arXiv:1809.01122}

\bibitem[FGR17]{2017arXiv170608164F}
V.~{Farsad}, A.~M. {Gainutdinov}, and I.~{Runkel}.
\newblock {The symplectic fermion ribbon quasi-Hopf algebra and the
  SL(2,Z)-action on its centre}.
\newblock {\tt arXiv:1706.08164}

\bibitem[GKP18]{2018arXiv180900499G}
N.~{Geer}, J.~{Kujawa}, and B.~{Patureau-Mirand}.
\newblock {M-traces in (non-unimodular) pivotal categories}.
\newblock {\tt arXiv:1809.00499}.

\bibitem[HN99a]{MR1669685}
F.~{Hausser} and F.~{Nill}.
Doubles of quasi-quantum groups.
Comm. Math. Phys. 199 (1999), no. 3, 547--589.

\bibitem[HN99b]{MR1696105}
F.~{Hausser} and F.~{Nill}.
Diagonal crossed products by duals of quasi-quantum groups.
Rev. Math. Phys. 11 (1999), no. 5, 553--629.

\bibitem[HN99c]{1999math4164H}
F.~{Hausser} and F.~{Nill}.
\newblock {Integral Theory for Quasi-Hopf Algebras}.
\newblock 1999. {\tt arXiv:math/9904164v2}

\bibitem[IM14]{MR3265394}
A.~Ishii and A.~Masuoka.
\newblock Handlebody-knot invariants derived from unimodular {H}opf algebras.
\newblock {\em J. Knot Theory Ramifications}, 23(7):1460001, 24, 2014.

\bibitem[Kas95]{MR1321145}
C.~Kassel.
\newblock {\em Quantum groups}, volume 155 of {\em Graduate Texts in
  Mathematics}.
\newblock Springer-Verlag, New York, 1995.

\bibitem[Maj98]{MR1631648}
S.~Majid.
\newblock Quantum double for quasi-{H}opf algebras.
\newblock {\em Lett. Math. Phys.}, 45(1):1--9, 1998.

\bibitem[ML98]{MR1712872}
S.~Mac~Lane.
\newblock {\em Categories for the working mathematician}, volume~5 of {\em
  Graduate Texts in Mathematics}.
\newblock Springer-Verlag, New York, second edition, 1998.

\bibitem[{Neg}18]{2018arXiv181202277N}
C.~{Negron}.
\newblock {Log-modular quantum groups at even roots of unity and the quantum
  Frobenius I}.
\newblock {\tt arXiv:1812.02277}.

\bibitem[Sch02]{MR1897403}
P.~Schauenburg.
\newblock Hopf modules and the double of a quasi-{H}opf algebra.
\newblock {\em Trans. Amer. Math. Soc.}, 354(8):3349--3378, 2002.

\bibitem[Sch04]{MR2037710}
P.~Schauenburg.
\newblock Two characterizations of finite quasi-{H}opf algebras.
\newblock {\em J. Algebra}, 273(2):538--550, 2004.

\bibitem[Shi17a]{MR3631720}
K.~Shimizu.
\newblock The monoidal center and the character algebra.
\newblock {\em J. Pure Appl. Algebra}, 221(9):2338--2371, 2017.

\bibitem[Shi17b]{MR3632104}
K.~Shimizu.
\newblock On unimodular finite tensor categories.
\newblock {\em Int. Math. Res. Not. IMRN}, (1):277--322, 2017.

\bibitem[Shi19]{MR3921367}
K.~{Shimizu}.
\newblock {Integrals for finite tensor categories}.
\newblock {\em Algebras and Representation Theory}, 22. 2019.

\bibitem[Swe69]{MR0242840}
M. E. Sweedler.
\newblock Integrals for {H}opf algebras.
\newblock {\em Ann. of Math. (2)}, 89:323--335, 1969.

\bibitem[Tur94]{MR1292673}
V.~G. Turaev.
\newblock {\em Quantum invariants of knots and 3-manifolds}, volume~18 of {\em
  de Gruyter Studies in Mathematics}.
\newblock Walter de Gruyter \& Co., Berlin, 1994.
\end{thebibliography}
\def\cprime{$'$}

\end{document}